\documentclass[11pt, reqno]{amsart}
\usepackage{amssymb,latexsym,amsmath,amsfonts,mathdots,enumitem}
\usepackage{latexsym}
\usepackage[mathscr]{eucal}
\usepackage{colortbl,xcolor}
\usepackage{lmodern}
\usepackage{sansmathaccent}
\usepackage[latin1]{inputenc}
\usepackage{tikz}
\usetikzlibrary{shapes,arrows}
\usetikzlibrary{matrix,calc,shapes,arrows,positioning}
\pdfmapfile{+sansmathaccent.map}

\numberwithin{equation}{section}
\theoremstyle{plain}

\newtheorem{thm}{Theorem}[section]

\newtheorem{cor}[thm]{Corollary}
\newtheorem{lem}[thm]{Lemma}
\newtheorem{prop}[thm]{Proposition}
\theoremstyle{definition}
\newtheorem{defn}[thm]{Definition}
\newtheorem{rem}[thm]{Remark}

\numberwithin{equation}{section}
\def\A{{\mathcal A}}

\def\D{{\mathcal D}}

\def\I{\mathcal{I}}

\def\Re{\operatorname{Re}}

\def\beq{\begin{eqnarray}}
\def\eeq{\end{eqnarray}}
\def\beqa{\begin{eqnarray*}}
\def\eeqa{\end{eqnarray*}}

\def\Ran{\operatorname{Ran}}

\def\beqn{\begin{equation}}
\def\eeqn{\end{equation}}

\def\mg#1{}

\def\Ran{\operatorname{Ran}}

\renewcommand{\epsilon}{\varepsilon}
\renewcommand{\phi}{\varphi}

\begin{document}
\title{On $\Gamma_n$-contractions and their Conditional Dilations }
\author{Avijit Pal}
\address[A. Pal]{Department of Mathematics , Indian Institute of Technology, Bhilai }
\email{A. Pal:avijit@iitbhilai.ac.in}

\subjclass[2010]{32A60, 32C15, 47A13, 47A15, 47A20, 47A25,
47A45.}

\keywords{Symmetrized polydisc, Spectral set, Complete spectral set, Wold decomposition, Pure isometry, Conditional dilation, Functional model}
\thanks{The work of A. Pal was supported through the NBHM Post-doctoral Fellowship. }

\begin{abstract}
We prove some estimates for elementary symmetric polynomials on $\mathbb D^n.$ We show  that these estimates are sharp which allow us to study the properties of closed symmetrized polydisc $\Gamma_n.$ Furthermore, we show the existence and uniqueness of solutions to the operator equations $$S_i-S_{n-i}^*S_n=D_{S_n}X_iD_{S_n}~~{\rm{and}}~~S_{n-i}-S_{i}^*S_n=D_{S_n}X_{n-i}D_{S_n},$$
where $X_i,X_{n-i}\in \mathcal B(\mathcal D_{S_n}), ~{\rm{for ~all~}} i=1,\ldots,(n-1),$ with numerical radius not greater than $1,$ for a $\Gamma_n$-contraction $(S_1,\ldots, S_n).$ We construct a conditional dilation of various classes of $\Gamma_n$-contractions. Various properties of a $\Gamma_n$-contraction and its explicit dilation allow us to construct a concrete functional model for a $\Gamma_n$-contraction. We describe the structure and additional characterization of $\Gamma_n$-unitaries  and $\Gamma_n$-isometries in  detail. 
\end{abstract}
\maketitle
\vskip-.5cm

\section{Introduction}
For $n\geq 2,$ let $\bold s:\mathbb C^n\rightarrow \mathbb C^n$ be the symmetrization map  given by the formula $$\bold s(\bold z)=(s_1(\bold z),\ldots,s_n(\bold z)),~~ \bold z=(z_1,\ldots,z_n)\in\mathbb C^n,$$ where $s_i(\bold z)=\sum_{1\leq k_1< \ldots< k_i\leq n}z_{k_1}\ldots z_{k_i}$ and $s_0=1.$ The image $\Gamma_n:=\bold s(\bar{\mathbb D}^n)$ under the map $\bold s$ of the unit polydisc is known as closed symmetrized polydisc. The map $\bold s$ is a proper holomorphic map \cite{Rudin}. The set $\Gamma_n$ is polynomially convex but not convex \cite{stout}. The open symmetrized polydisc is defined to be the set $\mathbb G_n:=\bold s(\mathbb D^n).$ Also, the distinguished boundary $b\Gamma_n$ of symmetrized polydisc is known to be $\bold s(\mathbb T^n),$ the image of the $n$-torus $\mathbb T^n$ under the map $\bold s
$ \cite{Zwonek}.

A commuting $n$-tuple of bounded operators $(S_1,\ldots,S_n)$ on a Hilbert space $\mathcal H$ having $\Gamma_n$ as a spectral set will be called a $\Gamma_n$-contraction. 
Let $f=\left(\!(f_{ij})\!\right)$  be a matrix valued polynomial defined on $\Gamma_n.$ 
Let
$\|f\|=\sup\{\|\left(\!(f_{ij}(z))\!\right)\|_{\rm op}: z \in
\Gamma_n\}.$ 
$\Gamma_n$ is said
to be complete spectral set for $(S_1,\ldots,S_n)$ or $(S_1,\ldots,S_n)$ to be a complete $\Gamma_n$-contraction if $\|f(S_1,\ldots,S_n) \| \leq \|f\|_{\infty,
\Gamma_n}$ for all $f\in \mathcal A\otimes \mathcal
M_k(\mathbb C), k\geq 1.$

Let $(S_1,\ldots,S_n)$ be a commuting $n$-tuple of bounded operators defined on a Hilbert space $\mathcal H$ whose joint spectrum lies in $\Gamma_n,$ we say that   $(S_1,\ldots,S_n)$  admits a $b \Gamma_n$ normal dilation if there is a Hilbert space $\mathcal K$ containing $\mathcal H$ as a subspace and a normal operator ${\bf{N}}=(N_1,\ldots,N_n)$ whose joint spectrum lies in $b\Gamma_n$ such that $S_i^j=P_{\mathcal H}N_i^j\mid_{\mathcal H},1\leq i\leq n,j\geq0,$ where $P_{\mathcal H}$ is the orthogonal projection of $\mathcal K$ onto $\mathcal H.$ 
In  \cite{A,AW} it was shown that the existence of a
$b\Gamma_n$ normal dilation of a commuting tuple $(S_1,\ldots,S_n)$ of operators whose joint spectrum lies in $\Gamma_n$ is equivalent to $ (S_1,\ldots,S_n)$ being a complete $\Gamma_n$-contraction.
We  recall the $\Gamma_n$-unitary, $\Gamma_n$-isometry and pure $\Gamma_n$-isometry from \cite{SS}.

\small{\begin{defn}
Let $(S_1,\ldots,S_n)$ be a commuting $n$-tuple of operators  on a Hilbert space $\mathcal H.$ We say that $(S_1,\ldots,S_n)$ is
\begin{enumerate}
\item a $\Gamma_n$-unitary if $S_1,\ldots,S_n$ are normal operators and the joint spectrum $\sigma(S_1,\ldots,S_n)$ of $(S_1,\ldots,S_n)$ is contained in the distinguished boundary of $\Gamma_n.$

\item a $\Gamma_n$-isometry if there exists a Hilbert space $\mathcal K\supseteq\mathcal H$ and a $\Gamma_n$-unitary  $(\tilde{S}_1,\ldots,\tilde{S}_n)$ on $\mathcal K$ such that $\mathcal H$ is a common invariant subspace for $\tilde{S}_1,\ldots,\tilde{S}_n$ and that $S_i=\tilde{S}_i\mid_{\mathcal H}$ for $i=1,\ldots,n.$ 

\item a $\Gamma_n$-co-isometry if $(S_1^*,\ldots,S_n^*)$ is a $\Gamma_n$-isometry.

\item a pure $\Gamma_n$-isometry  if $(S_1,\ldots,S_n)$ is a $\Gamma_n$-isometry and $S_n$ is also a pure isometry.
\end{enumerate}
\end{defn}}

For a commuting $n$-tuple of operators $(S_1,\ldots,S_n)$ with $\|S_n\|\leq 1,$ we consider the following operator  equations
\small{$$S_i-S_{n-i}^*S_n=D_{S_n}E_iD_{S_n}~~{\rm{and}}~~ S_{n-i}-S_i^*S_n=D_{S_n}E_{n-i}D_{S_n}~{\rm~for} ~i=1,\ldots,(n-1),$$}
where $D_{S_n}=(I-S_n^*S_n)^{\frac{1}{2}}$ and $\mathcal D_{S_n}=\overline{\Ran} D_{S_n}.$ This equations are called the fundamental equations for $(S_1,\ldots,S_n).$ The operators $E_i$ for $i=1,\ldots,(n-1)$ are called the fundamental operators.  These operators $E_i^{,}{\rm{s}}$ play an important role in constructing the conditional dilation of $\Gamma_n$.

The only  dilation theorems, in the
multi-variable context are proposed by Agler and Young
\cite[Theorem 1.1]{young} and Ando \cite[Chapter 5, Theorem 5.5]{paulsen}. 
Under some additional hypotheses, we prove that $\Gamma_n$ is a spectral set for a commuting tuple of operators $(S_1,\ldots,S_n)$  if and only if $\Gamma_n$ is a complete spectral set.

In Section $2,$  we prove some estimates for elementary symmetric polynomial on $\mathbb D^n$ which are important in their own right. These estimates play a pivotal role for studying the characterizations of $\Gamma_n.$ We also show that these estimates are sharp.
 
In Section $3,$  the operator pencils  $\Phi_{1}^{(i)}$ and $\Phi_{2}^{(i)}$  are defined for $i=1,\ldots,(n-1).$ These operator pencils play a significant role in characterizing the domain $\Gamma_n.$ In subsequent section, we find that these operator pencils are also important in  studying $\Gamma_n$-contractions.

Section $4$ is devoted to $\Gamma_n$-contraction and their fundamental operators. One can easily verify that if $(S_1,\ldots,S_n)$ is a $\Gamma_n$-contraction, then $\|S_i\|\leq \binom{n-1}{i}+\binom{n-1}{n-i}$ for $i=1,\ldots,(n-1)$ and $\|S_n\|\leq 1.$ In Proposition \ref{k(i)}, we prove that if $(S_1,\ldots,S_n)$ is a $\Gamma_n$-contraction, then \small{$$\Phi^{(i)}_{1}(\alpha^{i}S_i,\alpha^{n-i}S_{n-i},\alpha^{n}S_n)\geq 0 ~~{\rm{and}}~~ \Phi^{(i)}_{2}(\alpha^{i}S_{i},\alpha^{n-i}S_{n-i},\alpha^{n}S_n)\geq 0$$} for all $\alpha\in \overline{\mathbb D}$ with $ k(i)=\binom{n-1}{i}+\binom{n-1}{n-i}$ for $ i=1,\ldots,(n-1).$ However, it is not clear whether the converse holds. Using this result, we show the existence and uniqueness theorem for $(n-1)$-tuple of fundamental operators $E_i$ for $i=1,\ldots,(n-1).$

In Section $5,$ we prove the various properties of $\Gamma_n$-unitary and $\Gamma_n$-isometry. 
We show that if $(S_1,\ldots, S_n)$ is a $\Gamma_n$-contraction and $S_n$ is a unitary, then $(S_1,\ldots, S_n)$ is a $\Gamma_n$-unitary (see Theorem \ref{Gamma_n unitary}). We prove that if $(S_1,\ldots,S_n)$ is a $\Gamma_n$-isometry, then \small{$$\Phi^{(i)}_{1}(\alpha^{i}S_i,\alpha^{n-i}S_{n-i},\alpha^{n}S_n)= 0 ~~{\rm{and}}~~ \Phi^{(i)}_{2}(\alpha^{i}S_{i},\alpha^{n-i}S_{n-i},\alpha^{n}S_n)= 0~{\rm{for}}~  i=1,\ldots,n-1$$} for all $\alpha\in \overline{\mathbb D}$ with $ k(i)= \binom{n-1}{i}+\binom{n-1}{n-i}$  and vice-versa. Similar to the case of $\Gamma_n$-unitary, we further show that if $(S_1,\ldots, S_n)$ is a $\Gamma_n$-contraction and $S_n$ is a isometry, then $(S_1,\ldots, S_n)$ is also a $\Gamma_n$-isometry (see Theorem \ref{Gamma_n isometry}).

In Section $6,$ we  find  the necessary conditions for the existence of rational dilation in terms of its fundamental operators ( Proposition \ref{pro}). 
%
%

In Section $7,$ we prove that along with the necessary conditions, if we assume some conditions on the fundamental operators of $(S_1^*,\ldots,S_n^*),$ where $(S_1,\ldots,S_n)$ is a $\Gamma_n$-contraction, then $(S_1,\ldots,S_n)$ possesses a $\Gamma_{n}$-unitary dilation. We produce explicit $\Gamma_n$-isometric dilations of these $\Gamma_n$-contractions.


In Section $8,$ we construct an explicit functional model for a class of $\Gamma_n$-contractions assuming similar conditions on the fundamental operators of  $(S_1,\ldots,S_n)$ and $(S_1^*,\ldots,S_n^*)$ respectively, where $(S_1,\ldots,S_n)$ is a $\Gamma_n$-contraction.

We found some discrepancy in S. Pal's proof \cite[Page $6$, Proposition 2.5]{pal4}\cite[Page $8$,Proposition 2.6]{pal4},\cite[Page $9$,Theorem 3.1]{pal4}, \cite[Page $10$, Theorem 3.5]{pal3},\cite[Page $13,$ Theorem 3.7]{pal3},\cite[Page $15,$ Theorem 4.2]{pal3},\cite[Page $19,$ Theorem 4.3]{pal3}\cite[Page-$13$, Proposition 2.9]{pal5},\cite[Page $16,$ Theorem 3.2]{pal5}, where the author used the following assertion:  if a point $s_i\in \Gamma_n$ then $\|s_i\|\leq n$ for all $i$. We believe that this statement is incorrect. We now give a counter-example that supports our claim. For $n\geq 4,i\neq 1,n$ and $i\neq (n-1),$ we show that there exists a point $s_i\in \Gamma_n$ such that $\|s_i\|>n.$ Choose $\epsilon>0$ is arbitrary small such  that $z_i=(1-\epsilon)$ for all $i=1,\ldots,n.$ Then by definition we have $s_i\in \mathbb G_n$ and $\|s_i\|=(1-\epsilon)\binom{n}{i}.$ Since $n-1>i,$ for arbitrary small $\epsilon>0$  ensure that    
$\|s_i\|>n.$ From above discussion, it follows that for arbitrary small $\epsilon>0,$ there are lots of points $s_i\in \Gamma_n$ such that $\|s_i\|>n.$

\section{A sharp estimate for the elementary symmetric polynomial on $\mathbb D^n$ }
 Only for this section we use the notation $s_{i}^{(m)}(\bold z)$ in place of $s_{i}(\bold z)$ to emphasize the fact that the ambient space is $\mathbb C^m.$ For  $\bold z \in \mathbb D^{m+1}$ and $i^{\prime}=m+2-i,$ let us consider $h^{(m+1)}_{ij}(\bold z)=|s_{i}^{(m+1)}(\bold z)|^2-|s_{m+2+j-i}^{(m+1)}(\bold z)|^2$ and $h^{(m+1)}_{i^{\prime}j}(\bold z)=|s_{m+2-i^{\prime}}^{(m+1)}(\bold z)|^2-|s_{i^{\prime}+j}^{(m+1)}(\bold z)|^2.$ Then we see that $h^{(m+1)}_{i^{\prime}j}(\bold z)=-h^{(m+1)}_{ij}(\bold z).$ Our main goal is to prove the following estimate. 
\begin{thm}\label{main@}
For  $\bold z \in \mathbb D^{m+1}$ and $i^{\prime}=m+2-i.$ Then we have the following estimates
for all $i$ with $ i\in  \{1,\ldots,m
+1\}$ and $j=i-m-2,\ldots,(i-1)$ and $j\neq 2i-m-2$
\begin{eqnarray}\label{s_{m+1}}
\big|h^{(m+1)}_{ij}(\bold z)\big|=\big|h^{(m+1)}_{i^{\prime}j}(\bold z)\big| \leq \left\{ \begin{array}{ll}
         \big|\binom{m+1}{i}^2-\binom{m+1}{m+2+j-i}^2\big|;\vspace{2mm}\\
        \big|\binom{m+1}{m-i}^2-\binom{m+1}{i-1}^2\big| & {\rm{if}}~~ j=-1;
         \vspace{2mm}\\ 1 & {\rm{if}}~~j=-1,i=m+1
        \end{array} \right.
\end{eqnarray}
However,  if $j=2i-m-2,$ then $h^{(m+1)}_{i^{\prime}j}(\bold z)=0=h^{(m+1)}_{ij}(\bold z).$
%
\end{thm}

We prove  Theorem \ref{main@} by means of  several lemmas and propositions. For $ i\in  \{1,\ldots,m\}$ and $0\leq l\leq i,$ let $F^{(m)}_{il},G^{(m)}_{i^{\prime}l}:\{i+l-m-2,\ldots, i-2\}\rightarrow \mathbb R$ be two functions defined by \begin{equation}\label{fm}F^{(m)}_{il}(j)=\binom{m}{i}\binom{m}{i-l}-\binom{m}{m+2+j-i}\binom{m}{m+2+j-i-l}\end{equation} $$\textit{and} $$ \begin{equation}\label{gm}G^{(m)}_{i^{\prime}l}(j)=\binom{m}{i^{\prime}+j}\binom{m}{i^{\prime}+j-l}-\binom{m}{m+2-i^{\prime}}\binom{m}{m+2-i^{\prime}-l}=-F^{(m)}_{il},\end{equation} where  $i^{\prime} =m+2-i.$ 
For positive integer $i,l,m,$ we have either $l\geq 2i-m$ or $l<2i-m.$ Depending on $i,l,j,m$ and $l\geq 2i-m$ the following proposition determines the sign of $F^{(m)}_{il}$ and $G^{(m)}_{i^{\prime}l}$ which will  be  useful to prove  Theorem \ref{main@}.
\begin{prop}\label{binomial11}
Let $i,l,j$ be integers such that {\small{$i\in\{1,\ldots,m\},0\leq l\leq  i,l\geq2i-m$}} and
{\small{$ j=i+l-m-2,\ldots, i-2.$}}  Then the following conditions hold:
\begin{enumerate}
\item[(a)] $F^{(m)}_{il}(j)=F^{(m)}_{il}(2i-m-4+l-j)$ for all $j$ with $i+l-m-2\leq j\leq i-2.$ In particular, $F^{(m)}_{il}(l-2)=F^{(m)}_{il}(2i-m-2)=0.$

\item[(b)] for all $j$ with $2i-m-2\leq j\leq l-2,$ $F^{(m)}_{il}(j)\leq 0.$

\item[(c)] for all $j$ with $i+l-m-2\leq j<2i-m-2,$ $F^{(m)}_{il}(j)> 0$  

\item[(d)]$F^{(m)}_{il}(j)> 0$ for all $j$ with $i-2\geq j>l-2.$

\item[(e)]  $F^{(m)}_{il}([i-\frac{m-l+5}{2}]+1)=\min_{i+l-m-2\leq j\leq i-2}F^{(m)}_{il}(j),$ that is, $F^{(m)}_{il}$ attains minimum value at $[i-\frac{m-l+5}{2}]+1.$  $F^{(m)}_{il}([i-\frac{m-l+5}{2}]+1)\leq 0.$ $F^{(m)}_{il}$ is decreasing for all $j$ with $i+l-m-2\leq j\leq [i-\frac{m-l+5}{2}]+1$ and $F^{(m)}_{il}$ is increasing for all $j$ with $ [i-\frac{m-l+5}{2}]+1\leq j\leq i-2.$
\end{enumerate}
\end{prop}
\begin{proof}
\noindent{\bf{(a)}} For fixed but arbitrary $j,$ from equation \eqref{fm} we have $$F^{(m)}_{il}(2i-m-4+l-j)=\binom{m}{i}\binom{m}{i-l}-\binom{m}{i+l-j-2}\binom{m}{i-j-2} ,$$ which implies $F^{(m)}_{il}(j)=F^{(m)}_{il}(2i-m-4+l-j).$ This shows that $F^{(m)}_{il}(j)=F^{(m)}_{il}(2i-m-4+l-j)$ for all $j$ with $i+l-m-2\leq j\leq i-2.$

\noindent{\bf{(b)}} The condition $F^{(m)}_{il}(j)< 0$ is equivalent to $\frac{\binom{m}{m-i+j+2}\binom{m}{m+2+j-i-l}}{\binom{m}{m-i}\binom{m}{m-i+l}}>1.$ Since $2i-m-2<j<l-2,$ we note that $i-j-2<m-i, i-j-2>i-l$ and $m-i+j+2>i.$ Thus, we have
\begin{align}\label{bino}
\frac{\binom{m}{m-i+j+2}\binom{m}{m+2+j-i-l}}{\binom{m}{m-i}\binom{m}{m-i+l}}\nonumber&=\frac{i!(m-i)!(i-l)!(m-i+l)!}{(m-i+j+2)!(i-j-2)!(m+2+j-i-l)!(i+l-j-2)!}\\&=\frac{(m-i)\ldots (i-j-1)(m-i+l)\ldots (i+l-j-1)}{(m-i+j+2)\ldots (i+1)(m+2+j-i-l)\ldots (i-l+1)}.
\end{align}

Since $2i-m-2<j<l-2,$ we observe that $(m-i+j-2)<m-i+l,\ldots,(i+l-j-2)>(i+1)$ and so on. Therefore, from \eqref{bino} we conclude that $\frac{\binom{m}{m-i+j+2}\binom{m}{m+2+j-i-l}}{\binom{m}{m-i}\binom{m}{m-i+l}}>1.$ This completes the proof of $(b).$

\noindent{\bf{(c)}}
To prove $(c)$, we need to verify $\frac{\binom{m}{m-i+j+2}\binom{m}{m+2+j-i-l}}{\binom{m}{m-i}\binom{m}{m-i+l}}<1$ for all $j$ with $i+l-m-2 \leq j<2i-m-2.$ Since $j<2i-m-2,$ we get $m-i+2+j<i$ and $i+l-j-2>m-i+l.$ By observing the above facts, note that 
\begin{align}\label{bino1}
\frac{\binom{m}{m-i+j+2}\binom{m}{m+2+j-i-l}}{\binom{m}{m-i}\binom{m}{m-i+l}}&=\frac{i\ldots (m-i+3+j)(i-l)\ldots (m-i+j+3-l)}{(i+l-j-2)\ldots (m-i+l+1) (i-j-2)\ldots (m-i+1)}
\end{align}
Since $l>2i-m$ and $j<2i-m-2,$ we have $i+l-j-2>i,\ldots, m-i+l+1>m-i+3+j$ and so on. Therefore, from \eqref{bino1} we conclude that $\frac{\binom{m}{m-i+j+2}\binom{m}{m+2+j-i-l}}{\binom{m}{m-i}\binom{m}{m-i+l}}<1.$ This completes the proof of $(c).$

\noindent{\bf{(d)}}
One can prove $(d)$ by simply observing the fact that $F^{(m)}_{il}(j)=F^{(m)}_{il}(2i-m-4+l-j)$ for all $j$ with $i+l-m-2\leq j\leq i-2.$

\noindent{\bf{(e)}}
In order to prove $(e),$  we  show that $\binom{m}{m-i+j+2}\binom{m}{m+2+j-i-l}$ attains maximum value at $[i-\frac{m+l-5}{2}]+1,$ that is, find out the maximum value of $j$ such that $\frac{\binom{m}{m-i+j+2}\binom{m}{m+2+j-i-l}}{\binom{m}{m-i+j+3}\binom{m}{m-i-l+j+3}}<1.$
Now, we have 
\begin{align}
\frac{\binom{m}{m-i+j+2}\binom{m}{m+2+j-i-l}}{\binom{m}{m-i+j+3}\binom{m}{m-i-l+j+3}}\nonumber&=\frac{(m+3-i+j)(m+3-i+j-l)}{(i-j-2)(i+l-j-2)}\\&=\frac{(m+1-x)(m+1-x-l)}{x(x+l)},
\end{align}
where $i-j-2=x.$

The condition $\frac{\binom{m}{m-i+j+2}\binom{m}{m+2+j-i-l}}{\binom{m}{m-i+j+3}\binom{m}{m-i-l+j+3}}<1$ is equivalent to $\frac{(m+1-x)(m+1-x-l)}{x(x+l)}<1,$ which gives $2j<2i-m+l-5.$ Similarly, one can also show that $\frac{\binom{m}{m-i+j+2}\binom{m}{m+2+j-i-l}}{\binom{m}{m-i+j+3}\binom{m}{m-i-l+j+3}}>1$ implies that $2j>2i-m+l-5.$ The above two conditions together imply that $\binom{m}{m-i+j+2}\binom{m}{m+2+j-i-l}$ attains maximum value at $[i-\frac{m-l+5}{2}]+1.$ 

Now, we show that $2i-m-2\leq [i-\frac{m-l+5}{2}]+1.$ If not, we have $2i-m-2>[i-\frac{m-l
+5}{2}]+1,$ which implies that $l<2i-m+1.$ This gives a contradiction to our assumption $l>2i-m.$ Similarly, one can also show that $[i-\frac{m-l+5}{2}]+1\leq l-2.$ By observing the  above facts, one can prove that $F^{(m)}_{il}([i-\frac{m-l+5}{2}]+1)\leq 0. $ From above, it is also clear that $F^{(m)}_{il}$ is decreasing for all $j$ with $i+l-m-2\leq j\leq [i-\frac{m-l+5}{2}]+1$ and $F^{(m)}_{il}$ is increasing for all $j$ with $ [i-\frac{m-l+5}{2}]+1\leq j\leq i-2.$ This completes the proof.
\end{proof}
As a consequence, we prove the following corollary.
\begin{cor}\label{binomial112}
Let $i^{\prime},l,j$ be integers such that {\small{$i\in\{1,\ldots,m\},0\leq l\leq  m+2-i^{\prime},l>m-2i^{\prime}+4$}} and
{\small{$ j=l-i^{\prime},\ldots, m-i^{\prime}.$}}  Then the following conditions hold:
\begin{enumerate}
\item[(a)] $G^{(m)}_{i^{\prime}l}(j)=G^{(m)}_{i^{\prime}l}(m-2i^{\prime}+l-j)$ for all $j$ with $l-i^{\prime}\leq j\leq m-i^{\prime}.$ In particular, $G^{(m)}_{i^{\prime}l}(l-2)=G^{(m)}_{i^{\prime}l}(m-2i^{\prime}+2)=0.$

\item[(b)] for all $j$ with $m-2i^{\prime}+2\leq j\leq l-2,$ $G^{(m)}_{i^{\prime}l}(j)\geq 0.$

\item[(c)] for all $j$ with $l-i^{\prime}\leq j<m-2i^{\prime}+2,$ $G^{(m)}_{i^{\prime}l}(j)< 0.$  

\item[(d)]$G^{(m)}_{i^{\prime}l}(j)< 0$ for all $j$ with $m-i^{\prime}\geq j>l-2.$

\item[(e)]  $G^{(m)}_{i^{\prime}l}([\frac{m+l-1}{2}-i^{\prime}]+1)=\max_{l-i^{\prime}\leq j\leq m-i^{\prime}}G^{(m)}_{i^{\prime}l}(j),$ that is, $G^{(m)}_{i^{\prime}l}$ attains maximum value at $[\frac{m+l-1}{2}-i^{\prime}]+1.$ $G^{(m)}_{i^{\prime}l}([\frac{m+l-1}{2}-i^{\prime}]+1)\geq 0.$ $G^{(m)}_{i^{\prime}l}$ is increasing for all $j$ with $l-i^{\prime}\leq j\leq [\frac{m+l-1}{2}-i^{\prime}]+1$ and $G^{(m)}_{i^{\prime}l}$ is decreasing for all $j$ with $ [\frac{m+l-1}{2}-i^{\prime}]+1\leq j\leq m-i^{\prime}.$
\end{enumerate}
\end{cor}

\begin{proof}
Putting $i^{\prime}=m+2-i$ in \eqref{gm} and using Proposition \ref{binomial11}, one can easily show the required inequality.
\end{proof}
The following proposition also determines the signs of $F^{(m)}_{il}$ and $G^{(m)}_{i^{\prime}l}$ whenever $l<2i-m.$ As the proof of this proposition is similar to that of Proposition \ref{binomial11}, we skip this proof.
\begin{prop}\label{binomial1112}
Let $i,l,j$ be integers such that {\small{$i\in\{1,\ldots,m\},0\leq l\leq  i,l<2i-m$}} and
{\small{$ j=i+l-m-2,\ldots, i-2.$}}  Then the following conditions hold:
\begin{enumerate}
\item[(a)] $F^{(m)}_{il}(j)=F^{(m)}_{il}(2i-m-4+l-j)$ for all $j$ with $i+l-m-2\leq j\leq i-2.$ In particular, $F^{(m)}_{il}(l-2)=F^{(m)}_{il}(2i-m-2)=0.$

\item[(b)] for all $j$ with $l-2\leq j\leq 2i-m-2,$ $F^{(m)}_{il}(j)\leq 0.$

\item[(c)] for all $j$ with $i+l-m-2\leq j<l-2,$ $F^{(m)}_{il}(j)> 0$  

\item[(d)]$F^{(m)}_{il}(j)> 0$ for all $j$ with $i-2\geq j>2i-m-2.$

\item[(e)]  $F^{(m)}_{il}([i-\frac{m-l+5}{2}]+1)=\min_{i+l-m-2\leq j\leq i-2}F^{(m)}_{il}(j),$ that is, $F^{(m)}_{il}$ attains minimum value at $[i-\frac{m-l+5}{2}]+1.$  $F^{(m)}_{il}([i-\frac{m-l+5}{2}]+1)\leq 0.$ $F^{(m)}_{il}$ is decreasing for all $j$ with $i+l-m-2\leq j\leq [i-\frac{m-l+5}{2}]+1$ and $F^{(m)}_{il}$ is increasing for all $j$ with $ [i-\frac{m-l+5}{2}]+1\leq j\leq i-2.$
\end{enumerate}
\end{prop}
As a consequence, we will prove the following corollary and the proof of this corollary is similar to the Corollary \ref{binomial112}, we skip the proof.
\begin{cor}\label{binomial11222}
Let $i^{\prime},l,j$ be integers such that {\small{$i\in\{1,\ldots,m\},0\leq l\leq  m+2-i^{\prime},l<m-2i^{\prime}+4$}} and
{\small{$ j=l-i^{\prime},\ldots, m-i^{\prime}.$}}  Then the following conditions hold:
\begin{enumerate}
\item[(a)] $G^{(m)}_{i^{\prime}l}(j)=G^{(m)}_{i^{\prime}l}(m-2i^{\prime}+l-j)$ for all $j$ with $l-i^{\prime}\leq j\leq m-i^{\prime}.$ In particular, $G^{(m)}_{i^{\prime}l}(l-2)=G^{(m)}_{i^{\prime}l}(m-2i^{\prime}+2)=0.$

\item[(b)] for all $j$ with $l-2\leq j\leq m-2i^{\prime}+2,$ $G^{(m)}_{i^{\prime}l}(j)\geq 0.$

\item[(c)] for all $j$ with $l-i^{\prime}\leq j<l-2,$ $G^{(m)}_{i^{\prime}l}(j)< 0.$  

\item[(d)]$G^{(m)}_{i^{\prime}l}(j)< 0$ for all $j$ with $m-i^{\prime}\geq j>m-2i^{\prime}+2.$

\item[(e)]  $G^{(m)}_{i^{\prime}l}([\frac{m+l-1}{2}-i^{\prime}]+1)=\max_{l-i^{\prime}\leq j\leq m-i^{\prime}}G^{(m)}_{i^{\prime}l}(j),$ that is, $G^{(m)}_{i^{\prime}l}$ attains maximum value at $[\frac{m+l-1}{2}-i^{\prime}]+1.$ $G^{(m)}_{i^{\prime}l}([\frac{m+l-1}{2}-i^{\prime}]+1)\geq 0.$ $G^{(m)}_{i^{\prime}l}$ is increasing for all $j$ with $l-i^{\prime}\leq j\leq [\frac{m+l-1}{2}-i^{\prime}]+1$ and $G^{(m)}_{i^{\prime}l}$ is decreasing for all $j$ with $ [\frac{m+l-1}{2}-i^{\prime}]+1\leq j\leq m-i^{\prime}.$
\end{enumerate}
\end{cor}
\begin{rem}
If $l=2i-m,$ then from Proposition \ref{binomial11}, one can conclude that $F^{(m)}_{il}$ attains minimum value at $[i-\frac{m-l+5}{2}]=2i-m-2$ and $\min_{i+l-m-2\leq j<i-2}F^{(m)}_{il}(j)=0.$ Suppose $l<2i-m.$ Then between $l-2$ and $2i-m-2$ there are even or odd number of points. If there are even number of points between $l-2$ and $2i-m-2,$ then by observing the fact  $F^{(m)}_{il}(j)=F^{(m)}_{il}(2i-m-4+l-j)$ for all $j$ with $0\leq j\leq 2i-m-4+l,$ one can conclude that  $F^{(m)}_{il}([i-\frac{m+l-5}{2}]+1)=F^{(m)}_{il}([i-\frac{m+l-5}{2}]).$ If it contains odd number of points, then $F^{(m)}_{il}([i-\frac{m+l-5}{2}]+1)=\min_{i+l-m-2\leq j<i-2}F^{(m)}_{il}(j).$ We can also arrive the same conclusion when $l<2i-m.$
\end{rem}

For all $\bold z\in \mathbb D^{m}$ and $w\in \mathbb D$ and $i\in \{1,\ldots,m\},$ let $g_{ijl}^{(m)}$ and $g_{i^{\prime}lj}^{(m)}$ be the functions on $\mathbb D^{m+1}$ defined by  $$g_{ilj}^{(m)}(\bold z,w)=\Re [w(\bar{s}_{i}^{(m)}(\bold z)s_{i-l}^{(m)}(\bold z)-\bar{s}_{m+2+j-i}^{(m)}(\bold z)s_{m+2+j-l-i}^{(m)}(\bold z))]$$ $$\text{and}$$ $$g_{i^{\prime}lj}^{(m)}(\bold z,w)=\Re[ w(\bar{s}_{j+i^{\prime}}^{(m)}(\bold z)s_{j+i^{\prime}-l}^{(m)}(\bold z)-\bar{s}_{m+2-i^{\prime}}^{(m)}(\bold z)s_{m+2-i^{\prime}-l}^{(m)}(\bold z))]=-g_{ilj}^{(m)}(\bold z,w),$$  where $i^{\prime} =m+2-i.$ Let $\mathcal A:=\{(\bold z,w):
g_{ilj}^{(m)}(\bold z,w)> 0\} {\rm{~~and ~~}} \mathcal B:=\{(\bold z,w):g_{i^{\prime}lj}^{(m)}(\bold z,w)\geq 0\}$ be two subsets of $\mathbb D^{m+1}.$ The following proposition gives the relation between two sets $\mathcal A$ and $\mathcal B.$
\begin{prop}\label{mathcalA}
Let $\mathcal A$ and $\mathcal B$ be as above. Then $\mathcal B=\mathbb D^{m+1}\setminus\mathcal A.$  
\end{prop}
\begin{proof}
Consider  $(\bold z,w)\in \mathbb D^{m+1}\setminus\mathcal A.$ Then by definition of $\mathcal A,$ we get $g_{ilj}^{(m)}(\bold z,w)\leq 0$ which is equivalent to $g_{i^{\prime}lj}^{(m)}(\bold z,w)\geq 0.$
Hence, $(\bold z,w)\in \mathcal B.$ This shows that $\mathbb D^{m+1}\setminus\mathcal A\subset  \mathcal B.$ Using similar technique which is described above, one can verify that  $\mathcal B\subset  \mathbb D^{m+1}\setminus\mathcal A.$ Thus, we have $\mathbb D^{m+1}\setminus\mathcal A=\mathcal B.$ 
\end{proof}

From above Proposition, we observe that  $\mathcal A$ and  $\mathcal B$ are complement of each other. Therefore, for each $(\bold z,w)\in \mathbb D^{m+1},$ we have either $g_{ilj}^{(m)}(\bold z,w)> 0$ or $g_{i^{\prime}lj}^{(m)}(\bold z,w)\geq 0.$ For $i\geq 1, l\geq 1$ and $i^{\prime}\geq 1,$ every $g_{ilj}^{(m)}(\bold z,w)$ and $g_{i^{\prime}lj}^{(m)}(\bold z,w)$ can be written as

\begin{align}\label{m+j}
g_{ilj}^{(m)}(\bold z,w)\nonumber&=g_{il(j+1)}^{(m-1)}(\bold z^{\prime},w)+\Re[w\bar{z}_m (\bar{s}_{i-1}^{(m-1)}(\bold z^{\prime})s_{i-l}^{(m-1)}(\bold z^{\prime})-\bar{s}_{m+1+j-i}^{(m-1)}(\bold z^{\prime})s_{m+2+j-l-i}^{(m-1)}(\bold z^{\prime}))]\\\nonumber&+\Re[wz_m (\bar{s}_{i}^{(m-1)}(\bold z^{\prime})s_{i-l-1}^{(m-1)}(\bold z^{\prime})-\bar{s}_{m+2+j-i}^{(m-1)}(\bold z^{\prime})s_{m+1+j-l-i}^{(m-1)}(\bold z^{\prime}))]+|z_{m}|^2g_{(i-1)l(j-1)}^{(m-1)}(\bold z^{\prime},w)\\&=g_{il(j+1)}^{(m-1)}(\bold z^{\prime},w)+g_{(i-1)(l-1)(j-1)}^{(m-1)}(\bold z^{\prime},w)+g_{i(l+1)(j+1)}^{(m-1)}(\bold z^{\prime},w)+|z_{m}|^2g_{(i-1)l(j-1)}^{(m-1)}(\bold z^{\prime},w)
\end{align}  $$\text{and}$$ \begin{align}\label{m+j+J}
g_{i^{\prime}lj}^{(m)}(\bold z,w)&=g_{(i^{\prime}-1)l(j+1)}^{(m-1)}(\bold z^{\prime},w)+g_{i^{\prime}(l-1)(j-1)}^{(m-1)}(\bold z^{\prime},w)+g_{(i^{\prime}-1)(l+1)(j+1)}^{(m-1)}(\bold z^{\prime},w)+|z_{m}|^2g_{i^{\prime}l(j-1)}^{(m-1)}(\bold z^{\prime},w),
\end{align}
where $\bold z^{\prime}\in \mathbb D^{(m-1)}.$
The following lemma gives the estimate for the norm of $g_{ilj}^{(m)}(\bold z,w)$ and $g_{i^{\prime}lj}^{(m)}(\bold z,w).$  It is expected that this result may be important for proving some elementary properties of $\Gamma_m.$
\begin{lem}\label{c_i-c_{n-i}}
Let $\bold z \in \mathbb D^{m}$ and $i^{\prime}=m+2-i.$ Then 
for all $w\in \mathbb D$ and for all $i$ with $i\in\{1,\ldots,m\},0\leq l\leq  i$ and
$ j=(i+l-m-2),\ldots, (i-2) $ and $j\neq 2i-m-2$   
\begin{eqnarray}\label{gijl}
|g_{ilj}^{(m)}( \bold z,w)|=|g_{i^{\prime}lj}^{(m)}(\bold z,w)| \leq \left\{ \begin{array}{ll}
         \big|F_{il}^{(m)}(j)\big|=\big|G_{i^{\prime}l}^{(m)}(j)\big| & {\rm{if}}~~ j\neq (l-2) ;\vspace{2mm}\\
        \big|\binom{m}{m-1-i}\binom{m}{m-1-i+l}-\binom{m}{i-1}\binom{m}{i-l-1}\big| & {\rm{if}}~~ j=l-2,i\geq(l+1);
        \vspace{2mm}\\ \binom{m}{i} & {\rm{if}}~~j=l-2,i=l 
        \vspace{2mm}\\ \binom{m}{l} & {\rm{if}}~~j=l-2,i=m
        \end{array} \right.
\end{eqnarray}
However,  if $j=2i-m-2,$ then $g_{ilj}^{(m)}( \bold z,w)=0=g_{i^{\prime}lj}^{(m)}(\bold z,w).$

\end{lem} 
 
\begin{proof}
We prove it by the method of induction on $m.$ 

\noindent{\underline{\bf{m=1:}}} 
For $m=1,$ we get $$g_{1lj}^{(1)}(z,w)=\Re [w(\bar{s}^{(1)}_{1}(z)s^{(1)}_{1-l}(z)-\bar{s}^{(1)}_{2+j}(z)s^{(1)}_{2+j-l}(z))]$$ $$\textit{and}$$ $$g_{2lj}^{(1)}(z,w)=\Re [w(\bar{s}^{(1)}_{2+j}(z)s^{(1)}_{2+j-l}(z)-\bar{s}^{(1)}_{1}(z)s^{(1)}_{1-l}(z))].$$

By Proposition \ref{mathcalA}, we have either $g_{1lj}^{(1)}(z,w)>0$ or $g_{2lj}^{(1)}(z,w)\geq 0.$  Suppose $g_{1lj}^{(1)}(z,w)>0.$ For $i=1,$ we have to consider two values of $l,$ that is, $l=0,1.$  For $l=0,$ we can take two values of $j$, that is, $j=-2,-1.$ If $m=1,i=1,l=0$ and $j=-2,$ then the fourth case of the inequality \eqref{gijl} is satisfied. For $m=1,i=1,l=0$ and $j=-1,$ we have $g_{10(-1)}^{(1)}(z,w)=0.$ Now, for $m=1,i=1$ and $l=1,$ the value of $j$ is $j=-1.$  Clearly, $g_{11(-1)}^{(1)}(z,w)=0.$ Similarly, one can prove the required estimates whenever $g_{2lj}^{(1)}(z,w)\geq 0.$ 
This shows that the result is true for $m=1.$
As the hypothesis of the induction, we assume that $m>3$ and  the result holds for all integers up to $(m-1),$ that is, 
\begin{enumerate}
\item 
for all $\bold z^{\prime} \in \mathbb D^{m-1},w\in \mathbb D~$ and for all $i$ with $i\in\{1,\ldots,(m-1)\},0\leq l\leq  i,$ 
$ j=(i+l-m-1),\ldots, (i-2),j\neq 2i-m-1$ and $i^{\prime}=m+1-i,$ 
\[
|g_{ilj}^{(m-1)}( \bold z^{\prime},w)|= |g_{i^{\prime}lj}^{(m-1)}(\bold z^{\prime},w)|\leq \left\{ \begin{array}{ll}
         \big|F_{il}^{(m-1)}(j)\big|=\big|G_{i^{\prime}l}^{(m-1)}(j)\big| & {\rm{if}}~~ j\neq (l-2) ;\vspace{2mm}\\
        \big|\binom{m-1}{m-2-i}\binom{m-1}{m-2-i+l}-\binom{m-1}{i-1}\binom{m-1}{i-l-1}\big| & {\rm{if}}~~ j=l-2,i\geq(l+1);
        \vspace{2mm}\\ \binom{m-1}{i} & {\rm{if}}~~j=l-2,i=l 
        \vspace{2mm}\\ \binom{m-1}{l} & {\rm{if}}~~j=l-2,i=m-1
        \end{array} \right.
\]
\item   $g_{il(2i-m-1)}^{(m-1)}( \bold z^{\prime},w)=0=g_{i^{\prime}l(m-2i^{\prime}+1)}^{(m-1)}(\bold z^{\prime},w).$
\end{enumerate}

For each $(\bold z,w)\in \mathbb D^{m+1},$ we have either $g_{ilj}^{(m)}(\bold z,w)> 0$ or $g_{i^{\prime}lj}^{(m)}(\bold z,w)\geq 0$( See Proposition \ref{mathcalA}). We first present a general outline of the method used in the flowchart below. 

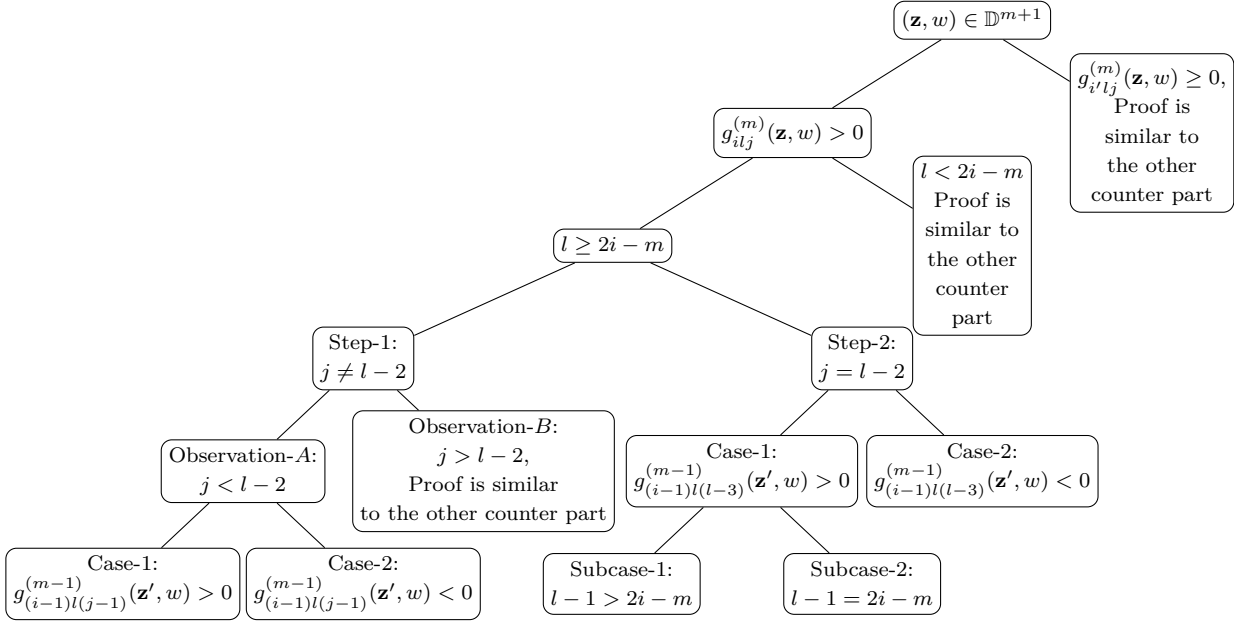
\scriptsize\begin{tikzpicture}[sibling distance=16em,
  every node/.style = {shape=rectangle, rounded corners,
    draw, align=center},level 3/.style={sibling distance=22.2em},
  level 4/.style={sibling distance=10.7em}]
  \node {$(\bold z,w)\in\mathbb{D}^{m+1}$}
    child { node {$g^{(m)}_{ilj}(\bold z,w)> 0$}
      child { node {$l\geq 2i-m$}
        child { node {Step-$1$: \\$j\neq l-2$}
          child { node {Observation-$A$: \\$j<l-2$}
            child { node {Case-$1$:\\$g_{(i-1)l(j-1)}^{(m-1)}(\bold z^{\prime},w)>0$}}
            child { node {Case-$2$:\\$g_{(i-1)l(j-1)}^{(m-1)}(\bold z^{\prime},w)<0$}}
          }
          child { node {Observation-$B$:\\$j>l-2,$\\Proof is similar \\ to the other counter part}}
        }
        child { node {Step-$2$: \\$j=l-2$}
          child { node {Case-$1$:\\$g_{(i-1)l(l-3)}^{(m-1)}(\bold z^{\prime},w)>0$}
            child { node {Subcase-$1$:\\$l-1>2i-m$}}
            child { node {Subcase-$2$:\\ $l-1=2i-m$}}
          }
          child { node {Case-$2$:\\$g_{(i-1)l(l-3)}^{(m-1)}(\bold z^{\prime},w)<0$}}
        }
      }
      child { node {$l<2i-m$\\Proof is \\similar to \\the other \\counter\\ part} }
    } 
    child { node {$g^{(m)}_{i^{\prime}lj}(\bold z,w)\geq 0,$\\Proof is \\similar to \\the other \\counter part}
  };
\end{tikzpicture}\normalsize

We prove the required estimate for $g_{ilj}^{(m)}(\bold z,w)> 0,$ where the proof of the other required estimate for $g_{i^{\prime}lj}^{(m)}(\bold z,w)\geq 0$ is similar. Also, for positive integers $i,l,m,$ we have two possibilities: either $l\geq 2i-m$ or $l<2i-m.$ To get the above estimates, we consider the case $l\geq 2i-m,$ while the proof of the other case holds similarly. We  prove the above estimates for $g_{ilj}^{(m)}(\bold z,w)> 0$ and $l\geq 2i-m$ in two steps. In the first step we  assume $j\neq l-2$ whereas in the second  we consider the case $j=l-2$.

\noindent{\bf Step-1 \{$j\neq l-2$\}}: In order to prove the estimate for $j\neq l-2$ we need to consider  the following observations either $j>l-2$ or $j<l-2.$ 

\noindent{\bf Observations-A:}\{ $j< l-2$\}:
For $j< l-2$, we need to consider the following cases

\noindent{\bf Case-1:} \{$g_{(i-1)l(j-1)}^{(m-1)}(\bold z^{\prime},w)>0$\}:
From \eqref{m+j}, we get 
\begin{eqnarray}\label{m+j+1}
g_{ilj}^{(m)}(\bold z,w)\nonumber &=&g_{il(j+1)}^{(m-1)}(\bold z^{\prime},w)+g_{(i-1)(l-1)(j-1)}^{(m-1)}(\bold z^{\prime},w)+g_{i(l+1)(j+1)}^{(m-1)}(\bold z^{\prime},w)+|z_{m}|^2g_{(i-1)l(j-1)}^{(m-1)}(\bold z^{\prime},w)\\\nonumber &\leq&(g_{il(j+1)}^{(m-1)}(\bold z^{\prime},w)+g_{(i-1)l(j-1)}^{(m-1)}(\bold z^{\prime},w))+g_{(i-1)(l-1)(j-1)}^{(m-1)}(\bold z^{\prime},w)+g_{i(l+1)(j+1)}^{(m-1)}(\bold z^{\prime},w)\\&=& (g_{ilj}^{(m-1)}(\bold z^{\prime},w)+g_{(i-1)lj}^{(m-1)}(\bold z^{\prime},w))+g_{(i-1)(l-1)(j-1)}^{(m-1)}(\bold z^{\prime},w)+g_{i(l+1)(j+1)}^{(m-1)}(\bold z^{\prime},w).
\end{eqnarray}
If any one real part of \eqref{m+j+1} is negative, then we can dominate it by zero and proceed similarly. Therefore, we consider only the  subcase when
\small{$g_{il(j+1)}^{(m-1)}(\bold z^{\prime},w)>0,g_{i(l+1)(j+1)}^{(m-1)}(\bold z^{\prime},w)>0, g_{(i-1)(l-1)(j-1)}^{(m-1)}(\bold z^{\prime},w)>0.$}
By mathematical induction, one can easily verify that $g_{ilj}^{(m-1)}(\bold z^{\prime},w)\leq \big|F_{il}^{(m-1)}(j)\big|,g_{(i-1)lj}^{(m-1)}(\bold z^{\prime},w)\leq \big|F_{(i-1)l}^{(m-1)}(j)\big|,$ $g_{(i(l+1)(j+1)}^{(m-1)}(\bold z^{\prime},w)\leq \big|F_{i(l+1)}^{(m-1)}(j+1)\big| $ and $g_{(i-1)(l-1)(j-1)}^{(m-1)}(\bold z^{\prime},w)\leq \big|F_{(i-1)(l-1)}^{(m-1)}(j-1)\big|.$ 
Sign of $F^{(m-1)}_{(i-1)l}(j),F^{(m-1)}_{il}(j),$ $F^{(m-1)}_{i(l+1)}(j+1)$ and  $F^{(m-1)}_{(i-1)(l-1)}(j-1)$ can be determined by Proposition \ref{binomial11} depending on $i,l,j,m.$
%

By Proposition \ref{binomial11}, we note that  $F^{(m-1)}_{(i-1)l}(j),F^{(m-1)}_{il}(j),F^{(m-1)}_{i(l+1)}(j+1)$ and  $F^{(m-1)}_{(i-1)(l-1)}(j-1)$ are positive for all $j$ with $i+l-m-2\leq j< 2i-m-3$ and $F^{(m-1)}_{(i-1)l}(j),F^{(m-1)}_{il}(j),F^{(m-1)}_{i(l+1)}(j+1)$ and $F^{(m-1)}_{(i-1)(l-1)}(j-1)$ are negative for all $j$ with $2i-m-2\leq j< l-2.$ Also, we have $F^{(m-1)}_{il}(2i-m-1)=0$ and $F^{(m-1)}_{(i-1)l}(2i-m-3)=0.$ Hence, if $i+l-m-2\leq j< 2i-m-3$, then from \eqref{m+j+1} we observe that 
\small{\begin{align*}
g_{ilj}^{(m)}(\bold z,w)\nonumber&\leq g_{ilj}^{(m-1)}(\bold z^{\prime},w)+g_{(i-1)(l-1)(j-1)}^{(m-1)}(\bold z^{\prime},w)+g_{i(l+1)(j+1)}^{(m-1)}(\bold z^{\prime},w)+g_{(i-1)lj}^{(m-1)}(\bold z^{\prime},w)\nonumber\\&\leq F^{(m-1)}_{(i-1)l}(j)+F^{(m-1)}_{il}(j)+F^{(m-1)}_{i(l+1)}(j+1)+F^{(m-1)}_{(i-1)(l-1)}(j-1)\\\nonumber&= \binom{m-1}{i-1}\binom{m-1}{i-l-1}-\binom{m-1}{m+2-i+j}\binom{m-1}{m+2-i-l+j}\\\nonumber\vspace{1mm}&+\binom{m-1}{i}\binom{m-1}{i-l}-\binom{m-1}{m+1-i+j}\binom{m-1}{m+1-i-l+j}\\\nonumber\vspace{1mm}&+\binom{m-1}{i}\binom{m-1}{i-l-1}-\binom{m-1}{m+2-i+j}\binom{m-1}{m+1-i-l+j}\\\nonumber\vspace{1mm}&+\binom{m-1}{i-1}\binom{m-1}{i-l}-\binom{m-1}{m+1-i+j}\binom{m-1}{m+-i-l+j}\\\nonumber\vspace{1mm}&=\binom{m}{i}\binom{m}{i-l}-\binom{m}{m+2-i+j}\binom{m}{m+2-i-l+j}=F^{(m)}_{il}(j).
\end{align*}} 
Similarly, we can show that $g_{ilj}^{(m)}(\bold z,w)\leq |F^{(m)}_{il}(j)|$ whenever $2i-m\leq j< l-2.$ Also, using Proposition \ref{binomial11}, we get  $g_{il(2i-m-2)}^{(m)}( \bold z,w)=0.$ So, we now  have to investigate the case $j=2i-m-1$ and $j=2i-m-3$ only. Since $F^{(m-1)}_{(i-1)l}(2i-m-3)>0,F^{(m-1)}_{il}(2i-m-3)>0,F^{(m-1)}_{i(l+1)}(2i-m-2)>0$ and $F^{(m-1)}_{(i-1)l}(2i-m-3)=0,$ from \eqref{m+j+1}, we conclude that 
\begin{align*}
g_{il(2i-m-3)}^{(m)}(\bold z,w)\nonumber&\leq g_{il(2i-m-3)}^{(m-1)}(\bold z^{\prime},w)+g_{(i-1)(l-1)(2i-m-4)}^{(m-1)}(\bold z^{\prime},w)+g_{i(l+1)(2i-m-2)}^{(m-1)}(\bold z^{\prime},w)+g_{(i-1)l(2i-m-3)}^{(m-1)}(\bold z^{\prime},w)\nonumber\\&\leq F^{(m-1)}_{(i-1)l}(j)+F^{(m-1)}_{il}(j)+F^{(m-1)}_{i(l+1)}(j+1)=F^{(m)}_{il}(2i-m-3).
\end{align*}
By using similar procedure, one can show that $g_{il(2i-m-1)}^{(m)}(\bold z,w)\leq |F^{(m)}_{il}(2i-m-1)|.$ This completes the proof of this subcase-$1.$

\noindent{\bf Case-2:} \{$g_{(i-1)l(j-1)}^{(m-1)}(\bold z^{\prime},w)<0$\}:
From \eqref{m+j}, we note that
\begin{align}\label{mjj}
g_{ilj}^{(m)}(\bold z,w)&\leq g_{il(j+1)}^{(m-1)}(\bold z^{\prime},w)+g_{(i-1)(l-1)(j-1)}^{(m-1)}(\bold z^{\prime},w)+g_{i(l+1)(j+1)}^{(m-1)}(\bold z^{\prime},w)
\end{align}
Similar to the Case-$1$, if any one real part of \eqref{mjj} is negative, then we can dominate it by zero and proceed similarly. Therefore, for this case also we  consider only the  subcase when \small{\mbox{$g_{il(j+1)}^{(m-1)}(\bold z^{\prime},w)>0, g_{(i-1)(l-1)(j-1)}^{(m-1)}(\bold z^{\prime},w)>0,$}} $g_{i(l+1)(j+1)}^{(m-1)}(\bold z^{\prime},w)>0.$
We  verify this subcase only when $j=l-3,$ because the proof of this subcase for other value of $j$ are similar to the  Case-$1.$ It follows from induction hypothesis that  $$g_{il(l-2)}^{(m-1)}(\bold z^{\prime},w)\leq \big|\binom{m-1}{m-2-i}\binom{m-1}{m-2-i+l}-\binom{m-1}{i-1}\binom{m-1}{i-l-1}\big|,$$ $$g_{(i-1)(l-1)(l-4)}^{(m-1)}(\bold z^{\prime},w)\leq \big|F_{(i-1)(l-1)}^{(m-1)}(l-4)\big|~{\rm and}~ g_{i(l+1)(l-2)}^{(m-1)}(\bold z^{\prime},w)\leq \big|F_{i(l+1)}^{(m-1)}(l-2)\big|.$$ Using similar technique which is described in  Case-$1,$ from \eqref{mjj} we conclude that $g_{il(l-3)}^{(m)}(\bold z,w)\leq \big|F_{il}^{(m-1)}(j)\big|.$

\noindent{\bf Observations-B:}\{ $j< l-2$\}:
The proof of the observation-$B$ is similar to the Observation-$A.$

\noindent{\bf Step-2:}\{$j=l-2$\}: 
In this step, as before we need to consider two cases.

\noindent{\bf Case-1:} \{$g_{(i-1)l(l-3)}^{(m-1)}(\bold z^{\prime},w)>0$\}: Because $l>2i-m,$ we need to consider two subcases

\noindent{\bf Subcase-1:} \{$l-1>2i-m$\}:  By Proposition \ref{binomial11}, we note that  $F^{(m-1)}_{(i-1)l}(l-2),F^{(m-1)}_{il}(l-2),F^{(m-1)}_{i(l+1)}(l-1)$ and  $F^{(m-1)}_{(i-1)(l-1)}(l-3)$ are positive. Using similar technique which is described in Step-$1$ of Case-$1$, one can show the required estimate.

\noindent{\bf Subcase-2:} \{$l-1=2i-m$\}:  By  Proposition \ref{binomial11}, we see that  $F^{(m-1)}_{(i-1)l}(l-2)=0$ and $F^{(m-1)}_{il}(l-2),F^{(m-1)}_{i(l+1)}(l-1)$ and  $F^{(m-1)}_{(i-1)(l-1)}(l-3)$ are positive. Arguing similarly as in Step-$1$ of Case-$1$, one can prove the required estimate.

\noindent{\bf Case-2:} \{$g_{(i-1)l(l-3)}^{(m-1)}(\bold z^{\prime},w)<0$\}: In this case, by Proposition \ref{binomial11}, it follows that  $F^{(m-1)}_{il}(l-1),F^{(m-1)}_{i(l+1)}(l-1)$ and  $F^{(m-1)}_{(i-1)(l-1)}(l-3)$ are positive. Using similar argument as in Step-$1$ of Case-$2$, one can get the desired estimate. This completes the proof.
\end{proof}
The previous lemma will be useful to prove Theorem \ref{main@}.
\begin{proof}[Proof of Theorem \ref{main@}]: We prove it by induction on $(m+1).$ For $m=1,$ the result follows from Lemma \ref{c_i-c_{n-i}}.
By mathematical induction, we assume that $m>1$ and the result holds for all integers up to $m$, that is,
\begin{enumerate}
\item $\bold z^{\prime} \in \mathbb D^{m}$ and for all $i$ with $ i\in  \{1,\ldots,m
\},$  $j=i-m-1,\ldots,(i-1),j\neq 2i-m-1$ and $i^{\prime}=m+1-i,$
\begin{eqnarray}\label{s_{m+1+1}}
\big|h^{(m)}_{i^{\prime}j}(\bold z^{\prime})\big|=\big|h^{(m)}_{ij}(\bold z^{\prime})\big| \leq \left\{ \begin{array}{ll}
         \big|\binom{m}{i}^2-\binom{m}{m+1+j-i}^2\big|& {\rm{if}}~~ j\neq1;\vspace{2mm}\\
        \big|\binom{m}{m-i-1}^2-\binom{m}{i-1}^2\big| & {\rm{if}}~~ j=-1;
         \vspace{2mm}\\ 1 & {\rm{if}}~~j=-1,i=m
        \end{array} \right.
\end{eqnarray}
\item However,  if $j=2i-m-1,$ then $|h^{(m)}_{i^{\prime}j}(\bold z^{\prime})=|h^{(m)}_{ij}(\bold z^{\prime})=0.$ 
\end{enumerate}
Note that 
\small{\begin{align}\label{s1m+1}
h^{(m+1)}_{ij}(\bold z)\nonumber &=(|s_{i}^{(m)}(\bold z^{\prime})|^2-|s_{m+2+j-i}^{(m)}(\bold z^{\prime})|^2)+|z_{m+1}|^2(|s_{i-1}^{(m)}(\bold z^{\prime})|^2-|s_{m+1+j-i}^{(m)}(\bold z^{\prime})|^2)\\ &+2\Re( \bar{s}_{i}^{(m)}(\bold z^{\prime})s_{i-1}^{(m)}(\bold z^{\prime})-\bar{s}_{m+2+j-i}^{(m)}(\bold z^{\prime})s_{m+1+j-i}^{(m)}(\bold z^{\prime}))
\end{align}}
By Proposition \ref{mathcalA}, we have either $h^{(m+1)}_{ij}(\bold z)>0$ or $h^{(m+1)}_{i^{\prime}j}(\bold z)\geq 0.$ We verify the required estimate for $h^{(m+1)}_{ij}(\bold z)>0,$ as the proof of the other required estimate for $h^{(m+1)}_{i^{\prime}j}(\bold z)\geq 0$ is similar. As before, for positive integer $i,m,$ we have either $2i>m$ or $2i<m.$ We consider the case when $2i>m,$ as the proof of the other case hold similarly. Thus, we wish to verify our required estimate for  $h^{(m+1)}_{ij}(\bold z)>0$ and $2i>m$ in two steps.     

\noindent{\bf Step-1:}\{$j\neq -1$\}:
In this step, we have two observations. 

\noindent{\bf Observation-A:}\{$j< -1$\}
In this observation, we need to consider two cases.

\noindent{\bf Case-1:}\{$|s_{i-1}^{(m)}(\bold z^{\prime})|^2>|s_{m+1+j-i}^{(m)}(\bold z^{\prime})|^2$\}: From \eqref{s1m+1}, we have 
\small{\begin{align}\label{s1m+11}
h^{(m+1)}_{ij}(\bold z) &\leq h^{(m)}_{ij}(\bold z^{\prime})+h^{(m)}_{(i-1)j}(\bold z^{\prime})+2\Re( \bar{s}_{i}^{(m)}(\bold z^{\prime})s_{i-1}^{(m)}(\bold z^{\prime})-\bar{s}_{m+2+j-i}^{(m)}(\bold z^{\prime})s_{m+1+j-i}^{(m)}(\bold z^{\prime}))
\end{align}}
If any one real part of \eqref{s1m+11} is negative, then we can dominate it by zero and proceed similarly. Therefore, we verify only the subcase
when \small{\mbox{$|s_{i}^{(m)}(\bold z^{\prime})|^2>|s_{m+2+j-i}^{(m)}(\bold z^{\prime})|^2, \Re( \bar{s}_{i}^{(m)}(\bold z^{\prime})s_{i-1}^{(m)}(\bold z^{\prime})-\bar{s}_{m+2+j-i}^{(m)}(\bold z^{\prime})s_{m+1+j-i}^{(m)}(\bold z^{\prime}))>0.$}} By induction hypothesis and using Lemma \ref{c_i-c_{n-i}}, we deduce the following estimates:
\begin{enumerate}
\item $|s_{i}^{(m)}(\bold z^{\prime})|^2-|s_{m+1+j-i}^{(m)}(\bold z^{\prime})|^2 \leq \big|\binom{m}{i}^2-\binom{m}{m+1+j-i}^2\big|,$

\item $\Re( \bar{s}_{i}^{(m)}(\bold z^{\prime})s_{i-1}^{(m)}(\bold z^{\prime})-\bar{s}_{m+2+j-i}^{(m)}(\bold z^{\prime})s_{m+1+j-i}^{(m)}(\bold z^{\prime}))\leq \big|\binom{m}{i}\binom{m}{i-1}-\binom{m}{m+2+j-i}\binom{m}{m+1+j-i}\big|$

\item $|s_{i-1}^{(m)}(\bold z^{\prime})|^2-|s_{m+2+j-i}^{(m)}(\bold z^{\prime})|^2 \leq \big|\binom{m}{i-1}^2-\binom{m}{m+2+j-i}^2\big|.$

\end{enumerate}  
By using similar argument which is described in Lemma \ref{c_i-c_{n-i}}, one can easily show the required estimate.

\noindent{\bf Case-2:}\{$|s_{i-1}^{(m)}(\bold z^{\prime})|^2<|s_{m+1+j-i}^{(m)}(\bold z^{\prime})|^2$\}: From \eqref{s1m+1}, we have 
\begin{align}\label{s1m+111}
h^{(m+1)}_{ij}(\bold z) &\leq (|s_{i}^{(m)}(\bold z^{\prime})|^2-|s_{m+2+j-i}^{(m)}(\bold z^{\prime})|^2)+2\Re( \bar{s}_{i}^{(m)}(\bold z^{\prime})s_{i-1}^{(m)}(\bold z^{\prime})-\bar{s}_{m+2+j-i}^{(m)}(\bold z^{\prime})s_{m+1+j-i}^{(m)}(\bold z^{\prime}))
\end{align}
Similar to the Lemma \ref{c_i-c_{n-i}}, one can verify the required estimate.

\noindent{\bf Step-2:}\{$j= -1$\}: Applying similar technique which in Lemma \ref{c_i-c_{n-i}}, one can easily prove the required estimate.
This completes the proof.
\end{proof}

\begin{rem}
For $j\neq l-2,$ we notice that $|g_{ilj}^{(m)}((1,\ldots,1),w)|=\big|F_{il}^{(m)}(j)\big|=|g_{i^{\prime}lj}^{(m)}((1,\ldots,1),w)|.$
Therefore, for $j\neq l-2,$ by using Lemma \ref{c_i-c_{n-i}}, we deduce that \small{$\sup_{(\bold z,w)\in\mathbb D^{m+1}}|g_{ilj}^{(m)}(\bold z,w)|=\big|F_{il}^{(m)}(j)\big|=\sup_{(\bold z,w)\in\mathbb D^{m+1}}|g_{i^{\prime}lj}^{(m)}(\bold z,w)|.$} This shows that for $j\neq l-2$ the above estimates are sharp. Similarly, for $j\neq -1,$ from Theorem \ref{main@},  we also have $\sup_{\bold z\in\mathbb D^{m+1}}\big|h^{(m+1)}_{ij}(\bold z)\big|=
         \big|\binom{m+1}{i}^2-\binom{m+1}{m+2+j-i}^2\big|=\sup_{\bold z\in\mathbb D^{m+1}}\big|h^{(m+1)}_{i^{\prime}j}(\bold z)\big|.$
\end{rem}

\section{Schur's criterion and some properties of Symmetrized polydisc}
It is  well known that a point $(s_1,\ldots,s_{n-1},s_n)\in \mathbb G_n$ is equivalent to the polynomial $P(z)=z^n-s_1z^{n-1}+\ldots+(-1)^ns_n$ has all roots in $\mathbb D.$ Schur's criterion \cite{patak,Schur} is a standard tool for deciding whether the zeros of the polynomial lie in $\mathbb D.$ 
Let $k(i)= \binom{n-1}{i}+\binom{n-1}{n-i}$ for $i=1,\ldots,n-1.$ For $i=1,\ldots,(n-1),$ we define operator pencils $\Phi_{1}^{(i)}$ and $\Phi_{2}^{(i)}$ for a commuting $n$ tuple of bounded operators $(S_1,\ldots,S_n)$ as 
\small{\begin{align}\label{phi^i}
\Phi^{(i)}_{1}(\alpha^{i}S_i,\alpha^{n-i}S_{n-i},\alpha^{n}S_n)\nonumber&=k(i)^2(1-|\alpha^nS_n|^2)+(|\alpha^{i}S_{i}|^2-|\alpha^{n-i}S_{n-i}|^2)\\&-k(i)\alpha^{i}(S_{i}-|\alpha|^{2n-2i}S^*_{n-i}S_n)-k(i)\bar{\alpha}^{i}(S^*_{i}-|\alpha|^{2n-2i}S_{n-i}S_n^*)
\end{align}}
$$\textit{\rm{and}}$$
\small{\begin{align}\label{phi^i1}
\Phi^{(i)}_{2}(\alpha^{i}S_i,\alpha^{n-i}S_{n-i},\alpha^{n}S_n)\nonumber&=k(i)^2(1-|\alpha^nS_n|^2)+(|\alpha^{n-i}S_{n-i}|^2-|\alpha^iS_i|^2)\\&-k(i)\alpha^{n-i}(S_{n-i}-|\alpha|^{2i}S^*_{i}S_n)-k(i)\bar{\alpha}^{n-i}(S^*_{n-i}-|\alpha|^{2i}S_{i}S_n^*).
\end{align}}
We show that these operator pencils play an important role in determining  the structure of $\Gamma_n$-contractions.
In particular, if $S_1,\ldots,S_n$ are scalar, then $\Phi_{1}^{(i)}$ and $\Phi_{2}^{(i)}$ are of the following form 
\small{\begin{align}\label{phi^i2}
\Phi^{(i)}_{1}(\alpha^{i}s_i,\alpha^{n-i}s_{n-i},\alpha^{n}s_n)\nonumber&=k(i)^2(1-|\alpha^ns_n|^2)+(|\alpha^{i}s_{i}|^2-|\alpha^{n-i}s_{n-i}|^2)\\&-k(i)\alpha^{i}(s_{i}-|\alpha|^{2n-2i}\bar{s}_{n-i}s_n)-k(i)\bar{\alpha}^{i}(\bar{s}_{i}-|\alpha|^{2n-2i}s_{n-i}\bar{s_n})
\end{align}}
$$\textit{\rm{and}}$$
\small{\begin{align}\label{phi^i3}
\Phi^{(i)}_{2}(\alpha^{i}s_i,\alpha^{n-i}s_{n-i},\alpha^{n}s_n)\nonumber&=k(i)^2(1-|\alpha^ns_n|^2)+(|\alpha^{n-i}s_{n-i}|^2-|\alpha^is_i|^2)\\&-k(i)\alpha^{n-i}(s_{n-i}-|\alpha|^{2i}\bar{s}_{i}s_n)-k(i)\bar{\alpha}^{n-i}(\bar{s}_{n-i}-|\alpha|^{2i}s_{i}\bar{s}_n).
\end{align}}
The following propositions are characterization of $b\Gamma_n$ and $\Gamma_n.$

\begin{prop}\label{character of gamma-n}\cite[Theorem 2.4]{SS}

For $(s_1, \ldots,s_{n-1},s_n)\in \mathbb C^n, n\geq 2$ the following properties are equivalent:
\begin{enumerate}
\item $(s_1, \ldots,s_{n-1},s_n)\in b\Gamma_n;$

\item $(s_1, \ldots,s_{n-1},s_n)\in \Gamma_n$ and $|s_n|=1;$

\item $|s_n|=1, s_i=\bar{s}_{n-i}s_n$ and  $(\gamma_1s_1,\ldots,\gamma_{n-1}s_{n-1})\in \Gamma_{n-1},$ where $\gamma_i=\frac{n-i}{n}$ for $i=1,\ldots,(n-1);$ 

\item $|s_n|=1$ and there exist $(c_1, \ldots, c_{n-1}) \in b\Gamma_{n-1}$ such that $s_i=c_i+\bar{c}_{n-i}s_n, s_{n-i}=c_{n-i}+\bar{c}_{i}s_n$ for $i=1,\ldots,(n-1).$

\end{enumerate}
\end{prop}

\begin{prop}\label{gamma_n}\cite[Theorem 3.7]{Ccostara}

Suppose  $(s_1,\ldots, s_{n-1},s_n)\in \mathbb C^{n}$. Then the following conditions are equivalent:
\begin{enumerate}
\item $(s_1,\ldots, s_{n-1},s_n)\in \Gamma_{n};$

\item $|s_n|\leq 1$ and there exist $(c_1,\ldots,c_{n-1})\in \Gamma_{n-1}$ such that $s_i=c_i+\bar{c}_{n-i}s_n$ and $s_{n-i}=c_{n-i}+\bar{c}_{i}s_n$ for~ $i=1,\ldots,(n-1).$
\end{enumerate}
\end{prop}

For closed symmetrized polydisc $\Gamma_n,$ we have the following characterization.
\begin{thm}\label{main}
Suppose $(s_1,\ldots, s_{n-1},s_n)\in \Gamma_{n}.$ Then for $ i=1,\ldots,(n-1)$
\small{$$\Phi^{(i)}_{1}(\alpha^{i}s_i,\alpha^{n-i}s_{n-i},\alpha^{n}s_n)\geq 0~~{\rm~ and}~~ \Phi^{(i)}_{2}(\alpha^{i}s_{i},\alpha^{n-i}s_{n-i},\alpha^{n}s_n)\geq 0$$} for all $\alpha\in \overline{\mathbb D}$ with $ k(i)=\binom{n-1}{i}+\binom{n-1}{n-i}.$  

\end{thm}
\begin{proof}
Let $(s_1,\ldots,s_n)$ be a point in $ \Gamma_{n}.$ Then  $(\alpha s_1,\ldots, \alpha^n s_n)\in \mathbb G_{n}$ for all $\alpha \in \mathbb D.$ Consider the polynomial $$p(z)=z^n-\alpha s_1z^{n-1}+\ldots+(-1)^m\alpha^ms_mz^m+(-1)^{m+1}\alpha^{m+1}s_{m+1}z^{m+1}+\ldots+(-1)^n\alpha^ns_n.$$ Suppose $z_1,\ldots,z_n$ are the roots of the polynomial $p(z)=0.$ Then  $\alpha^is_i=\sum_{1\leq k_1< \ldots< k_i\leq n}z_{k_1}\ldots z_{k_i}.$  Let $S$ be the $n \times n$ ``Shift matrix'' which is of the form $S=\left( \begin{smallmatrix}0 &1 & 0 &0 & \ldots &0\\\vdots& \vdots &\vdots & \vdots &\vdots & \vdots \\0 &0 &0 &0 &\ldots&1\\0 &0 &0 &0 &\ldots&0
\end{smallmatrix}\right).$ Then by a  simple computation one can show that  \small{$$p(S)=\left( \begin{smallmatrix}(-1)^n\alpha^ns_n &(-1)^{n-1}\alpha^{n-1}s_{n-1}   & \ldots &\alpha^2s_2 &-\alpha s_1\\0 & (-1)^n\alpha^ns_n & \ldots& -\alpha^3s_3 &\alpha^2s_2\\\vdots& \vdots &\vdots & \vdots &\vdots  \\0 & 0 & \ldots & (-1)^n\alpha^ns_n & (-1)^{n-1}\alpha^{n-1}s_{n-1}\\0 &0  &\ldots&0 &(-1)^n\alpha^n s_n
\end{smallmatrix}\right)$$} $$\text{and }~~$$  \small{$$q(S)=\left( \begin{smallmatrix}1 &-\bar{\alpha}\bar{s}_1   & \ldots &(-1)^{n-2}\bar{\alpha}^{n-2}\bar{s}_{n-2} &(-1)^{n-1}\bar{\alpha}^{n-1}\bar{s}_{n-1}\\0 & 1 & \ldots &(-1)^{n-3}\bar{\alpha}^{n-3}\bar{s}_{n-3} &(-1)^{n-2}\bar{\alpha}^{n-2}\bar{s}_{n-2}\\\vdots& \vdots &\vdots & \vdots  & \vdots \\0 & 0 &\ldots & 1& -\bar{\alpha}\bar{s}_1\\0 &0  &\ldots& 0& 1
\end{smallmatrix}\right),$$} where $q(z)=z^n\overline{p(\frac{1}{z}}).$ Let $L=q(S)^*q(S)-p(S)^*p(S),a=1-|\alpha^ns_n|^2$ and $c=(-1)^{n-1}\bar{\alpha}^{n-1}(\bar{s}_{n-1}-|\alpha|^2s_1\bar{s}_n).$ Then
\begin{align*}
L&=\left( \begin{smallmatrix} a &-\bar{\alpha}(\bar{s}_1-|\alpha^{n-1}|^2s_{n-1}\bar{s}_n)  & \ldots  & c\\-\alpha(s_1-|\alpha^{n-1}|^2\bar{s}_{n-1}s_n) &a+(|\alpha s_1|^2-|\alpha^{n-1}s_{n-1}|^2)  & \ldots&(-1)^{n-2}\bar{\alpha}^{n-2}(\bar{s}_{n-2}-|\alpha|^4s_2\bar{s}_n)  \\\vdots &\vdots & \vdots & \vdots \\\bar{c} &(-1)^{n-2}\alpha^{n-2}(s_{n-2}-|\alpha|^4\bar{s}_2s_n) &\ldots & a
\end{smallmatrix}\right)\\&=\left( \begin{matrix} A_{(i)} & B_{(i)}\\ B_{(i)}^{*} & L_{(i)}
\end{matrix}\right)\end{align*} where \small{$L_{(i)}=\left( \begin{smallmatrix}
a+(|\alpha s_1|^2-|\alpha^{n-1}s_{n-1}|^2)+\ldots+(|\alpha_i s_i|^2-|\alpha^{n-i}s_{n-i}|^2)  & \ldots&(-1)^{n-i-1}\bar{\alpha}^{n-i-1}(\bar{s}_{n-i-1}-|\alpha|^{2(i+1)}s_{i+1}\bar{s}_n) \\ \vdots  & \vdots &\vdots  \\(-1)^{n-i-1}\alpha^{n-i-1}(s_{n-i-1}-|\alpha|^{2(i+1)}\bar{s}_{i+1}s_n) &\ldots & a
\end{smallmatrix}\right),$} $ A_{(i)}$ and $B_{(i)}$ are $(n-i)\times (n-i),i\times i$ and $i\times(n-i)$ matrices respectively and $L_{(0)}=L$ for $i=1,\ldots,n.$

By Schur's criterion \cite{patak}, the matrix $L$ is positive definite. It is well known that if the matrix $A$ is positive definite, then all leading principal minors of $A$ are positive definite\cite{Horn}. Since $L$ is positive definite,  $L_{(i)}$ and $S(1,n-i)$ are positive definite for $i=1,\ldots,n$, where \small{$$S(1,n-i)=\Big( \begin{smallmatrix}  a+(|\alpha s_1|^2-|\alpha^{n-1}s_{n-1}|^2)+\ldots+(|\alpha_i s_i|^2-|\alpha^{n-i}s_{n-i}|^2)  &(-1)^{n-i-1}\bar{\alpha}^{n-i-1}(\bar{s}_{n-i-1}-|\alpha|^{2(i+1)}s_{i+1}\bar{s}_n)  \\(-1)^{n-i-1}\alpha^{n-i-1}(s_{n-i-1}-|\alpha|^{2(i+1)}\bar{s}_{i+1}s_n)  & a\end{smallmatrix}\Big).$$}  The positive definiteness of $S(1,n-i)$ is equivalent to $\det S(1,n-i)>0$ and $a>0$ which gives $$a^2+(|\alpha s_1|^2-|\alpha^{n-1}s_{n-1}|^2)a+\ldots+(|\alpha_i s_i|^2-|\alpha^{n-i}s_{n-i}|^2)a>|\bar{\alpha}^{n-i-1}(\bar{s}_{n-i-1}-|\alpha|^{2(i+1)}s_{i+1}\bar{s}_n) |^2.$$ In order to prove the required inequality, we need to consider two cases, namely $n$ is odd or $n$ is even.  As the proof of the other case is similar, we only consider the case when $n$ is odd,.

\noindent{\bf{Case-1:}} \{$n=(2m+1):$ \} For this case, we first prove the following inequality for $i=1,\ldots,m,$
\small{\begin{equation}\label{s_{s-i}}
|\alpha^{n-i}(s_{n-i}-|\alpha|^{2i}\bar{s}_{i}s_n)|^2 <
         \frac{a^2}{((i-1)!)^2}\prod _{k=1}^{i-1}(n-k)^2
\end{equation}} $$\text{and}$$
\small{\begin{equation}\label{s_{s-i-n-1}}
|\alpha^{i}(s_{i}-|\alpha|^{2(n-i)}\bar{s}_{n-i}s_n)|^2 <
        \frac{a^2}{(i!)^2}\prod _{k=1}^{i}(n-k)^2 
 ~{\rm{for}}~i=1,\ldots,(m+1).
 \end{equation}}
Let $m(i)=|\alpha^is_i|^2-|\alpha^{n-i}s_{n-i}|^2.$ To prove the above required inequality, we need to show the following inequalities: \small{\[
m(i) \leq \left\{ \begin{array}{ll}
         \frac{n(n-2i)a}{((i)!)^2}\prod _{k=1}^{i-1}(n-k)^2 & \mbox{if $n>2i $};\\
        \frac{n(2i-n)a}{((i)!)^2}\prod _{k=1}^{i-1}(n-k)^2 & \mbox{if $ n<2i$}.\end{array} \right.
\]}

To show this, we apply Proposition \ref{gamma_n}, which implies that there exists $(c_i,\ldots,c_{n-1})\in \Gamma_{n-1}$ such that \small{$$\alpha^is_i=c_i+\bar{c}_{n-i}(\alpha^ns_n), \alpha^{n-i}s_{n-i}=c_{n-i}+\bar{c}_{i}(\alpha^ns_n)~{\rm for~} i=1,\ldots, n-1.$$} Therefore, for $n>2i$ we see that
$$m(i)=|\alpha^is_i|^2-|\alpha^{n-i}s_{n-i}|^2= |c_i+\bar{c}_{n-i}(\alpha^ns_n)|^2-|c_{n-i}+\bar{c}_{i}(\alpha^ns_n)|^2=(|c_i|^2-|c_{n-i}|^2)a,$$ which, by Theorem \ref{main@}, gives
\beq\label{alpha}
m(i) &\leq & \Big(\binom{n-1}{i}^2-\binom{n-1}{n-i}^2\Big)a =\frac{n(n-2i)a}{(i!)^2}\prod _{k=1}^{i-1}(n-k)^2.
\eeq
Similarly, for $n<2i$ we also prove that
\small{\beq\label{beta}m(n-i) \leq \frac{n(2i-n)a}{(i!)^2}\prod _{k=1}^{i-1}(n-k)^2.\eeq}
We  prove the  inequality \eqref{s_{s-i}} by mathematical induction. 
For $i=1,$ the result follows from the positivity of $S(1,n)>0.$ 
Assuming the truth of inequality \eqref{s_{s-i}} for a fixed value of $i,$ we wish to show that it is true with $i$ replaced by $i+1.$ Let $d_{i}$ be the $i$ th diagonal entries of $L.$ Also we observe that
$$d_{i}=d_{n-i+1}= a+(|\alpha s_1|^2-|\alpha^{n-1}s_{n-1}|^2)+\ldots+(|\alpha_{i-1} s_{i-1}|^2-|\alpha^{n-i+1}s_{n-i+1}|^2)$$ $$\text{and}$$
$$d_{m+1}= a+(|\alpha s_1|^2-|\alpha^{n-1}s_{n-1}|^2)+\ldots+(|\alpha_{m} s_{m}|^2-|\alpha^{m+1}s_{m+1}|^2).$$
Thus, by observing above fact and the positive definiteness of $S(1,n-i)>0,$ we conclude that
\small{\begin{align*}
|\alpha^{n-i-1}(s_{n-i-1}-|\alpha|^{2(i+1)}\bar{s}_{i+1}s_n)|^2 &<\frac{a^2}{((i-1)!)^2}\prod _{k=1}^{i-1}(n-k)^2+(|\alpha^is_i|^2-|\alpha^{n-i}s_{n-i}|^2)a\\&\leq \frac{(i^2+n^2-2in)a^2}{(i!)^2}\prod _{k=1}^{i-1}(n-k)^2=\frac{a^2}{(i!)^2}\prod _{k=1}^{i}(n-k)^2
\end{align*}}
This completes the induction step.

We now show the  inequality \eqref{s_{s-i-n-1}}. We prove it by induction on $m.$ We can write the matrix $L$ as
{\scriptsize\begin{align*}L=\left[
\begin{smallmatrix}
\begin{array}{cc|cc}
a   & \ldots &(-1)^{n-2}\bar{\alpha}^{n-2}(\bar{s}_{n-2}-|\alpha|^4s_2\bar{s}_n) & c \\
\vdots& \vdots  & \vdots &\vdots  \\\hline -\alpha(s_1-|\alpha^{n-1}|^2\bar{s}_{n-1}s_n)   & \ldots & a+(|\alpha s_1|^2-|\alpha^{n-1}s_{n-1}|^2)&-\bar{\alpha}(\bar{s}_{1}-|\alpha|^{2n-2}s_{n-1}\bar{s}_n) \\ \bar{c}  &\ldots &-\alpha(s_{1}-|\alpha|^{2n-2}\bar{s}_{n-1}s_n)& a
\end{array}
\end{smallmatrix}
\right].
\end{align*}\normalsize}
Since $L$ is positive definite, then the $2\times 2$ principal  minor $D_{2,2}$ which is obtained from $L$ is positive definite, where \small{$$D_{2,2}=\left[\begin{smallmatrix} a+(|\alpha s_1|^2-|\alpha^{n-1}s_{n-1}|^2)  &-\bar{\alpha}(\bar{s}_{1}-|\alpha|^{2n-2}s_{n-1}\bar{s}_n) \\-\alpha(s_{1}-|\alpha|^{2n-2}\bar{s}_{n-1}s_n) & a
\end{smallmatrix}\right].$$} The positive definiteness of $D_{2,2}$ is equivalent to $|\alpha(s_{1}-|\alpha|^{2n-2}\bar{s}_{n-1}s_n)|^2< (n-1)^2a^2,$ which shows that the result is true for $i=1.$
Assume that the inequality \eqref{s_{s-i-n-1}} is true for fixed values of $i.$ Now, we wish to show that it is true with $i$ replaced by $i+1.$ Let $$D_{i+2,i+2}=\left( \begin{smallmatrix}
a+(|\alpha s_1|^2-|\alpha^{n-1}s_{n-1}|^2)+\ldots+(|\alpha_{i+1} s_{i+1}|^2-|\alpha^{n-i-1}s_{n-i-1}|^2)  & \ldots&(-1)^{i+1}\bar{\alpha}^{i+1}(\bar{s}_{i+1}-|\alpha|^{2(n-i-1)}s_{n-i-1}\bar{s}_n) \\ \vdots  & \vdots &\vdots  \\(-1)^{i+1}\alpha^{i+1}(s_{i+1}-|\alpha|^{2(n-i-1)}\bar{s}_{n-i-1}s_n) &\ldots & a
\end{smallmatrix}\right)$$ $$\text{and}$$ $$S(1,i+2)=\left( \begin{smallmatrix}
a+(|\alpha s_1|^2-|\alpha^{n-1}s_{n-1}|^2)+\ldots+(|\alpha_{i+1} s_{i+1}|^2-|\alpha^{n-i-1}s_{n-i-1}|^2)  &(-1)^{i+1}\bar{\alpha}^{i+1}(\bar{s}_{i+1}-|\alpha|^{2(n-i-1)}s_{n-i-1}\bar{s}_n)  \\(-1)^{i+1}\alpha^{i+1}(s_{i+1}-|\alpha|^{2(n-i-1)}\bar{s}_{n-i-1}s_n)  & a
\end{smallmatrix}\right).$$ The positive definiteness of $L$ implies that $D_{i+2,i+2}$ and $S(1,i+2)$ are positive definite. Therefore, we have  \scriptsize\begin{align*}|\alpha^{i+1}(s_{i+1}-|\alpha|^{2(n-i-1)}\bar{s}_{n-i-1}s_n)|^2|& < \frac{a^2}{((i)!)^2}\prod _{k=1}^{i}(n-k)^2+(|\alpha^{i+1}s_{i+1}|^2-|\alpha^{n-i-1}s_{n-i-1}|^2)a< \frac{a^2}{((i+1)!)^2}\prod _{k=1}^{i+1}(n-k)^2.\end{align*}\normalsize This completes the induction step.
Our next goal is to show that for  $ k(i)=\binom{n-1}{i}+\binom{n-1}{n-i}$
\small{\begin{equation}\label{phii1}k(i)^2a-m(n-i)> 2k(i)|\alpha^{i}(s_{i}-|\alpha|^{2(n-i)}\bar{s}_{n-i}s_n)|~\mbox{for} ~i=1,\ldots,m\end{equation} $$\text{ and}$$ \begin{equation}\label{phi2}k(i)^2a-m(i)> 2k(i)|\alpha^{n-i}(s_{n-i}-|\alpha|^{2i}\bar{s}_{i}s_n)|~\mbox{for} ~i=1,\ldots,m.\end{equation}}
Note that 
\begin{equation}\label{k(i)} k(i)= \binom{n-1}{i}+\binom{n-1}{n-i}= \frac{1}{i!}\prod _{k=0}^{i-1}(n-k).\end{equation}  
If $m(i)$ and $m(n-i)$ are negative, then the inequalities in \eqref{phii1} and \eqref{phi2} are obvious. Therefore, we assume that 
$m(i)\geq 0$ and $m(n-i)\geq 0.$ By using  \eqref{s_{s-i}}, \eqref{s_{s-i-n-1}},\eqref{alpha} and  \eqref{beta}, we see that for $ k(i)= \binom{n-1}{i}+\binom{n-1}{n-i}$
\small{\begin{equation}\label{mni} k(i)^2a-m(n-i)>k(i)^2a-\frac{n(n-1)^2\ldots (n-i+1)^2(n-2i)a}{(i!)^2}> 2k(i)|\alpha^{n-i}(s_{n-i}-|\alpha|^{2i}\bar{s}_{i}s_n)|~{\rm{for}}~ i=1,\ldots,m\end{equation}}  $$\text{and}$$ \small{\begin{equation}\label{mi}k(i)^2a-m(i)> k(i)^2a-\frac{n(n-1)^2\ldots (n-i+1)^2(2i-n)a}{(i!)^2}> 2k(i)|\alpha^{i}(s_{i}-|\alpha|^{2(n-i)}\bar{s}_{n-i}s_n)|~{\rm{ for}}~ i=1,\ldots,m. \end{equation}} 
Therefore, for all $\omega\in \mathbb T$ and $i=1,\ldots,m,$ from \eqref{mi} we have
\small{\beq\label{omega} k(i)^2a-m(i)\nonumber & >& 2k(i)\Re\omega\alpha^{n-i}(s_{n-i}-|\alpha|^{2i}\bar{s}_{i}s_n)\\&=& k(i)\omega\alpha^{n-i}(s_{n-i}-|\alpha|^{2i}\bar{s}_{i}s_n)+k(i)\bar{\omega}\bar{\alpha}^{n-i}(\bar{s}_{n-i}-|\alpha|^{2i}s_{i}\bar{s}_n).\eeq}  
Choosing $\omega=1$ and substituting the value of $a,m(i)$ in  \eqref{omega}, for $ i=1,\ldots,m,$ we get
\small{\begin{align*}
\Phi^{(i)}_{2}(\alpha^{i}s_i,\alpha^{n-i}s_{n-i},\alpha^{n}s_n)&=k(i)^2(1-|\alpha^ns_n|^2)+(|\alpha^{n-i}s_{n-i}|^2-|\alpha^is_i|^2)\\&-k(i)\alpha^{n-i}(s_{n-i}-|\alpha|^{2i}\bar{s}_{i}s_n)-k(i)\bar{\alpha}^{n-i}(\bar{s}_{n-i}-|\alpha|^{2i}s_{i}\bar{s}_n)\\&>0.
\end{align*}} 
For $i=1,\ldots,m,$ by continuity, we deduce that $\Phi^{(i)}_{2}(\alpha^{i}s_i,\alpha^{n-i}s_{n-i},\alpha^{n}s_n)\geq 0$ for  all $\alpha\in \overline{\mathbb D}.$ Similarly, for $i=1,\ldots,m,$ we can prove that $\Phi^{(i)}_{1}(\alpha^{n-i}s_{n-i},\alpha^{i}s_{i},\alpha^{n}s_n)\geq 0,$ for  all $\alpha\in \overline{\mathbb D}.$ This completes the proof.

\end{proof}
\begin{rem}
For $ i=1,\ldots,(n-1),$ the positivity of two functions $\Phi^{(i)}_{1}$ and $\Phi^{(i)}_{2}$ are equivalent to the conditions \small{$|k(i)\alpha^{n}s_n-\alpha^{n-i}s_{n-i}|\leq |k(i)-\alpha^{i}s_i|$} and \small{$ |k(i)\alpha^{n}s_n-\alpha^{i}s_{i}|\leq |k(i)-\alpha^{n-i}s_{n-i}|$} with $ k(i)= \binom{n-1}{i}+\binom{n-1}{n-i}.$  One can easily check that the positivity of two functions $\Phi^{(i)}_{1}$ and $\Phi^{(i)}_{2}$ imply that for $i=1,\ldots,(n-1)$ \small{\begin{equation}\label{bgamma_n}|s_i-\bar{s}_{n-i}s_n|+|s_{n-i}-\bar{s}_is_n|\leq \Big(\binom{n-1}{i}+\binom{n-1}{n-i}\Big)(1-|s_n|^2).\end{equation}} If we further assume that $|s_n|=1$
and $(\gamma_1s_1,\ldots, \gamma_{n-1}s_{n-1})\in \Gamma_{n-1},$ then from \eqref{bgamma_n},  it follows immediately that $(s_1, \ldots,s_n)\in b\Gamma_n,$ where $\gamma_i=\frac{n-i}{n}.$ For $|s_n|<1,$ however it is not clear whether there exists $(c_1,\ldots,c_{n-1})\in \Gamma_{n-1}$ such that $s_i=c_i+\bar{c}_{n-i}s_n, s_{n-i}=c_{n-i}+\bar{c}_{i}s_n$ for $ i=1,\ldots,(n-1).$
\end{rem}

\section{$\Gamma_n$-contractions and their fundamental operators}
Let $(S_1,\ldots,S_n)$ be the commuting $n$-tuples of bounded operators defined on a Hilbert space $\mathcal H$ with $\|S_n\|\leq 1.$ Given a contraction $S_n,$ denote $D_{S_n}=(I-S_n^*S_n)^\frac{1}{2}$ and  $\mathcal D_{S_n}=\overline{\Ran} D_{S_n}$. In this section, we want to find the solutions $X_i$ and $X_{n-i}$ of the fundamental equations which are defined in the following:
\small{$$S_i-S_{n-i}^*S_n=D_{S_n}X_iD_{S_n}~~{\rm{and}}~~S_{n-i}-S_{i}^*S_n=D_{S_n}X_{n-i}D_{S_n}, X_i,X_{n-i}\in \mathcal B(\mathcal D_{S_n}), ~{\rm{for ~all~}} i=1,\ldots,(n-1).$$}
The following proposition gives an estimate for the norm of $S_i,$ for $i=1,\ldots,(n-1).$ As the proposition follows immediately from  the definition of $\Gamma_n$-contraction, we skip the proof.
\begin{prop}
If $(S_1,\ldots,S_n)$ is a $\Gamma_n$-contraction, then $\|S_i\|\leq \binom{n-1}{i}+\binom{n-1}{n-i}$ and $\|S_n\|\leq 1.$
\end{prop}
Under some suitable condition the following proposition says that the operator pencils $\Phi_{1}^{(i)}$ and $\Phi_{2}^{(i)}$ are positive.
\begin{prop}\label{k(i)}
For $n \geq 2,$ let $(S_1,\ldots,S_n)$ be a $\Gamma_n$-contraction. Then for all $ i=1,\ldots,(n-1)$ \small{$$\Phi^{(i)}_{1}(\alpha^{i}S_i,\alpha^{n-i}S_{n-i},\alpha^{n}S_n)\geq 0~{\rm{ and }}~~\Phi^{(i)}_{2}(\alpha^{i}S_{i},\alpha^{n-i}S_{n-i},\alpha^{n}S_n)\geq 0$$} for all $\alpha\in \overline{\mathbb D}$ with $ k(i)= \binom{n-1}{i}+\binom{n-1}{n-i}.$ 
\end{prop}
\begin{proof}
Since $(S_1,\ldots,S_n)$ is a $\Gamma_n$-contraction, $\sigma_{T}(S_1,\ldots,S_{n-1},S_n)\subseteq \Gamma_{n}.$ Let $f$ be a holomorphic function in a neighbourhood of $\Gamma_n.$ Since $\Gamma_n$ is a polynomially convex, by Oka-Weil theorem \cite[Theorem 5.1]{Gamelin}, there exists a sequence of polynomials $\{p_n\}$ that converges uniformly to $f$ on $\Gamma_n.$ 
Thus, we have \small{$$\|f(S_1,\ldots,S_n)\|=\lim_{n\rightarrow \infty}\|p_n(S_1,\ldots,S_n)\|\leq \lim_{n\rightarrow \infty}\|p_n\|_{\infty,\Gamma_n}=\|f\|.$$}
For fix $\alpha \in \mathbb D,$  let \small{$f_i(s_1,\ldots,s_n)=\frac{k(i)\alpha^ns_n-\alpha^{n-i}s_{n-i}}{k(i)-\alpha^is_i}~{\rm{for~ all}}~ i=1,\ldots,(n-1).$}  Since $ k(i)= \binom{n-1}{i}+\binom{n-1}{n-i},$ $f$ is well defined holomorphic function in a neighbourhood of $\Gamma_n.$  Therefore, we note that \small{$$\|(k(i)\alpha^ns_n-\alpha^{n-i}S_{n-i})(k(i)-\alpha^iS_i)^{-1}\|\leq \|f\|_{\infty,\Gamma_n}\leq1$$}
which gives
\small{$$(k(i)-\alpha^iS_i)^*(k(i)-\alpha^iS_i)\geq (k(i)\alpha^ns_n-\alpha^{n-i}S_{n-i})^*(k(i)\alpha^ns_n-\alpha^{n-i}S_{n-i}).$$} So by definition of $\Phi^{(i)}_{1},$ we get that $\Phi^{(i)}_{1}(\alpha^{i}S_i,\alpha^{n-i}S_{n-i},\alpha^{n}S_n)\geq 0$ for all $\alpha\in \mathbb D.$ By continuity we have  $\Phi^{(i)}_{1}(\alpha^{i}S_i,\alpha^{n-i}S_{n-i},\alpha^{n}S_n)\geq 0$ for all $\alpha\in \bar{\mathbb D}.$ Similarly, one can also prove that $\Phi^{(i)}_{2}(\alpha^{i}S_{i},\alpha^{n-i}S_{n-i},\alpha^{n}S_n)\geq 0$ for all $\alpha\in \overline{\mathbb D}.$ This completes the proof.
\end{proof}
Recall that the numerical radius of a bounded operator $A$ on a Hilbert space $\mathcal H$ is defined to be
$$\omega(A)=\sup\{|\langle Ax,x\rangle|:\|x\|=1\}.$$
It is  well known that \small{$r(A)\leq \omega(A)\leq \|A\|,$} where $r(A)$ is the spectral radius of $A.$ Also if the numerical radius of a bounded operator $A$ is not greater than $n$ then $\Re\alpha A\leq nI$ for all complex number $\alpha$ with $|\alpha|=1$ and vice-versa. In \cite{pal3}, it is shown that if two bounded operators $A_1,A_2$ with $\omega(A_1+zA_2)\leq n$ for all $z\in \mathbb T,$ then $\omega(A_1+zA_2^*)\leq n$ and  $\omega(A_1^*+zA_2)\leq n$ for all $z\in \mathbb T.$ We start with pivotal theorem which says that if $(S_1,\ldots,S_n)$ is a $\Gamma_n$-contraction then the fundamental equations have unique solutions. The proof of the theorem needs the following lemma. We skip the proof as it is easy to check.
\begin{lem}
Let $A_1,\ldots,A_{n-1}$ be  bounded operators on a Hilbert space $\mathcal H.$ Then the following are equivalent:
\small{\begin{enumerate}
\item $[A_i,A_j]=0$ and $[A_i,A_{n-j}^*]=[A_j,A_{n-i}^*]$ for $1\leq i,j \leq n-1,$ where $[P,Q]=PQ-QP$ for two operators $P,Q;$

\item $A_i^*+A_{n-i}z$ and $A_j^*+A_{n-j}z$ commute for every $z$ with $|z|=1$ and for $1\leq i,j \leq n-1;$

\item $A_i^*+A_{n-i}z$ is a normal operator for every $z$ with $|z|=1$ and for $1\leq i \leq n-1;$

\item $A_{n-i}^*+A_{i}z$ is a normal operator for every $z$ with $|z|=1$ and for $1\leq i \leq n-1.$
\end{enumerate}}
\end{lem}
%
%
Now here is the theorem which gives the existence and uniqueness of solution to the fundamental equations of a $\Gamma_n$-contraction.

\small{\begin{thm}\label{uniqueness}({\bf{Existence and Uniqueness}})
For $n \geq 2,$ let $(S_1,\ldots,S_n)$ be a $\Gamma_n$-contraction on a Hilbert space $\mathcal H.$ Then there are unique operators $E_1,\ldots,E_{n-1} \in \mathcal B(\mathcal D_{S_n})$ such that $S_i-S_{n-i}^*S_n=D_{S_n}E_iD_{S_n}$ and $S_{n-i}-S_{i}^*S_n=D_{S_n}E_{n-i}D_{S_n}, E_i,E_{n-i}\in \mathcal B(\mathcal D_{S_n}),$ for $i=1,\ldots,(n-1).$ Moreover, $\omega(E_i+E_{n-i}z) \leq  \binom{n-1}{i}+\binom{n-1}{n-i}$ for all  $z\in \mathbb T.$
\end{thm}}
\begin{proof}
For $ i=1,\ldots,(n-1),$ applying Proposition \ref{k(i)},  we have
\small{\begin{equation}\label{phi1}
\Phi^{(i)}_{1}(\alpha^{i}S_i,\alpha^{n-i}S_{n-i},\alpha^{n}S_n)\geq 0 ~{\rm{and}}~~
\Phi^{(i)}_{2}(\beta^{i}S_{i},\beta^{n-i}S_{n-i},\beta^{n}S_n)\geq 0
\end{equation}}
for all $\alpha\in \overline{\mathbb D}$ with $ k(i)= \binom{n-1}{i}+\binom{n-1}{n-i}.$ 
In particular, for all $\alpha,\beta\in \mathbb T$ and $i=1,\ldots,(n-1)$ , in view of \eqref{phi1}, this implies that  
\small{\begin{equation}\label{Pphi1}k(i)^2D_{S_n}^2+(S_i^*S_i-S_{n-i}^*S_{n-i})-k(i)\alpha ^i(S_i-S_{n-i}^*S_n)-k(i)\bar{\alpha}^i(S_i^*-S_n^*S_{n-i})\geq 0\end{equation}} and
\small{\begin{equation} \label{Pphi2}k(i)^2D_{S_n}^2+(S_{n-i}^*S_{n-i}-S_{i}^*S_{i})-k(i)\beta ^{n-i}(S_{n-i}-S_{i}^*S_n)-k(i)\bar{\beta}^{n-i}(S_{n-i}^*-S_n^*S_{i})\geq 0\end{equation}}  respectively.
For all $i=1,\ldots,(n-1),$ set
\small{\begin{align}\label{eta1}
\eta^{(i)}(\alpha)&=2k(i)D_{S_n}^2- \alpha ^i(S_i-S_{n-i}^*S_n)-\bar{\alpha}^i(S_i^*-S_n^*S_{n-i})-\alpha ^{2i}(S_{n-i}-S_{i}^*S_n)-\bar{\alpha}^{2i}(S_{n-i}^*-S_n^*S_{i}).
\end{align}}
Choose $\beta$ such that $\beta^{n-i}=\alpha^{2i}.$ Hence we deduce from \eqref{Pphi1}, \eqref{Pphi2} and \eqref{eta1} that $\eta^{(i)}(\alpha)\geq 0,$ for all $\alpha \in \mathbb T$ and  $i=1,\ldots,(n-1).$ Therefore, for all $i=1,\ldots,(n-1)$ and $\alpha \in \mathbb T,$ by operator Fejer-Reisz Theorem \cite[Theorem 1.2]{DriRon}, there is a polynomial of degree $2i$ say $P^{(i)}(\alpha)=X_{0}^{(i)}+X_{1}^{(i)}\alpha^i+X_{2}^{(i)}\alpha^{2i}$ such that   
\small{\begin{align}\label{eta2}
\eta^{(i)}(\alpha)=P^{(i)}(\alpha)^*P^{(i)}(\alpha)\nonumber&=(X_{0}^{(i)}+X_{1}^{(i)}\alpha^i+X_{2}^{(i)}\alpha^{2i})^*(X_{0}^{(i)}+X_{1}^{(i)}\alpha^i+X_{2}^{(i)}\alpha^{2i})\nonumber\\&=(X_{0}^{(i)*}X_{0}^{(i)}+X_{1}^{(i)*}X_{1}^{(i)}+X_{2}^{(i)*}X_{2}^{(i)})+(X_{0}^{(i)*}X_{1}^{(i)}+X_{1}^{(i)*}X_{2}^{(i)})\alpha^i\nonumber \\&+(X_{0}^{(i)*}X_{1}^{(i)}+X_{1}^{(i)*}X_{2}^{(i)})^*\bar{\alpha}^i+X_{0}^{(i)*}X_{2}^{(i)}\alpha^{2i}+X_{2}^{(i)*}X_{0}^{(i)}\bar{\alpha}^{2i}.
\end{align}}
Comparing  \eqref{eta1} and  \eqref{eta2}, for all $i=1,\ldots,(n-1),$  it follows that
\small{\begin{equation}\label{D_{S_n}}
2k(i)D_{S_n}^2=X_{0}^{(i)*}X_{0}^{(i)}+X_{1}^{(i)*}X_{1}^{(i)}+X_{2}^{(i)*}X_{2}^{(i)}
\end{equation}
\begin{equation}\label{S_i}
S_{i}-S_{n-i}^*S_n=-X_{0}^{(i)*}X_{1}^{(i)}+X_{1}^{(i)*}X_{2}^{(i)}
,S_{n-i}-S_{i}^*S_n=-X_{0}^{(i)*}X_{2}^{(i)}
\end{equation}}
From  \eqref{D_{S_n}}, we have
\small{$$2k(i)D_{S_n}^2\geq X_{0}^{(i)*}X_{0}^{(i)}, 2k(i)D_{S_n}^2\geq X_{1}^{(i)*}X_{1}^{(i)}~~{\rm{and}}~~2k(i)D_{S_n}^2\geq X_{2}^{(i)*}X_{2}^{(i)}~~{\rm{ for~~ all}}~~ i=1,\ldots,(n-1).$$} which, by Douglas's lemma \cite[Lemma 2.1]{dmp}, implies that there are contractions $Z_{0}^{(i)},Z_{1}^{(i)},Z_{2}^{(i)}$ such that \small{$$X_{0}^{(i)*}=\sqrt{2k(i)}D_{S_n}Z_{0}^{(i)},X_{1}^{(i)*}=\sqrt{2k(i)}D_{S_n}Z_{1}^{(i)}~~{\rm{and}}~~X_{2}^{(i)*}=\sqrt{2k(i)}D_{S_n}Z_{2}^{(i)}~~{\rm{ for~~ all }}~~i=1,\ldots,(n-1).$$}
For all $i=1,\ldots,(n-1),$ putting the value of $X_{0}^{(i)*}, X_{1}^{(i)*}$ and $X_{2}^{(i)*}$ in \eqref{S_i}, we deduce that \small{$$S_{i}-S_{n-i}^*S_n=D_{S_n}[-2k(i)(Z_{0}^{(i)}Z_{1}^{(i)*}+Z_{1}^{(i)}Z_{2}^{(i)*})]D_{S_n}~~{\rm{ and}} ~~S_{n-i}-S_{i}^*S_n=D_{S_n}(-2k(i)Z_{0}^{(i)}Z_{2}^{(i)*})D_{S_n}.$$} For all $i=1,\ldots,(n-1),$ let $$E_i=P_{\mathcal D_{S_n}}[-2k(i)(Z_{0}^{(i)}Z_{1}^{(i)*}+Z_{1}^{(i)}Z_{2}^{(i)*})]\mid_{\mathcal D_{S_n}}, E_{n-i}=P_{\mathcal D_{S_n}}(-2k(i)(Z_{0}^{(i)}Z_{2}^{(i)*})\mid_{\mathcal D_{S_n}}.$$ Then from above it is clear that $E_i,E_{n-i}$ are the solutions to the equations $S_{i}-S_{n-i}^*S_n=D_{S_n}X_{i}D_{S_n}$ and $S_{n-i}-S_{i}^*S_n=D_{S_n}X_{n-i}D_{S_n}$ respectively, for all $i=1,\ldots,(n-1).$

\noindent{\bf{Uniqueness:}} Let $E_i,G_i$ be two solutions of the equation $S_i-S_{n-i}^*S_n=D_{S_n}X_iD_{S_n}$ for all $i=1,\ldots,(n-1).$ Then $D_{S_n}(E_i-G_i)D_{S_n}=0,$ which implies that $E_i=G_i$ on $\mathcal D_{S_n}$ for all $i=1,\ldots,(n-1).$

Also, for all $i=1,\ldots,(n-1),$ we have that \begin{align}\label{addinginequality}
2k(i)D_{S_n}^2\nonumber&\geq 2\Re\alpha ^i\{(S_i-S_{n-i}^*S_n)+\alpha ^{i}(S_{n-i}-S_{i}^*S_n)\}\\&=2\Re\alpha^i[D_{S_n}E_iD_{S_n}+\alpha ^{i}D_{S_n}E_{n-i}D_{S_n}]=\Re\alpha^iD_{S_n}F_i(\alpha)D_{S_n},
\end{align} where $F_i(\alpha)=\frac{1}{k(i)}(E_i+\alpha ^{i}E_{n-i}).$  This implies that $I_{\mathcal D_{S_n}}-\Re\alpha^iF_i(\alpha)\geq 0,$ because $F_i(\alpha)$ is defined on $\mathcal D_{S_n}$ for all $i=1,\ldots,(n-1),$ which gives  $\omega(F_i(\alpha))\leq 1.$ Thus, we have $\omega(E_i+zE_{n-i})\leq \binom{n-1}{i}+\binom{n-1}{n-i}$ for all $z \in \mathbb T$ and $i=1,\ldots,(n-1).$ This completes the proof.
\end{proof}
\begin{rem}
From Theorem \ref{uniqueness}, it is clear that $S_i-S_{n-i}^*S_n$ is equal to zero on the orthogonal complement of $\mathcal D_{S_n}$ in $\mathcal H$ for all $i=1,\ldots,(n-1).$ Also for the case of $\Gamma_n$-isometry or $\Gamma_n$-unitary of $(S_1,\ldots,S_n),$ the $(n-1)$-tuple of fundamental operators  $E_i$ for $i=1,\ldots,(n-1),$ are zero, because for this case $\mathcal D_{S_n}=\{0\}.$ 
\end{rem}
The following proposition says that two fundamental operators are invariant under unitary equivalence. We skip the proof because it is easy to verify. 
\begin{prop}
If two $\Gamma_n$-contractions are unitarily equivalent then  their fundamental operators are also unitarily equivalent. 
\end{prop}

\section{$\Gamma_n$-Unitaries and $\Gamma_n$-Isometries}
We recall that a $\Gamma_n$-unitary is a commuting $n$-tuple of normal operators $(S_1,\ldots, S_n)$
whose Taylor joint spectrum contained in the distinguished boundary of $\Gamma_n$ and a
$\Gamma_n$-isometry is the restriction of a $\Gamma_n$-unitary to a joint invariant subspace
of $S_1, \ldots  S_n.$   In this section we would like to discuss several properties of $\Gamma_n$-unitaries and $\Gamma_n$-isometries. In order to describe several properties of $\Gamma_n$-unitaries and $\Gamma_n$-isometries, we need a known fact from \cite{BPR}, which says that if $T$ is a bounded operator on a Hilbert space $\mathcal H$ with $\Re
\beta T \leq 0$ for all  $\beta\in \mathbb T,$ then $T = 0.$ Parts of the following theorem tells the new characterization of $\Gamma_n$-unitary which were obtained  in \cite[Theorem-4.2]{SS}. Parts $(4)$ and $(5)$ are new. This theorem plays an important role for proving the conditional dilation on $\Gamma_n$ which we will see later.
\begin{thm}\label{Gamma_n unitary}
Let $(S_1, \ldots, S_n)$ be a commuting $n$-tuple of operator defined on a Hilbert space $\mathcal H.$ Then the following are equivalent:
\begin{enumerate}
\item $(S_1,\ldots,S_n)$ is a $\Gamma_n$-unitary;

\item There exists commuting unitary operators $U_1,\ldots,U_n$ on $\mathcal H$ such that $S_i=\sum_{1\leq k_1< \ldots< k_i\leq n}U_{k_1}\ldots U_{k_i}$ for $i=1,\ldots,n-1;$ 

\item $S_n$ is unitary, $(\gamma_1S_1,\ldots,\gamma_{n-1}S_{n-1})\in \Gamma_{n-1}$ and $S_i=S_{n-i}^*S_n$ for $i=1,\ldots,n-1;$
\item  $(S_1,\ldots, S_n)$ is a $\Gamma_n$-contraction and $S_n$ is an unitary ;

\item  $S_n$ is unitary and there exist $\Gamma_{n-1}$-unitary $( C_1,\dots,C_{n-1})$ on $\mathcal H$  such that $C_1,\ldots,C_{n-1}, S_n$ commute and
$S_i = C_i + C_{n-i}^*S_n , S_{n-i} = C_{n-i} + C_i^*S_n$ for $i = 1,\ldots, n-1.$
\end{enumerate}
\end{thm}
\begin{proof}
We will prove that $(1)\Leftrightarrow (2)\Leftrightarrow(3),(2)\Rightarrow(4)\Rightarrow(3)$ and $(2)\Leftrightarrow (5).$  

As the equivalence of  $(1),(2)$ and $(3)$ is due to \cite[Theorem 4.2]{SS}, we skip the proof. It is also easy to verify $(2)$ implies $(4).$

We will now show that $(4)$ implies $(3).$ To show $(4)$ implies $(3),$ let $(S_1,\ldots,S_n)$ be a $\Gamma_n$-contraction and $S_n$ be a unitary. Then by Proposition \ref{k(i)}, we see that for $ i=1,\ldots,(n-1)$ and for all $\beta\in \mathbb T$ with $ k(i)=\binom{n-1}{i}+\binom{n-1}{n-i}$  \small{$$\Phi^{(i)}_{1}(\beta^{i}S_i,\beta^{n-i}S_{n-i},\beta^{n}S_n)\geq 0~~{\rm and}~~ \Phi^{(i)}_{2}(\beta^{i}S_{i},\beta^{n-i}S_{n-i},\beta^{n}S_n)\geq 0$$} which implies 
 \small{\begin{equation}\label{Pphi11}k(i)^2(I-S_n^*S_n)+(S_i^*S_i-S_{n-i}^*S_{n-i})-k(i)\beta ^i(S_i-S_{n-i}^*S_n)-k(i)\bar{\beta}^i(S_i^*-S_n^*S_{n-i})\geq 0\end{equation}} $$\textit{\rm and}$$
\small{\begin{equation} \label{Pphi22}k(i)^2(I-S_n^*S_n)+(S_{n-i}^*S_{n-i}-S_{i}^*S_{i})-k(i)\beta ^{n-i}(S_{n-i}-S_{i}^*S_n)-k(i)\bar{\beta}^{n-i}(S_{n-i}^*-S_n^*S_{i})\geq 0\end{equation}}  Since $S_n$ is unitary, from  \eqref{Pphi11} and  \eqref{Pphi22}  it follows that for $ i=1,\ldots,(n-1)$ and for all $\beta\in \mathbb T$ with $ k(i)=\binom{n-1}{i}+\binom{n-1}{n-i}$ 
\small{\begin{equation}\label{Pphi111}(S_i^*S_i-S_{n-i}^*S_{n-i})-k(i)\beta ^i(S_i-S_{n-i}^*S_n)-k(i)\bar{\beta}^i(S_i^*-S_n^*S_{n-i})\geq 0\end{equation}} $$\textit{\rm and}$$
\small{\begin{equation} \label{Pphi222}(S_{n-i}^*S_{n-i}-S_{i}^*S_{i})-k(i)\beta ^{n-i}(S_{n-i}-S_{i}^*S_n)-k(i)\bar{\beta}^{n-i}(S_{n-i}^*-S_n^*S_{i})\geq 0.\end{equation}}  Putting $\beta^i=1$ and $\beta^i=-1$ respectively in  \eqref{Pphi111} and adding them, we obtain for $ i=1,\ldots,n-1,$
\small{\begin{equation}\label{PS_i}S_i^*S_i-S_{n-i}^*S_{n-i}\geq 0.\end{equation}}  Again, putting $\beta^{n-i}=1$ and $\beta^{n-i}=-1$ respectively in  \eqref{Pphi222} and adding them, we have 
\small{\begin{equation}\label{PS_n-i}S_i^*S_i-S_{n-i}^*S_{n-i}\leq 0~~{\rm for }~~i=1,\ldots,(n-1).\end{equation}} Thus,  \eqref{PS_i} and  \eqref{PS_n-i} together imply that  $S_i^*S_i=S_{n-i}^*S_{n-i}$ for $ i=1,\ldots,(n-1).$ From  \eqref{Pphi111} we deduce that $\Re\beta^i(S_i-S_{n-i}^*S_n)\leq 0$ for all $\beta\in \mathbb T$ and for  $ i=1,\ldots,(n-1),$ which, by above fact, gives $S_i=S_{n-i}^*S_n.$  Since $(S_1,\ldots,S_n)$ is a $\Gamma_n$-contraction, it is easy to see that  $(\gamma_1S_1,\ldots,\gamma_{n-1}S_{n-1})$ is also a $\Gamma_{n-1}$-contraction.

We will prove $(2)$ implies $(5).$ Suppose $(2)$ holds. Then, for $i=1,\ldots,(n-1),$
\small{\begin{align*}
S_i&=\sum_{1\leq k_1< \ldots< k_i\leq n-1}U_{k_1}\ldots U_{k_i}=C_i+C_{n-i}^*S_n,
\end{align*}} where $C_i=\sum_{1\leq l_1< \ldots< l_i\leq n-1}U_{l_1}\ldots U_{l_i}.$  Clearly, $( C_1,\dots,C_{n-1})$ are $\Gamma_{n-1}$-unitary.

Now we will show that $(5)$ implies $(2).$ Let $C_i=\sum_{1\leq l_1< \ldots< l_i\leq n-1}U_{l_1}\ldots U_{l_i}$ for some commuting unitary unitary $U_1,\ldots,U_{n-1}$ on $\mathcal H.$ We choose $U_n=C^{*}_{n-1}S_n.$ Clearly, $U_n$ is unitary on $\mathcal H.$ Thus, we have $$S_i=C_i+C_{n-i}^*C_{n-1}U_n=\sum_{1\leq k_1< \ldots< k_i\leq n-1}U_{k_1}\ldots U_{k_i}$$ and $S_n=\Pi_{i=1}^{n}U_{i}.$ Hence $(5)$ implies $(2).$ This completes the proof.

\end{proof}
Parts of the following theorem gives the new characterization of $\Gamma_n$-isometry which were obtained  in \cite[Theorem 4.12]{SS}. Parts $(4),(5)$ and $(6)$ are new. The proof of the following theorem works along the lines of \cite{BPR}.
\begin{thm}\label{Gamma_n isometry} Let $S_1, \ldots,S_{n}$ be commuting operators on a Hilbert space $\mathcal H.$
Then the following are equivalent:
\begin{enumerate}
\item $(S_1, \ldots, S_n)$ is a $\Gamma_n$-isometry ;

\item  $S_n$ is a isometry, $S_i = S_{n-i}^*
S_n$ and $(\gamma_1S_1,\ldots,\gamma_{n-1}S_{n-1})$ is $\Gamma_{n-1}$-contraction;

\item ( Wold-Decomposition ): there is an orthogonal decomposition $\mathcal H =
\mathcal H_1 \oplus \mathcal H_2$ into common invariant subspaces of $S_1,\ldots, S_{n}$  such
that $(S_1\mid \mathcal H_1 , \ldots , S_n\mid \mathcal H_1 )$ is a $\Gamma_n$-unitary and  $(S_1\mid \mathcal H_2 , \ldots , S_n\mid \mathcal H_2 )$ is a pure $\Gamma_n$-isometry ;

\item $(S_1, \ldots, S_n)$ is a $\Gamma_n$-contraction and $S_n$ is a isometry;

\item $(\gamma_1S_1,\ldots,\gamma_{n-1}S_{n-1})$ is a $\Gamma_{n-1}$-contraction and for all $\beta\in \mathbb T$ with $ k(i)=\binom{n-1}{i}+\binom{n-1}{n-i}$ $$\Phi^{(i)}_{1}(\beta^{i}S_i,\beta^{n-i}S_{n-i},\beta^{n}S_n)=0~~{\rm and}~~ \Phi^{(i)}_{2}(\beta^{i}S_{i},\beta^{n-i}S_{n-i},\beta^{n}S_n)= 0~{\rm for~} i=1,\ldots,(n-1) ;$$ 

Moreover, if  $r(S_i)< \binom{n-1}{i}+\binom{n-1}{n-i}$  for $i = 1,\ldots, (n-1),$  then all of the above are equivalent to :

\item   $(\gamma_1S_1,\ldots,\gamma_{n-1}S_{n-1})$ is a $\Gamma_{n-1}$-contraction and $(k(i)\beta^nS_n-S_{n-i})(k(i)I-\beta^iS_i)^{-1}$ and $(k(i)\beta^{n}S_n-S_{i})(k(i)I-\beta^{n-i}S_{n-i})^{-1}$ are isometry for all $\beta\in \mathbb T$ and for $i = 1,\ldots, (n-1)$ with $ k(i)= \binom{n-1}{i}+
\binom{n-1}{n-i}.$
\end{enumerate}
\end{thm}
\begin{proof}
As the equivalence of  $(1),(2)$ and $(3)$ is due to \cite[Theorem 4.12]{SS}, we skip the proof. It is also easy to verify $(1)$ implies $(4).$

Suppose $(4)$ holds. Then, by Proposition \ref{k(i)} for all $\beta\in \mathbb T$ with $ k(i)= \binom{n-1}{i}+\binom{n-1}{n-i}$ and for $ i=1,\ldots,(n-1),$ we have \small{$$\Phi^{(i)}_{1}(\beta^{i}S_i,\beta^{n-i}S_{n-i},\beta^{n}S_n)\geq 0~~{\rm and}~~\Phi^{(i)}_{2}(\beta^{i}S_{i},\beta^{n-i}S_{n-i},\beta^{n}S_n)\geq 0$$}  which, together with $S_n$ is isometry, implies that
\small{\begin{equation}\label{Pphi111k}(S_i^*S_i-S_{n-i}^*S_{n-i})-k(i)\beta ^i(S_i-S_{n-i}^*S_n)-k(i)\bar{\beta}^i(S_i^*-S_n^*S_{n-i})\geq 0\end{equation}} $$\textit{\rm and}$$
\small{\begin{equation} \label{Pphi222k}(S_{n-i}^*S_{n-i}-S_{i}^*S_{i})-k(i)\beta ^{n-i}(S_{n-i}-S_{i}^*S_n)-k(i)\bar{\beta}^{n-i}(S_{n-i}^*-S_n^*S_{i})\geq 0.\end{equation}} 
Using similar argument which is described in Theorem \ref{Gamma_n unitary}, one can easily verify that $S_i^*S_i=S_{n-i}^*S_{n-i},S_i=S_{n-i}^*S_n$ and $S_{n-i}=S_{i}^*S_n$ for $ i=1,\ldots,(n-1).$ 
This gives for all $\beta\in \mathbb T$ with $ k(i)= \binom{n-1}{i}+\binom{n-1}{n-i}$ 
\small{$$\Phi^{(i)}_{1}(\beta^{i}S_i,\beta^{n-i}S_{n-i},\beta^{n}S_n)=0~~{\rm and}~~ \Phi^{(i)}_{2}(\beta^{i}S_{i},\beta^{n-i}S_{n-i},\beta^{n}S_n)= 0~{\rm for~~}  i=1,\ldots,(n-1).$$} Since $(S_1,\ldots,S_{n-1},S_n)$ is a $\Gamma_n$-contraction, one can easily show that $(\gamma_1S_1,\ldots,\gamma_{n-1}S_{n-1})$ is a $\Gamma_{n-1}$-contraction. Thus, $(4)$ implies $(5).$

Suppose $(5)$ holds. Then for all $\beta\in \mathbb T$ with $ k(i)= \binom{n-1}{i}+\binom{n-1}{n-i}$ and for $ i=1,\ldots,(n-1),$ we see that
\small{\begin{equation}\label{Pphi11kp}k(i)^2(I-S_n^*S_n)+(S_i^*S_i-S_{n-i}^*S_{n-i})-k(i)\beta ^i(S_i-S_{n-i}^*S_n)-k(i)\bar{\beta}^i(S_i^*-S_n^*S_{n-i})= 0\end{equation}} $$\textit{\rm and}$$
\small{\begin{equation} \label{Pphi22kp}k(i)^2(I-S_n^*S_n)+(S_{n-i}^*S_{n-i}-S_{i}^*S_{i})-k(i)\beta ^{n-i}(S_{n-i}-S_{i}^*S_n)-k(i)\bar{\beta}^{n-i}(S_{n-i}^*-S_n^*S_{i})= 0.\end{equation}}  Putting $\beta^i=1$ and $\beta^i=-1$ respectively in  \eqref{Pphi11kp} and $\beta^{n-i}=1$ and $\beta^{n-i}=-1$ respectively in  \eqref{Pphi22kp}  and adding them, we get $S_n^*S_n=I$ from which it follows by the same argument as above $S_i=S_{n-i}^*S_n$ and $S_{n-i}=S_{i}^*S_n$ for  $ i=1,\ldots,(n-1).$ Thus, we get $(5)$ implies $(2).$

Now we will prove that $(5)$ is equivalent to $(6).$  Suppose $(5)$ holds. Then for all $\beta\in \mathbb T$ with $ k(i)= \binom{n-1}{i}+\binom{n-1}{n-i}$ and for $ i=1,\ldots,(n-1),$ we obtain
\small{\begin{equation}\label{Pphi11kp}k(i)^2(I-S_n^*S_n)+(S_i^*S_i-S_{n-i}^*S_{n-i})-k(i)\beta ^i(S_i-S_{n-i}^*S_n)-k(i)\bar{\beta}^i(S_i^*-S_n^*S_{n-i})= 0\end{equation}} $$\textit{\rm and}$$
\small{\begin{equation} \label{Pphi22kp}k(i)^2(I-S_n^*S_n)+(S_{n-i}^*S_{n-i}-S_{i}^*S_{i})-k(i)\beta ^{n-i}(S_{n-i}-S_{i}^*S_n)-k(i)\bar{\beta}^{n-i}(S_{n-i}^*-S_n^*S_{i})= 0.\end{equation}} 
Since $r(S_i)<\binom{n-1}{i}+\binom{n-1}{n-i},$ the operator $(k(i)I-\beta^iS_i)$ and $(k(i)I-\beta^{n-i}S_{n-i})$ are invertible for $i=1,\ldots,(n-1).$ Therefore from  \eqref{Pphi11kp} and  \eqref{Pphi22kp}, we conclude that $(k(i)\beta^nS_n-S_{n-i})(k(i)-\beta^iS_i)^{-1}$ and $(k(i)\beta^{n}S_n-S_{i})(k(i)-\beta^{n-i}S_{n-i})^{-1}$ are isometry for all $\beta\in \mathbb T$ and for $i = 1,\ldots, (n-1).$

Clearly, $(6)$ implies $(5).$ Combining all we conclude that all the above conditions are equivalent. This completes the proof.
\end{proof}

\section{A Necessary condition for the existence of Dilation}
In this section we  find out the necessary conditions for the existence of rational dilation. First, we  define $\Gamma_n$-isometric dilation of $\Gamma_n$-contraction.  
Also, we would like to discuss several properties
of $\Gamma_n$-isometric dilation of  $\Gamma_n$-contractions.

\begin{defn}
Let $(S_1,\ldots,S_n)$ be a $\Gamma_n$-contraction on a Hilbert space $\mathcal H.$ A commuting $n$-tuple of operators $(T_1,\ldots,T_{n-1},V)$ defined on a Hilbert space $\mathcal K$ containing $\mathcal H$ as subspace, is said to be $\Gamma_n$-isometric dilation of $(S_1,\ldots,S_n)$ if it satifies the following properties:
\begin{itemize}
\item  $(T_1,\ldots,T_{n-1},V)$ is $\Gamma_n$-isometric;

\item $P_{\mathcal H}T_1^{m_1}\ldots T_{n-1}^{m_{n-1}}V^n\mid_{\mathcal H}=S_1^{m_1}\ldots S_{n}^{n},$ for all non-negative integers $m_1,\ldots,m_{n-1},n.$
\end{itemize}
\end{defn}
If $\mathcal K={\overline{\rm{span}}}\{T_1^{m_1}\ldots T_{n-1}^{m_{n-1}}V^nh:h\in\mathcal H~{\rm{and}}~m_1,\ldots,m_{n-1},n \in \mathbb N\cup\{0\}\},$ then we call the $n$-tuple a minimal $\Gamma_n$-isometric dilation of $(S_1,\ldots,S_n).$ Similarly we can define $\Gamma_n$-unitary dilation of a $\Gamma_n$-contraction.

\begin{prop}\label{isometric dilation1}
Let $(S_1,\ldots,S_n)$ be a $\Gamma_n$-contraction defined on a Hilbert space $\mathcal H.$ Also assume that $(S_1,\ldots,S_n)$ has a $\Gamma_n$-isometric dilation. Then $(S_1,\ldots,S_n)$ has a minimal $\Gamma_n$-isometric dilation.
\end{prop}
\begin{proof}
Let $(T_1,\ldots,T_{n-1},V)$ defined on a Hilbert space $\mathcal K$ containing $\mathcal H$ as a subspace, be a $\Gamma_n$-isometric dilation of $(S_1,\ldots,S_n).$ Let $\mathcal K_{0}$ be the space defined by $$\mathcal K_{0}=\overline{{\rm{span}}}\{T_1^{m_1}\ldots T_{n-1}^{m_{n-1}}V^nh:h\in\mathcal H~{\rm{and}}~m_1,\ldots,m_{n-1},n \in \mathbb N\cup\{0\}\}.$$ One can easily verify that $\mathcal K_{0}$ is invariant under $T_1^{m_1},\ldots, T_{n-1}^{m_{n-1}},V^n,$ for any non-negative integers $m_1,\ldots,m_{n-1},n.$ If we denote $T_{11}=T_1\mid_{\mathcal K_{0}},\ldots,T_{1(n-1)}=T_{n-1}\mid_{\mathcal K_{0}}$ and $V_1=V\mid_{\mathcal K_{0}},$ then we get $$\mathcal K_{0}=\overline{{\rm{span}}}\{T_{11}^{m_1}\ldots T_{1(n-1)}^{m_{n-1}}V_1^nh:h\in\mathcal H~{\rm{and}}~m_1,\ldots,m_{n-1},n \in \mathbb N\cup\{0\}\}.$$ Therefore for any any non-negative integers $m_1,\ldots,m_{n-1},n$ we have
$$P_{\mathcal H}(T_1^{m_1}\ldots T_{n-1}^{m_{n-1}}V^n)h=P_{\mathcal H}(T_{11}^{m_1}\ldots T_{1(n-1)}^{m_{n-1}}V_1^n)h~{\rm{for~all~h\in \mathcal H}}.$$
$(T_{11},\ldots, T_{1(n-1)},V_1)$ is a  $\Gamma_n$-contraction, because it is the restriction of a  $\Gamma_n$-contraction $(T_1,\ldots,V)$ to a common invariant subspace $\mathcal K_{0}.$ Again, $V_1$ is also isometry,  because it is the restriction of an isometry to a common invariant subspace $\mathcal K_{0}.$ Therefore, by Theorem \ref{Gamma_n isometry}, $(T_{11},\ldots, T_{1(n-1)},V_1)$ is a  $\Gamma_n$-isometry. Hence $(T_{11},\ldots, T_{1(n-1)},V_1)$ is a minimal $\Gamma_n$-isometry of $(S_1,\ldots,S_n).$
\end{proof}
\begin{prop}\label{isometric dilation2}
Let $(T_1,\ldots,T_{n-1},V)$ be $n$-tuple of operators defined on a Hilbert space $\mathcal K$ containing $\mathcal H$ as a subspace. Then $(T_1,\ldots,T_{n-1},V)$ is  a minimal $\Gamma_n$-isometric dilation of a $\Gamma_n$-contraction $(S_1,\ldots,S_n)$ if and only if  $(T_1^*,\ldots,T_{n-1}^*,V^*)$ is a $\Gamma_n$-co-isometric extension of $(S_1^*,\ldots,S_n^*).$
\end{prop}
\begin{proof}
We first prove that $S_iP_{\mathcal H}=P_{\mathcal H}T_i$ for $i=1,\ldots,(n-1)$ and  $S_nP_{\mathcal H}=P_{\mathcal H}V,$ where $P_{\mathcal H}:\mathcal K\rightarrow \mathcal H$ is orthogonal projection onto $\mathcal H.$ It is also evident that $$\mathcal K={\overline{\rm{span}}}\{T_1^{m_1}\ldots T_{n-1}^{m_{n-1}}V^nh:h\in\mathcal H~{\rm{and}}~m_1,\ldots,m_{n-1},n \in \mathbb N\cup\{0\}\}.$$ Also, for $h\in \mathcal H,$ we have
\begin{align*}
S_iP_{\mathcal H}(T_1^{m_1}\ldots T_{n-1}^{m_{n-1}}V^nh)=S_i(S_1^{m_1}\ldots S_{n-1}^{m_{n-1}}S_n^nh)&=(S_1^{m_1}\ldots S_{i}^{m_{i}+1}\ldots S_{n-1}^{m_{n-1}}S_n^nh)\\&=P_{\mathcal H}(T_1^{m_1}\ldots T_{i}^{m_{i}+1}\ldots T_{n-1}^{m_{n-1}}V^nh)\\&=P_{\mathcal H}T_i(T_1^{m_1}\ldots T_{n-1}^{m_{n-1}}V^nh)
\end{align*}
which gives  $S_iP_{\mathcal H}=P_{\mathcal H}T_i$ for $i=1,\ldots,(n-1).$ Similarly we can also prove that $S_nP_{\mathcal H}=P_{\mathcal H}V.$
Now for any $h \in \mathcal H$ and for $k \in \mathcal K,$ we get
$$\langle S_i^*h,k\rangle=\langle P_{\mathcal H}S_i^*h,k\rangle=\langle S_1^*h,P_{\mathcal H}k\rangle=\langle h,S_iP_{\mathcal H}k\rangle=\langle h,P_{\mathcal H}T_i k\rangle=\langle T_i^*h,k\rangle.$$ Thus, we have $S_i^*=T_{i}^{*}\mid_{\mathcal H}$ for $i=1,\ldots,(n-1).$ Similarly, we can also show that $S_n^*=V^*{\mid\mathcal H}.$

The converse part is very easy to verify. This completes the proof.
\end{proof}
The following proposition, describes a model for pure $\Gamma_n$-isometry which will be used to prove the next proposition.
\begin{prop}\cite[Theorem 4.10]{SS}\label{pure isometry}
Let $S_1,\ldots,S_n$ be commuting operators on a Hilbert space $\mathcal H.$ Then $(S_1,\ldots,S_n)$ is a pure $\Gamma_n$-isometry if and only if there exist a separable Hilbert space $\Sigma$ and a unitary operator $U:\mathcal H\rightarrow H^2(\Sigma)$ and function $\phi_1,\ldots,\phi_{n-1}$ in $H^{\infty}(\mathcal B(\Sigma))$ and operators $E_i \in \mathcal B(\Sigma),i=1,\ldots,(n-1)$ such that
\begin{enumerate}
\item $S_i=U^*M_{\phi_i}U,i=1,\ldots,(n-1),S_n=U^*M_{z}U;$

\item $\phi_i(z)=E_i+E_{n-i}^*z$ for $i=1,\ldots,(n-1);$

\item $\|E_i+E_{n-i}^*z\|\leq \binom{n-1}{i}+\binom{n-1}{n-i}$ for $i=1,\ldots,(n-1) ;$

\item $[E_{i},E_{j}]=0$ and $[E_i,E_{n-j}^*]=[E_j,E_{n-i}^*]$ for $i,j=1,\ldots,(n-1).$

\end{enumerate} 
\end{prop}

\begin{prop}\label{pro}
Let $\mathcal H_1$ be a Hilbert space and let $(S_1,\ldots,S_n)$ be a $\Gamma_n$-contraction on $\mathcal H$ with $(n-1)$-tuple of fundamental operators $(E_1,\ldots,E_{n-1})$ and $S_n$ is such that
$S_n(\mathcal D_{S_n})=\{0\}$ and $S_n\ker(D_{S_n})\subseteq \mathcal D_{S_n}.$ Also assume that $(S_1^*,\ldots,S_n^*)$ has $\Gamma_n$-isometric dilation. Then the $(n-1)$-tuple of fundamental operators $(E_1,\ldots,E_{n-1})$ satisfies the following conditions:
\begin{enumerate}
\item $[E_i,E_j]=0,$ where $[P,Q]=PQ-QP;$

\item $[E_i,E_{n-j}^*]=[E_j,E_{n-i}^*]$ for $1\leq i,j\leq (n-1).$
\end{enumerate}
\end{prop}
\begin{proof}
Let $(T_1,\ldots,T_{n-1},V)$ on a Hilbert space $\mathcal K$ containing $\mathcal H$ as a subspace, be a  $\Gamma_n$-isometric dilation of $(S_1^*,\ldots,S_n^*).$ From Proposition \ref{isometric dilation1}, it follows that $(T_1,\ldots,T_{n-1},V)$ is a minimal $\Gamma_n$-isometric dilation. By Proposition \ref{isometric dilation2}, we conclude that
$(T_1^*,\ldots,T_{n-1}^*,V^*)$ is a $\Gamma_n$-co-isometric extension of $(S_1,\ldots,S_n).$ Since $(T_1,\ldots,T_{n-1},V)$  is a $\Gamma_n$-isometriy on $\mathcal K$, by Theorem \ref{Gamma_n isometry}, we have
$$(T_1,\ldots,T_{n-1},V)=(T_1^{(1)},\ldots,T_{n-1}^{(1)},U^{(1)})\oplus (T_1^{(2)},\ldots,T_{n-1}^{(2)},V^{(2)})~{\rm{on}}~\mathcal K_1\oplus \mathcal K_2,$$ where $(T_1^{(1)},\ldots,T_{n-1}^{(1)},U^{(1)})$ on $\mathcal K_1$ is a $\Gamma_n$-unitary and $(T_1^{(2)},\ldots,T_{n-1}^{(2)},V^{(2)})$ on $\mathcal K_2$ is a pure $\Gamma_n$-isometry. Since $(T_1^{(2)},\ldots,T_{n-1}^{(2)},V^{(2)})$ on $\mathcal K_2$  is a pure $\Gamma_n$-isometry, by Proposition \ref{pure isometry}, $\mathcal K_2$ can be identified with $H^{2}(E)\cong H^{2}(\mathbb D)\otimes E$ and $T_1^{(2)},\ldots,T_{n-1}^{(2)},V^{(2)}$ can be identified with the multiplication operators $M_{\phi_{1}},\ldots,M_{\phi_{n-1}},M_z$ on $H^{2}(E)$ for some $\phi_i \in H^{\infty}(\mathcal B(E)),$  where $\phi_i(z)=A_i+A_{n-i}^*z$ for $i=1,\ldots,(n-1),z\in \mathbb D, E=\mathcal D_{V^{(2)*}}$ and $(A_1^*,\ldots,A_{n-1}^*)$ is the $(n-1)$-tuples of fundamental operators of $(T_1^{(2)*},\ldots,T_{n-1}^{(2)*},V^{(2)*}).$ Since $H^{2}(E)\cong H^{2}(\mathbb D)\otimes E,$ therefore we have
$$M_{\phi_i}=\left[\begin{smallmatrix} A_i &0 & 0 &\ldots\\A_{n-i}^* & A_i & 0 &\ldots\\0 & A_{n-i}^* & A_i & \ldots \\\vdots &\vdots &\vdots  & \ddots
\end{smallmatrix}\right] ~~ {\rm{and}}~~M_z=\left[\begin{smallmatrix} 0 &0 & 0 &\ldots\\I & 0& 0 &\ldots\\0 & I & 0 & \ldots \\\vdots &\vdots &\vdots  & \ddots
\end{smallmatrix}\right]~{\rm{for}}~ i=1,\ldots,n-1.$$ 

We first prove $\mathcal D_{S_n}\subseteq E\oplus\{0\}\oplus\ldots\subseteq H^2(E)=\mathcal K_2.$ To prove this, let $h=h_1\oplus h_2 \in \mathcal D_{S_n} \subseteq \mathcal H,$ where $h_1 \in \mathcal K_1$ and $(c_0,c_1,\ldots)^{T}\in H^2(E).$ Then, together with $S_n(\mathcal D_{S_n})=\{0\}$ and $U^{(1)}$ is unitary, we get
$$S_nh=V^*h=V^*(h_1\oplus h_2)=U^{(1)*}h_1\oplus M_z^*h_2=U^{(1)*}h_1\oplus (c_1,c_2,\ldots)^{T}$$ which, implies that  $h_1=0$ and $c_1=c_2=\ldots=0.$

Now we  prove that $\ker(D_{S_n})\subseteq \{0\}\oplus E \oplus \{0\}\oplus \ldots \subseteq H^2(E).$
To prove this, take $k=k_1\oplus k_2 \in \ker (D_{S_n})\subseteq \mathcal H,$ where $k_1\in \mathcal K_1$ and $k_2=(g_0,g_1,\ldots)^{T}\in H^2(E).$  Notice that
\begin{align}\label{a}
D_{S_n}^2k=(I-S_n^*S_n)k=P_{\mathcal H}(I-VV^*)k\nonumber&=P_{\mathcal H}(k_1\oplus k_2-k_1\oplus M_zM_z^*k_2)\\&=k_1\oplus (g_0,g_1,\ldots)^{T}-P_{\mathcal H}(k_1\oplus (0,g_1,\ldots)^{T}).
\end{align}
If $D_{S_n}^2k=0,$ then, from \eqref{a}, we have $k_1\oplus (g_0,g_1,\ldots)^{T}=P_{\mathcal H}(k_1\oplus (0,g_1,\ldots)^{T}$ from which it follows that $g_0=0.$
Since $S_n\ker(D_{S_n})\subseteq \mathcal D_{S_n},$ we get for every $k=k_1\oplus (0,g_1,\ldots)^{T}\in \ker(D_{S_n}),$
$$P(k_1\oplus (0,g_1,\ldots)^{T})=U^{(1)*}k_1\oplus (g_1,g_2,\ldots)^{T}\in \mathcal D_{S_n}.$$ Also by $S_n(\mathcal D_{S_n})=\{0\},$ from above we conclude that $k_1=0$ and $g_2=g_3=\ldots=0.$ Hence, $\ker(D_{S_n})\subseteq \{0\}\oplus E \oplus \{0\}\oplus \ldots \subseteq H^2(E).$

Since $\mathcal H=\mathcal D_{S_n}\oplus \ker(D_{S_n}),$ we have $\mathcal H\subseteq E\oplus E \oplus \{0\}\oplus \ldots \subseteq H^2(E).$ Therefore, $(M_{\phi_{1}}^*,\ldots,M_{\phi_{n-1}}^*,M_z^*)$ on $H^{2}(E)$ is a $\Gamma_n$-co-isometric extension of $(S_1,\ldots,S_n).$ We now compute the fundamental operator pairs of $(M_{\phi_{1}}^*,\ldots,M_{\phi_{n-1}}^*,M_z^*).$ Note that
\begin{align*}
M_{\phi_{i}}^*-M_{\phi_{n-i}}M_z^*&=\left[\begin{smallmatrix} A_i^* & A_{n-i} & 0 &\ldots\\0 & A_i^* & A_{n-i} &\ldots\\0 & 0 & A_i^* & \ldots \\\vdots &\vdots &\vdots  & \ddots
\end{smallmatrix}\right]-\left[\begin{smallmatrix} 0 & A_{n-i} & 0 &\ldots\\0 & A_i^* & A_{n-i} &\ldots\\0 & 0 & A_i^* & \ldots \\\vdots &\vdots &\vdots  & \ddots
\end{smallmatrix}\right]=\left[\begin{smallmatrix} A_i^* & 0 & 0 &\ldots\\0 & 0 & 0 &\ldots\\0 & 0 & 0 & \ldots \\\vdots &\vdots &\vdots  & \ddots
\end{smallmatrix}\right]~{\rm{for}}~ i=i,\ldots,(n-1).
\end{align*} 
Similarly one can also prove that
$$M_{\phi_{n-i}}^*-M_{\phi_{i}}M_z^*=\left[\begin{smallmatrix} A_{n-i}^* & 0 & 0 &\ldots\\0 & 0 & 0 &\ldots\\0 & 0 & 0 & \ldots \\\vdots &\vdots &\vdots  & \ddots
\end{smallmatrix}\right]~{\rm{for}}~i=1,\ldots,(n-1).$$
It is also easy to verify that
\begin{align*}
D_{M_z^*}^2=I-M_zM_z^*=\left[\begin{smallmatrix} I & 0 & 0 &\ldots\\0 & 0 & 0 &\ldots\\0 & 0 & 0 & \ldots \\\vdots &\vdots &\vdots  & \ddots
\end{smallmatrix}\right]
\end{align*}
Thus, we have $\mathcal D_{M_z^*}=E\oplus\{0\}\oplus \ldots$ and $D_{M_z^*}^2=D_{M_z^*}=Id$ on $E\oplus\{0\}\oplus \ldots.$ If we consider
\begin{equation}\label{hatE}\hat{E}_i=\left[\begin{smallmatrix} A_i^* & 0 & 0 &\ldots\\0 & 0 & 0 &\ldots\\0 & 0 & 0 & \ldots \\\vdots &\vdots &\vdots  & \ddots
\end{smallmatrix}\right]~{\rm{and}}~\hat{E}_{n-i}=\left[\begin{smallmatrix} A_{n-i}^* & 0 & 0 &\ldots\\0 & 0 & 0 &\ldots\\0 & 0 & 0 & \ldots \\\vdots &\vdots &\vdots  & \ddots
\end{smallmatrix}\right],\end{equation}  then
we have \begin{equation}\label{M_z^*}
M_{\phi_{i}}^*-M_{\phi_{n-i}}M_z^*=D_{M_z^*}\hat{E}_iD_{M_z^*} ~{\rm{and}}~M_{\phi_{n-i}}^*-M_{\phi_{i}}M_z^*=D_{M_z^*}\hat{E}_{n-i}D_{M_z^*}~{\rm{for}}~ i=i,\ldots,(n-1).\end{equation}  Therefore, from \eqref{M_z^*}, we conclude that $\hat{E}_1,\ldots,\hat{E}_{n-1}$ are the $(n-1)$-tuples of fundamental operators of $(M_{\phi_{1}}^*,\ldots,M_z^*).$

Our next claims are $\hat{E}_iD_{M_z^*}\mid_{\mathcal D_{S_n}} \subseteq \mathcal D_{S_n}$ and $\hat{E}_i^*D_{M_z^*}\mid_{\mathcal D_{S_n}} \subseteq \mathcal D_{S_n}$ for $i=i,\ldots,(n-1).$ To prove this, let $h=(c_0,0,\ldots)^{T}\in \mathcal D_{S_n}.$ Then we get $\hat{E}_{i}M_z^*h=(A_i^*c_0,0,\ldots)^{T}=M_{\phi_{i}}^*h,$ which, together with $M_{\phi_{i}}^*h \in \mathcal H$ and $M_{\phi_{i}}^*\mid{\mathcal H}=S_i,$ implies that $(A_i^*c_0,0,\ldots)^{T} \in \mathcal D_{S_n}.$ Hence, $\hat{E}_iD_{M_z^*}\mid_{\mathcal D_{S_n}} \subseteq \mathcal D_{S_n}$ for $i=1,\ldots,(n-1).$

We now compute the adjoint of $S_n.$ Let $(c_0,c_1,0,\ldots)^{T}$ and $(d_0,d_1,0,\ldots)^{T}$ be two arbitrary elements in $\mathcal H,$ where $(c_0,0,0,\ldots)^{T},(d_0,0,0,\ldots)^{T}\in \mathcal D_{S_n}$ and $(0,c_1,0,\ldots)^{T},(0,d_1,0,\ldots)^{T}\in \ker(D_{S_n}).$ Note that
\begin{align*}
\langle S_n^*(c_0,c_1,0,\ldots)^{T},(d_0,d_1,0,\ldots)^{T}\rangle &=\langle(c_0,c_1,0,\ldots)^{T}, P(d_0,d_1,0,\ldots)^{T}\rangle\\&=(c_0,c_1,0,\ldots)^{T}, M_z^*(d_0,d_1,0,\ldots)^{T}\rangle\\&=\langle(c_0,c_1,0,\ldots)^{T}, (d_1,0,\ldots)^{T}\rangle\\&=\langle c_0,d_1\rangle_{E}\\&=\langle(0,c_0,c_1,0,\ldots)^{T}, (d_0,d_1,0,\ldots)^{T}\rangle.
\end{align*}
This shows that $S_n^*(c_0,c_1,0,\ldots)^{T}=(0,c_0,c_1,0,\ldots)^{T}.$ Therefore, for $h_0=(c_0,0,0,\ldots)^{T}\in \mathcal D_{S_n},$ we have $S_n^*h_0=(0,c_0,0,\ldots)^{T}\in \mathcal H$ and $\hat{E}_i^*D_{M_z^*}h_0=(A_ic_0,0,\ldots)^{T}=M_{\phi_{n-i}}^*(0,c_0,0,\ldots)^{T}$ for $i=1,\ldots,(n-1).$ As $(A_ic_0,0,\ldots)^{T}\in \mathcal D_{S_n},$ we deduce that $\hat{E}_i^*D_{M_z^*}\mid_{\mathcal D_{S_n}} \subseteq \mathcal D_{S_n}$ for $i=i,\ldots,(n-1).$ Our next aim is to show
 $\hat{E}_i\mid_{\mathcal D_{S_n}}=E_i$ and $\hat{E}_i^*\mid_{\mathcal D_{S_n}}=E_i^*$ for $i=i,\ldots,(n-1).$
Since $M_z^*\mid_{\mathcal H}=S_n, \mathcal D_{S_n}\subseteq \mathcal M_z^*=E\oplus\{0\}\oplus\ldots$ and $D_{M_z^*}$ is the projection onto $\mathcal D_{M_z^*},$ we see that $D_{M_z^*}\mid_{\mathcal H}=D_{M_z^*}^2\mid_{\mathcal D_{S_n}}=D_{S_n}^2.$ Therefore, from  \eqref{M_z^*}, we get 
\begin{equation}\label{PM_z^*}
P_{\mathcal H}(M_{\phi_{i}}^*-M_{\phi_{n-i}}M_z^*)\mid_{\mathcal H}=P_{\mathcal H}D_{M_z^*}\hat{E}_iD_{M_z^*}\mid_{\mathcal H} ~{\rm{for}} i=i,\ldots,(n-1).\end{equation}  Since $(M_{\phi_{1}}^*,\ldots,M_z^*)$ on $H^{2}(E)$ is a $\Gamma_n$-co-isometric extension of $(S_1,\ldots,S_n),$ from \eqref{PM_z^*}, we have
\begin{equation}\label{SPM_z^*}
S_{i}-S_{n-i}^*S_n=P_{\mathcal H}D_{M_z^*}\hat{E}_iD_{M_z^*}\mid_{\mathcal H},~{\rm{for}}~ i=i,\ldots,(n-1).\end{equation}  As $(E_1,\ldots,E_{n-1})$ are $(n-1)$-tuple of fundamental operators of $(S_1,\ldots,S_{n-1},S_n),$  we get
\begin{equation}\label{SSPM_z^*}
S_{i}-S_{n-i}^*S_n=D_{S_n}E_iD_{S_n} , ~~E_i\in \mathcal B(\mathcal D_{S_n})
~{\rm{for}}~ i=i,\ldots,(n-1).\end{equation}  Since $S_{i}-S_{n-i}^*S_n$ is zero on the orthogonal complement of $\mathcal D_{S_n},$
 $D_{M_z^*}\mid_{\mathcal D_{S_n}}=D_{S_n}$ and $\hat{E}_iD_{M_z^*}\mid_{\mathcal D_{S_n}} \subseteq \mathcal D_{S_n},$ from \eqref{SPM_z^*}, we see that
\begin{equation}\label{SPPM_z^*}
S_{i}-S_{n-i}^*S_n=P_{\mathcal H}(M_{\phi_{i}}^*-M_{\phi_{n-i}}M_z^*)\mid_{\mathcal H}=P_{\mathcal H}D_{M_z^*}\hat{E}_iD_{M_z^*}\mid_{\mathcal H}=D_{S_n}\hat{E}_iD_{S_n} ~{\rm{for}} i=i,\ldots,(n-1).\end{equation} 
By uniqueness of $E_i,$  we have $\hat{E}_i\mid_{\mathcal D_{S_n}}=E_i$ for $i=1,\ldots,(n-1).$ Similarly, one can prove also $\hat{E}_i^*\mid_{\mathcal D_{S_n}}=E_i^*$ for $i=i,\ldots,(n-1).$
 
Since $M_{\phi_i}^*$'s are commute, we have
$$M_{\phi_{n-i}}^* M_{\phi_{n-j}}^*=M_{\phi_{n-j}}^* M_{\phi_{n-i}}^*~~{\rm{for}}~~ i,j=1,\ldots,(n-1).$$ The above condition is equivalent to for $i,j=1,\ldots,(n-1),$ $$\left[\begin{smallmatrix} A_{n-i}^* & A_i & 0 &\ldots\\0& A_{n-i}^* & A_i &\ldots\\0 & 0& A_{n-i}^* & \ldots \\\vdots &\vdots &\vdots  & \ddots
\end{smallmatrix}\right]\left[\begin{smallmatrix} A_{n-j}^* & A_j & 0 &\ldots\\0& A_{n-j}^* & A_j &\ldots\\0 & 0 & A_{n-j}^* & \ldots \\\vdots &\vdots &\vdots  & \ddots
\end{smallmatrix}\right]=\left[\begin{smallmatrix} A_{n-j}^* & A_j & 0 &\ldots\\0& A_{n-j}^* & A_j &\ldots\\0 & 0 & A_{n-j}^* & \ldots \\\vdots &\vdots &\vdots  & \ddots
\end{smallmatrix}\right]\left[\begin{smallmatrix} A_{n-i}^* & A_i & 0 &\ldots\\0& A_{n-i}^* & A_i &\ldots\\0 & 0& A_{n-i}^* & \ldots \\\vdots &\vdots &\vdots  & \ddots
\end{smallmatrix}\right].$$
Equivalently, for $i,j=1,\ldots,n-1,$ $$\left[\begin{smallmatrix} A_{n-i}^*A_{n-j}^* & A_{n-i}^*A_j+A_iA_{n-j}^* & 0 &\ldots\\0& A_{n-i}^*A_{n-j}^* &  A_{n-i}^*A_j+A_iA_{n-j}^*&\ldots\\0 & 0& A_{n-i}^*A_{n-j}^*& \ldots \\\vdots &\vdots &\vdots  & \ddots
\end{smallmatrix}\right]=\left[\begin{smallmatrix} A_{n-j}^*A_{n-i}^*  & A_{n-j}^*A_i+A_jA_{n-i}^* & 0 &\ldots\\0& A_{n-j}^*A_{n-i}^* & A_{n-j}^*A_i+A_jA_{n-i}^* &\ldots\\0 & 0 & A_{n-j}^*A_{n-i}^* & \ldots \\\vdots &\vdots &\vdots  & \ddots
\end{smallmatrix}\right]$$
Comparing both sides we get
\small{$[A_{n-i}^*,A_{n-j}^*]=0~{\rm{ and}}~[A_i,A_{n-j}^*]=[A_j,A_{n-i}^*]~{\rm{ for}}~ i,j=1,\ldots,(n-1).$}
Also, from  \eqref{hatE} we obtain
$[\hat{E}_{n-i},\hat{E}_{n-j}]=0~{\rm{and}}~[\hat{E}_i,\hat{E}_{n-j}^*]=[\hat{E}_j,\hat{E}_{n-i}^*]~{\rm{for}}~ i,j=1,\ldots,(n-1).$ Taking restriction  to the subspace $\mathcal D_{S_n},$ we have
\mbox{$[E_{n-i},E_{n-j}]=0~{\rm{and}}~[E_i,E_{n-j}^*]=[E_j,E_{n-i}^*]~{\rm{ for}}~ i,j=1,\ldots,(n-1).$}
This completes the proof.
\end{proof}

\section{Conditional Dilation}
In this section, we want to show that if the commuting $(n-1)$-tuple of fundamental operators $E_1,\ldots,E_{n-1}$ and $F_1,\ldots,F_{n-1}$  of a $\Gamma_{n}$-contractions of $(S_1,\ldots,S_n)$ and  $(S_1^{*},\ldots,S_{n}^{*})$ respectively, satisfy the  conditions
$E_lE_{n-k}^{*}-E_kE_{n-l}^{*}=E_{n-k}^{*}E_l-E_{n-l}^{*}E_k$ and $F_l^{*}F_{n-k}-F_k^{*}F_{n-l}=F_{n-k}F_{l}^{*}-F_{n-l}F_{k}^{*},$ then $(S_1,\ldots,S_n)$ possesses a $\Gamma_{n}$-unitary dilation. The two conditions of $(n-1)$-tuple of fundamental operators of $(S_1,\ldots,S_n)$ and of $(S_1^{*},\ldots,S_{n}^{*})$, mentioned above, are the sufficient to have a such $\Gamma_{n}$-unitary dilation.
Throughout this section, we use the definitions of  $(n-1)$-tuple of fundamental operators as
\small{\begin{equation}\label{Si}
S_i-S_{n-i}^{*}S_n=D_{S_n}E_iD_{S_n}, S_{n-i}-S_{i}^{*}S_n=D_{S_n}E_{n-i}D_{S_n},
S_{i}^{*}-S_{n-i}S_n^{*}=D_{S_n^{*}}F_iD_{S_n^{*}}~{\rm{and}}~ S_{n-i}^{*}-S_{i}S_n^{*}=D_{S_n^{*}}F_{n-i}D_{S_n^{*}}.
\end{equation}}
The following proposition whose proof we skip because it is routine, gives a relation between $S_n,D_{S_n},D_{S_n^{*}}.$ 
\begin{prop}\label{PD_{S_n}}\cite[Chapter- I]{BFK}
Suppose $S_n$ is any contraction on a Hilbert space $\mathcal H.$ Then $S_nD_{S_n}=D_{S_n^{*}}S_n.$
\end{prop}
We begin with the following lemma which will be used later to prove the conditional dilation of $\Gamma_n.$
\begin{lem}\label{gamma_n contraction}
Let $(S_1, \ldots, S_n)$ be a $\Gamma_n$-contraction defined on a Hilbert space $\mathcal H.$ Let $(E_1,\ldots,E_{n-1})$ and $(F_1,\ldots,F_{n-1})$ be  the $(n-1)$-tuple of fundamental operators of $(S_1,\ldots,S_n)$ and  $(S_1^{*},\ldots,S_{n}^{*})$ respectively. Then the following properties hold:
\begin{enumerate}
\item $S_nE_i=F_{i}^{*}S_n\mid_{\mathcal D_{S_n}}$ and $S_n^*F_i=E_i^*S_n^*\mid_{\mathcal D_{S_n^*}}$ for $i=1,\ldots,(n-1),$

\item $D_{S_n}S_i=E_{i}D_{S_n}+E_{n-i}^{*}D_{S_n}S_n$ and $D_{S_n}S_{n-i}=E_{n-i}D_{S_n}+E_{i}^{*}D_{S_n}S_n$ for $i=1,\ldots,(n-1),$

\item $S_iD_{S_n^*}=D_{S_n^*}F_{i}^{*}+S_nD_{S_n^*}F_{n-i}$ and $S_{n-i}D_{S_n^*}=D_{S_n^*}F_{n-i}^{*}+S_nD_{S_n^*}F_{i}$ for $i=1,\ldots,(n-1),$

\item $S_{n-j}^{*}S_i-S_{n-i}^{*}S_j= D_{S_n}(E_{n-j}^{*}E_i-E_{n-i}^{*}E_j)D_{S_n},$ when $[E_i, E_j]=0,$ for $i,j=1,\ldots,(n-1),$

\item $S_iS_{n-j}^{*}-S_jS_{n-i}^{*}=D_{S_n^*}(F_{i}^{*}F_{n-j}-F_{j}^{*}F_{n-i})D_{S_n^*},$ when $[F_i, F_j]=0,$ for $i,j=1,\ldots,(n-1),$

\item $\omega(E_{n-i}+E_i^*z)\leq \binom{n-1}{i}+\binom{n-1}{n-i}$ and $\omega(F_{n-i}^{*}+E_iz)\leq \binom{n-1}{i}+\binom{n-1}{n-i}$ for $i=1,\ldots,(n-1).$
\end{enumerate}
\end{lem}
\begin{proof}
\begin{enumerate}
\item For $D_{S_n}h\in \mathcal D_{S_n}$ and $D_{S_n^*}h^{\prime}\in \mathcal D_{S_n^*},$ we have
\begin{align*}
\langle S_nE_iD_{S_n}h,D_{S_n^*}h^{\prime}\rangle&=\langle D_{S_n^*}S_nE_iD_{S_n}h,h^{\prime}\rangle\\&=\langle S_nD_{S_n}E_iD_{S_n}h,h^{\prime}\rangle\\&=\langle S_n(S_i-S_{n-i}^{*}S_n)h, h^{\prime}\rangle\\&=\langle (S_i-S_nS_{n-i}^{*})S_nh, h^{\prime}\rangle\\&=\langle D_{S_n^*}F_i^*D_{S_n^*}S_nh,h^{\prime}\rangle\\&=\langle F_i^*S_nD_{S_n}h,D_{S_n^*}h^{\prime}\rangle,
\end{align*}
which implies first part of $(1).$
Similarly, it is also easy to verify the other equality. 

\item For $h, h^{\prime}\in \mathcal H,$ we see that 
\begin{align*}
\langle E_{i}D_{S_n}h+E_{n-i}^{*}D_{S_n}S_nh ,D_{S_n}h^{\prime}\rangle&=\langle D_{S_n}E_{i}D_{S_n}h+D_{S_n}E_{n-i}^{*}D_{S_n}S_nh ,h^{\prime}\rangle\\&=\langle D_{S_n}^2S_ih, h^{\prime}\rangle\\&=\langle D_{S_n}S_ih, D_{S_n}h^{\prime}\rangle,
\end{align*} 
which gives the first part of $(1).$ The other part is also easy to verify.
\item The proof $(3)$  is similar to $(2).$

\item Note that
\small{\begin{align*}
S_{n-j}^{*}S_i-S_{n-i}^{*}S_j& =S_{n-j}^{*}(S_{n-i}^{*}+D_{S_n}E_iD_{S_n})-S_{n-i}^{*}(S_{n-j}^{*}+D_{S_n}E_jD_{S_n})\\&=S_{n-j}^{*}D_{S_n}E_iD_{S_n}-S_{n-i}^{*}D_{S_n}E_jD_{S_n}\\&=(D_{S_n}E_{n-j}^*+S_n^*D_{S_n}E_j)E_iD_{S_n}-(D_{S_n}E_{n-i}^*+S_n^*D_{S_n}E_i)E_jD_{S_n}~[by~ part (2)~ of~ this~ lemma]\\& =D_{S_n}(E_{n-j}^{*}E_i-E_{n-i}^{*}E_j)D_{S_n}.
\end{align*}}

\item The proof of $(5)$  is same as that of $(4).$

\item By Theorem \ref{uniqueness}, we have for all  $z\in \mathbb T, $ $\omega(E_i+E_{n-i}z) \leq  \binom{n-1}{i}+\binom{n-1}{n-i}$ for $i=1,\ldots,n-1.$ The condition $\omega(E_i+E_{n-i}z) \leq  \binom{n-1}{i}+\binom{n-1}{n-i},$ implies that $\omega(E_i^*+E_{n-i}z) \leq  \binom{n-1}{i}+\binom{n-1}{n-i}$ for all $z\in \mathbb T $ and $i=1,\ldots,n-1.$ Similarly, one can also show that $\omega(F_{n-i}^{*}+E_iz)\leq \binom{n-1}{i}+\binom{n-1}{n-i}$ for $i=1,\ldots,(n-1).$

\end{enumerate}
\end{proof}
The following theorem allows us to construct the conditional dilation of $\Gamma_n$ of a $\Gamma_n$-contraction.
\begin{thm}\label{main dilation}
Suppose $(S_1,\ldots,S_n)$ is a $\Gamma_n$-contraction defined on a Hilbert space $\mathcal H$ such that the commuting $(n-1)$-tuple of fundamental operators $E_1,\ldots,E_{n-1}$ and $F_1,\ldots,F_{n-1}$ of $(S_1,\ldots,S_n)$ and  $(S_1^{*},\ldots,S_{n}^{*})$  respectively satisfy the  conditions
$E_lE_{n-k}^{*}-E_kE_{n-l}^{*}=E_{n-k}^{*}E_l-E_{n-l}^{*}E_k$ and $F_l^{*}F_{n-k}-F_k^{*}F_{n-l}=F_{n-k}F_{l}^{*}-F_{n-l}F_{k}^{*}.$ Let $\mathcal K=\ldots \oplus \mathcal D_{S_n}\oplus \mathcal D_{S_n}\oplus \mathcal D_{S_n}\oplus \mathcal H \oplus \mathcal D_{S_n^*}\oplus \mathcal D_{S_n^*}\oplus \mathcal D_{S_n^*}\oplus \ldots$ and let $(R_1,\ldots,R_{n-1},U)$ be a $n$-tuple of operators defined on $\mathcal K$ by
\scriptsize\begin{equation}R_i=\left[
\begin{array}{cccc|c|cccc}
  \ddots & \vdots & \vdots & \vdots & \vdots & \vdots & \vdots & \vdots &\iddots \\
  \ldots & E_i & E_{n-i}^* & 0 & 0 & 0 &0 & 0 & \ldots \\
  \ldots & 0 & E_i & E_{n-i}^* & 0 & 0 & 0 &0  & \ldots\\
  \ldots & 0 & 0& E_i & E_{n-i}^*D_{S_n} & - E_{n-i}^*S_n^* & 0 & 0   & \ldots\\
  \hline
  \ldots& 0 & 0 & 0 & S_i & D_{S_n^*}F_{n-i} & 0 & 0 & \ldots\\
  \hline
  \ldots &0 & 0 & 0 & 0 & F_{i}^{*} & F_{n-i} & 0 & \ldots\\
\ldots &0 &0 & 0 & 0 & 0 & F_{i}^{*} & F_{n-i}  & \ldots\\
  \iddots & \vdots & \vdots & \vdots & \vdots & \vdots & \vdots & \vdots &\ddots
\end{array}
\right]
\end{equation}\normalsize and \scriptsize\begin{equation}U=\left[
\begin{array}{cccc|c|cccc}
  \ddots & \vdots & \vdots & \vdots & \vdots & \vdots & \vdots & \vdots &\iddots \\
  \ldots & 0 & I & 0 & 0 & 0 &0 & 0 & \ldots \\
  \ldots & 0 & 0 & I & 0 & 0 & 0 &0  & \ldots\\
  \ldots & 0 & 0& 0 & D_{S_n} & -S_n^* & 0 & 0   & \ldots\\
  \hline
  \ldots& 0 & 0 & 0 & P & D_{S_n^*} & 0 & 0 & \ldots\\
  \hline
  \ldots &0 & 0 & 0 & 0 & 0 & I & 0 & \ldots\\
\ldots &0 &0 & 0 & 0 & 0 & 0 & I  & \ldots\\
  \iddots & \vdots & \vdots & \vdots & \vdots & \vdots & \vdots & \vdots &\ddots
\end{array}
\right].
\end{equation}\normalsize
Also we assume that $(R_1,\ldots,R_{n-1},U)$ is a $\Gamma_n$-contraction. Then $(R_1,\ldots,R_{n-1},U)$ is a minimal $\Gamma_n$-unitary dilation of $(S_1,\ldots,S_n).$
\end{thm}
\begin{proof}
It is immediate from Sz.-Nazy-Foias dilation \cite[Chapter- I]{BFK} that $U$ is the minimal unitary dilation of $S_n.$ The minimality of $\Gamma_n$-unitary dilation follows from the fact that $\mathcal K$ and $U$ are the minimal dilation space and minimal unitary dilation of $S_n$ respectively. Since $U$ is unitary, in order to show that $(R_1,\ldots,R_{n-1},U)$ is a  a minimal $\Gamma_n$-unitary dilation of $(S_1,\ldots,S_n)$ one has to verify the following steps:
\small{\begin{enumerate}
\item $R_iR_j=R_jR_i$ for $i,j=1,\ldots,(n-1),$

\item $R_iU=UR_i$ for $i=1,\ldots,(n-1),$

\item $R_i=R_{n-i}^{*}U$ for $i=1,\ldots,(n-1),$

\item $R_iR_i^*=R_i^*R_i$  for $i=1,\ldots,(n-1).$
\end{enumerate}}

\noindent{\bf{Step~1:}}

\scriptsize\begin{align*}
R_iR_j=\left[\begin{smallmatrix}
\begin{array}{ccc|c|ccc}
  \ddots &\vdots  & \vdots & \vdots & \vdots & \vdots   &\iddots \\
  \ldots   & E_iE_j & E_iE_{n-j}^* +E^{*}_{n-i}E_j   & E_{n-i}^{*}E_{n-j}^{*}D_{S_n} & E_{n-i}^{*}E_{n-j}^{*}S_n^* &0& \ldots\\
  \ldots  & 0 & E_iE_j & E_iE_{n-j}^*D_{S_n} & - E_iE_{n-j}^*S_n^* & -E^{*}_{n-i}S_n^*F_{n-j}    & \ldots\\&        &        & +E^{*}_{n-i}D_{S_n}S_j & +E_{n-i}^{*}D_{S_n}D_{S_n^*}F_{n-j} & &\\ & & & &-E^{*}_{n-i}S_n^*F_{j}^{*}& &\\
  \hline
  \ldots &  0 & 0 & S_iS_j & S_iD_{S_n^*}F_{n-j} & D_{S_n^*}F_{n-i}F_{n-j}   & \ldots\\  & & & &+D_{S_n^*}F_{n-i}F^*_j &  &\\
  \hline
  \ldots  & 0 & 0 & 0 & F_{i}^{*}F_{j}^{*} &F_i^*F_{n-j} &  \ldots\\ & & & & &+ F_{n-i}F_j^* &\\
\iddots  & \vdots & \vdots & \vdots & \vdots & \vdots  &\ddots
\end{array}
\end{smallmatrix}\right]
\end{align*}\normalsize
$$\textit{\rm and}$$
\scriptsize\begin{align*}
R_jR_i=\left[\begin{smallmatrix}
\begin{array}{ccc|c|ccc}
  \ddots &\vdots  & \vdots & \vdots & \vdots & \vdots   &\iddots \\
  \ldots   & E_jE_i & E_jE_{n-i}^* +E^{*}_{n-j}E_i   & E_{n-j}^{*}E_{n-i}^{*}D_{S_n} & E_{n-j}^{*}E_{n-i}^{*}S_n^* &0& \ldots\\
  \ldots  & 0 & E_jE_i & E_jE_{n-i}^*D_{S_n} & - E_jE_{n-i}^*S_n^* & -E^{*}_{n-j}S_n^*F_{n-i}    & \ldots\\&        &        & +E^{*}_{n-j}D_{S_n}S_i & +E_{n-j}^{*}D_{S_n}D_{S_n^*}F_{n-i} & &\\ & & & &-E^{*}_{n-j}S_n^*F_{i}^{*}& &\\
  \hline
  \ldots &  0 & 0 & S_jS_i & S_jD_{S_n^*}F_{n-i} & D_{S_n^*}F_{n-j}F_{n-i}   & \ldots\\  & & & &+D_{S_n^*}F_{n-j}F^*_i &  &\\
  \hline
  \ldots  & 0 & 0 & 0 & F_{j}^{*}F_{i}^{*} &F_j^*F_{n-i} &  \ldots\\ & & & & &+ F_{n-j}F_i^* &\\
\iddots  & \vdots & \vdots & \vdots & \vdots & \vdots  &\ddots
\end{array}
\end{smallmatrix}\right]
\end{align*}\normalsize

First we  verify the equality of the entities $(-1,0),(0,1),(-1,1)$ in the matrices of $R_iR_j$ and $R_jR_i.$ To show this, we need to check the following operator identities.
\begin{enumerate}
\item $E_iE_{n-j}^*D_{S_n}+E^{*}_{n-i}D_{S_n}S_j = E_jE_{n-i}^*D_{S_n}+E^{*}_{n-j}D_{S_n}S_i$

\item $- E_iE_{n-j}^*S_n^* +E_{n-i}^{*}D_{S_n}D_{S_n^*}F_{n-j}-E^{*}_{n-i}S_n^*F_{j}^{*}$\\$=- E_jE_{n-i}^*S_n^* +E_{n-j}^{*}D_{S_n}D_{S_n^*}F_{n-i}-E^{*}_{n-j}S_n^*F_{i}^{*}$

\item $S_iD_{S_n^*}F_{n-j}+D_{S_n^*}F_{n-i}F^*_j=S_jD_{S_n^*}F_{n-i}+D_{S_n^*}F_{n-j}F^*_i$
\end{enumerate}

\noindent{\bf{(1)}} By applying Part $(2)$ of Lemma \ref{gamma_n contraction} and the identity $E_iE_{n-j}^{*}-E_jE_{n-i}^{*}=E_{n-j}^{*}E_i-E_{n-i}^{*}E_j$, we have
\begin{align*}
E_iE_{n-j}^*D_{S_n}+E^{*}_{n-i}D_{S_n}S_j & = E_iE_{n-j}^*D_{S_n}+E^{*}_{n-i}(E_{j}D_{S_n}+E_{n-j}^{*}D_{S_n}S_n)\\&=(E_iE_{n-j}^*+E^{*}_{n-i}E_{j})D_{S_n}+E^{*}_{n-i}E_{n-j}^{*}D_{S_n}S_n\\&=(E_jE_{n-i}^*+E^{*}_{n-j}E_{i})D_{S_n}+E^{*}_{n-i}E_{n-j}^{*}D_{S_n}S_n\\&=E_jE_{n-i}^*D_{S_n}+E^{*}_{n-j}D_{S_n}S_i
\end{align*}

\noindent{\bf{(2)}}  By using Lemma \ref{gamma_n contraction}, we see that both of left hand side and right hand side of the above equality are defined on $\mathcal D_{S_n^*}$ to $\mathcal D_{S_n}.$ Let $$H_1=- E_iE_{n-j}^*S_n^* +E_{n-i}^{*}D_{S_n}D_{S_n^*}F_{n-j}-E^{*}_{n-i}S_n^*F_{j}^{*}~{\rm{ and}}~ H_2= - E_jE_{n-i}^*S_n^* +E_{n-j}^{*}D_{S_n}D_{S_n^*}F_{n-i}-E^{*}_{n-j}S_n^*F_{i}^{*}.$$ Then by Part $(4),(5)$ of Lemma \ref{gamma_n contraction} and \eqref{Si}, we get
\small{\begin{align*}
D_{S_n}(H_2-H_1)D_{S_n^*}&=D_{S_n}( E_iE_{n-j}^{*}- E_jE_{n-i}^*)S_n^*D_{S_n^*}+D_{S_n}(E_{n-j}^{*}D_{S_n}D_{S_n^*}F_{n-i} -E_{n-i}^{*}D_{S_n}D_{S_n^*}F_{n-j})D_{S_n^*}\\&+D_{S_n}(E^{*}_{n-i}S_n^*F_{j}^{*} -E^{*}_{n-j}S_n^*F_{i}^{*})D_{S_n^*}\\&=D_{S_n}( E_iE_{n-j}^{*}- E_jE_{n-i}^*)D_{S_n}S_n^*+D_{S_n}(E_{n-j}^{*}D_{S_n}D_{S_n^*}F_{n-i} -E_{n-i}^{*}D_{S_n}D_{S_n^*}F_{n-j})D_{S_n^*}\\&+S_n^*D_{S_n^*}(F_{n-i}F_{j}^{*} -F_{n-j}F_{i}^{*})D_{S_n^*}\\&=(S_{n-j}^{*}S_i-S_{n-i}^{*}S_j)S_n^*-S_n^*(S_iS_{n-j}^{*}-S_jS_{n-i}^{*})+(S^{*}_{n-j}-S_n^*S_j)(S_{n-i}^{*}-S_iS_n^*)\\&-(S^{*}_{n-i}-S_n^*S_i)(S_{n-j}^{*}-S_jS_n^*)\\&=0,
\end{align*}}
which gives the desired identity.

\noindent{\bf{(3)}} Applying  Part $(3)$ of Lemma \ref{gamma_n contraction}, we have
\begin{align*}
S_iD_{S_n^*}F_{n-j}+D_{S_n^*}F_{n-i}F^*_j&=(D_{S_n^*}F_{i}^{*}+S_nD_{S_n^*}F_{n-i})F_{n-j}+D_{S_n^*}F_{n-i}F^*_j\\&=D_{S_n^*}(F_{i}^{*}F_{n-j}+F_{n-i}F^*_j)+S_nD_{S_n^*}F_{n-i}F_{n-j}\\&=D_{S_n^*}(F_{j}^{*}F_{n-i}+F_{n-j}F^*_i)+PD_{S_n^*}F_{n-i}F_{n-j}\\&=S_jD_{S_n^*}F_{n-i}+D_{S_n^*}F_{n-j}F^*_i
\end{align*}
By above given conditions, all other entries of $R_iR_j$ and $R_jR_i$ are equal. Hence, $R_iR_j=R_jR_i.$

\noindent{\bf{Step~2:}} We  now show that $R_iU=UR_i.$ Notice that

\scriptsize\begin{align*}
R_iU=\left[\begin{smallmatrix}
\begin{array}{cccc|c|cccc}
  \ddots &\vdots &\vdots  & \vdots & \vdots & \vdots & \vdots &\vdots  &\iddots \\
  \ldots   & 0 & E_i   & E_{n-i}^{*} & 0 &0& 0& 0&
  \ldots\\
  \ldots  & 0 & 0 & E_i & E_{n-i}^*D_{S_n} & -E^{*}_{n-i}S_n^* &0 &0    & \ldots\\ \ldots & 0 & 0  &0 & E_iD_{S_n}+E_{n-i}^{*}D_{S_n} S_n      & -E_iS_n^*+E_{n-i}^*D_{S_n}D_{S_n^*} &-E_{n-i}^{*}S_n^* &0 & \ldots\\
  \hline
  \ldots &0&  0 & 0 & S_iS_n & S_iD_{S_n^*} & D_{S_n^*}F_{n-i}  &0 & \ldots\\
  \hline
  \ldots &0 & 0 & 0 & 0& 0& F_{i}^{*} & F_{n-i} &  \ldots\\\ldots &0 & 0 & 0 & 0 & 0 & 0 &    F_{i}^{*}  &  \ldots\\
\iddots  & \vdots & \vdots & \vdots & \vdots & \vdots &\vdots & \vdots &\ddots
\end{array}
\end{smallmatrix}\right]
\end{align*}\normalsize
$$\textit{\rm and}$$
\small{\begin{align*}
UR_i=\left[\begin{smallmatrix}
\begin{array}{cccc|c|cccc}
  \ddots &\vdots &\vdots  & \vdots & \vdots & \vdots & \vdots &\vdots  &\iddots \\
  \ldots   & 0 & E_i   & E_{n-i}^{*} & 0 &0& 0& 0&
  \ldots\\
  \ldots  & 0 & 0 & E_i & E_{n-i}^*D_{S_n} & -E^{*}_{n-i}S_n^* &0 &0    & \ldots\\ \ldots & 0 & 0  &0 & D_{S_n} S_i       & D_{S_n}D_{S_n^*}F_{n-i}-S_n^*F_i^* &-S_n^*F_{n-i} &0 & \ldots\\
  \hline
  \ldots &0&  0 & 0 & S_nS_i & S_nD_{S_n^*}F_{n-i}+D_{S_n^*}F_i^* & D_{S_n^*}F_{n-i}  &0 & \ldots\\
  \hline
  \ldots &0 & 0 & 0 & 0& 0& F_{i}^{*} & F_{n-i} &  \ldots\\\ldots &0 & 0 & 0 & 0 & 0 & 0 &    F_{i}^{*}  &  \ldots\\
\iddots  & \vdots & \vdots & \vdots & \vdots & \vdots &\vdots & \vdots &\ddots
\end{array}
\end{smallmatrix}\right]
\end{align*}}
Using Lemma \ref{gamma_n contraction}, one can  verify that the entities of the positions $(-1,2),(-1,0)$ and $(0,1)$ of $R_iU$ and $UR_i$ are equal. To complete the above equality we need to verify
\small{\begin{equation}\label{R_iU,UR_i}
-E_iS_n^*+E_{n-i}^*D_{S_n}D_{S_n^*}=D_{S_n}D_{S_n^*}F_{n-i}-S_n^*F_i^*.\end{equation}}
Since the left hand side and right hand side of  \eqref{R_iU,UR_i} are defined on $\mathcal D_{S_n^*}$ to $\mathcal D_{S_n},$  we have
\small{\begin{align*}
-D_{S_n}E_iS_n^*D_{S_n^*}+D_{S_n}E_{n-i}^*D_{S_n}D^2_{S_n^*}&=-D_{S_n}E_iD_{S_n}S_n^*+D_{S_n}E_{n-i}^*D_{S_n}D^2_{S_n^*}\\&=-(S_i-S_{n-i}^*S_n)S_n^*+(S_{n-i}^*-S_n^*S_i)(I-PS_n^*)\\&=-S_n^*(S_i-S^{*}_{n-i})+(I-S_n^*S_n)(S_{n-i}^{*}-S_iS_n^*)\\&=D^{2}_{S_n}D_{S_n^*}F_{n-i}D_{S_n^*}-D_{S_n}S_n^*F^*_iD_{S_n^*},
\end{align*}}
which gives $R_iU=UR_i.$

\noindent{\bf{Step~3:}} We  now show that $R_i=R^{*}_{n-i}U.$ Note that

\scriptsize\begin{align*}
R^{*}_{n-i}U=\left[\begin{smallmatrix}
\begin{array}{cccc|c|cccc}
  \ddots &\vdots &\vdots  & \vdots & \vdots & \vdots & \vdots &\vdots  &\iddots \\
 \ldots   & 0 & E_i & E_{n-i}^* & 0& 0 &0 &0    & \ldots\\ \ldots & 0 & 0  &E_i & E_{n-i}^*D_{S_n}       & -E_{n-i}^*S_n^* &0 &0 & \ldots\\
  \hline
  \ldots &0&  0 & 0 & D_{S_n}E_iD_{S_n}+S^{*}_{n-i}S_n & -D_{S_n}E_iS_n^*+S^{*}_{n-i}D_{S_n^*} & 0  &0 & \ldots\\
  \hline
  \ldots &0 & 0 & 0 & -S_nE_iD_{S_n}+F_i^*D_{S_n^*}S_n& S_nE_iS_n^*+F_i^*D_{S_n^*}^{2}&  F_{n-i} & 0& \ldots\\\ldots &0 & 0 & 0 & 0 & 0  &    F_{i}^{*} &F_{n-i}  &  \ldots\\
\iddots  & \vdots & \vdots & \vdots & \vdots & \vdots &\vdots & \vdots &\ddots
\end{array}
\end{smallmatrix}\right]
\end{align*}\normalsize
To prove $R_i=R^{*}_{n-i}U,$ we need to verify the following operator identities:
\small{\begin{enumerate}
\item[(a)] $-D_{S_n}E_iS_n^*+S^{*}_{n-i}D_{S_n^*}=D_{S_n^*}F_{n-i},$

\item[(b)] $S_nE_iS_n^*+F_i^*D_{S_n^*}^{2}=F_i^*.$
\end{enumerate}}
\noindent{\bf{(a)}} As the LHS and RHS are defined on $\mathcal D_{S_n^*},$  we need only to show that $$-D_{S_n}E_iS_n^*D_{S_n^*}+S^{*}_{n-i}D_{S_n^*}^{2}=D_{S_n^*}F_{n-i}D_{S_n^*}.$$ Note that
\small{\begin{align*}
-D_{S_n}E_iS_n^*D_{S_n^*}+S^{*}_{n-i}D_{S_n^*}^{2}=-(S_i-S_{n-i}^*S_n)S_n^*+S_{n-i}^*(I-S_nS_n^*)=S_{n-i}^*-S_iS_n^*=D_{S_n^*}F_{n-i}D_{S_n^*}.
\end{align*}}

\noindent{\bf{(b)}} By Part $(1)$ of Lemma \ref{gamma_n contraction}, we get
\begin{align*}
S_nE_iS_n^*+F_i^*D_{S_n^*}^{2}&=F_i^*S_nS_n^*+F_i^*(I-S_nS_n^*)=F_i^*
\end{align*}
The other equality follow from Lemma \ref{gamma_n contraction} and  \eqref{Si}. Hence, we have $R_i=R^{*}_{n-i}U.$

\noindent{\bf{Step~4:}} First we will prove that $R_i$ is normal. Clearly, $R_{n-i}=R_i^*U.$ Since $R_iU=UR_i,$ using Fuglede's theorem \cite{Rudin}, we conclude that $R_i^*U=UR_i^*.$ Therefore, we have
\small{\begin{align*}
R_iR_i^*&=R_{n-i}^*UR_i^*=R_i^*R_{n-i}^*U=R_i^*R_i
\end{align*}}
%
Since $R_i$'s are normal  for $i=1,\ldots,(n-1)$ and $(R_1,\ldots,R_{n-1},U)$ is a $\Gamma_n$-contraction with $U$ is a unitary, by Theorem \ref{Gamma_n isometry}, we conclude that $(R_1,\ldots,U)$ is a $\Gamma_n$-unitary. This completes the proof.
\end{proof}

\begin{cor}\label{mathcal N}
Let $\mathcal N \subseteq \mathcal K$ be defined as $\mathcal N=\mathcal H\oplus l^2(\mathcal D_{S_n})$ and $(R_1,\ldots,R_{n-1},U)$ be a $\Gamma_n$-contraction. Then $\mathcal N$ is a common invariant subspace of $R_1,\ldots,R_{n-1},U$ and $(T_1,\ldots,T_{n-1},V)=(R_1\mid_{\mathcal N},\ldots,R_{n-1}\mid_{\mathcal N},U\mid_{\mathcal N})$ is a minimal $\Gamma_n$-isometric dilation of $(S_1,\ldots,S_n).$
\end{cor}
\begin{proof}
Clearly, $\mathcal N$ is a common  invariant subspace of $R_1,\ldots,R_{n-1},U.$ Therefore, by definition of $\Gamma_n$ -isometry, the restriction of $(R_1,\ldots,R_{n-1},U)$ to the common invariant subspace $\mathcal N,$ that is, $(T_1,\ldots,T_{n-1},V)$ is a $\Gamma_n$-isometry. The matrices of $T_1,\ldots,T_{n-1},V$ with respect to the decomposition $\mathcal H \oplus \mathcal D_{S_n}\oplus \ldots$ of $\mathcal N$ are as follows:

{\scriptsize $$T_i=\begin{bmatrix} S_i & 0 & 0 &\ldots\\F_{n-i}^*D_{S_n}& F_i & 0 &\ldots\\0 & F_{n-i}^* & F_i& \ldots \\\vdots &\vdots &\vdots  & \ddots
\end{bmatrix}~~{\rm{and}}~~ V=\begin{bmatrix} S_n & 0 & 0 &\ldots\\D_{S_n}&0 & 0 &\ldots\\0 & I & 0 & \ldots \\\vdots &\vdots &\vdots  & \ddots
\end{bmatrix} ~{\rm{for}}~ i=1,\ldots,n-1.$$\normalsize}

One can easily check that the adjoint of $(T_1,\ldots,T_{n-1},V)$ is a $\Gamma_n$-co-isometry extension of $(S_1^*,\ldots,S_n^*).$
Hence, by Proposition \ref{isometric dilation2}, we conclude that $(T_1,\ldots,T_{n-1},V)$ is a $\Gamma_n$-isometric dilation of $(S_1,\ldots,S_n).$ Since $\mathcal N$ and $V$ are the minimal isometric space and minimal isometric dilation of $S_n,$ the minimality of this $\Gamma_n$-isometric dilation follows from this fact. This completes the proof.
\end{proof}

\begin{rem}
From Theorem \ref{main dilation} and Corollary \ref{mathcal N}, we conclude that $(R_1,\ldots, R_{n-1},U)$ is a minimal $\Gamma_n$-unitary extension of $(T_1,\ldots,T_{n-1},V).$
\end{rem}

As an easy consequence the following theorem gives a sufficient condition for an $n$-tuple of operators $(S_1,\ldots,S_n)$ to become a complete spectral set for $\Gamma_n.$
\begin{thm}\label{main22}
Let $(S_1,\ldots,S_n)$ be a commuting $n$-tuple of operators on a Hilbert space $\mathcal H$ having $\Gamma_n$ is a spectral set. Let $E_1,\ldots,E_{n-1}$ be a commuting $(n-1)$-tuple of bounded operators on $\mathcal D_{S_n}$ satisfying the following conditions:
\small{\begin{enumerate}
\item $S_i-S_{n-i}^{*}S_n=D_{S_n}E_iD_{S_n}, S_{n-i}-S_{i}^{*}S_n=D_{S_n}E_{n-i}D_{S_n},$ for $i=1,\ldots,n-1;$

\item $[E_{i}^*,E_{j}^*]=0$ for $i,j=1,\ldots,n-1;$

\item $[E_i,E_{n-j}^*]=[E_j,E_{n-i}^*]$ for $i,j=1,\ldots,n-1.$
\end{enumerate}}
Then $\Gamma_n$ is a complete spectral set for $(S_1,\ldots,S_n).$ 
\end{thm}

\section{A functional model for a class of $\Gamma_n$-contractions}
In this section, we want to construct a concrete and explicit functional model for a
class of $\Gamma_n$-contractions  $(S_1,\ldots, S_{n})$ for
which the adjoint $(S_1^*,\ldots,S_n^*)$ admits an $(n-1)$-tuple of
fundamental operators $(F_1,\ldots,F_{n-1})$ satisfying the
following conditions:
\small{\begin{enumerate}
\item $[F_{n-i},F_{n-j}]=0$ for $i,j=1,\ldots,n-1;$

\item $[F_i,F_{n-j}^*]=[F_j,F_{n-i}^*]$ for $i,j=1,\ldots,n-1.$

\end{enumerate}}
The following proposition plays an important role for proving the
model theorem.

\begin{prop}\cite[Proposition 10.1]{pal3}\label{TV}
Suppose $T$ is a contraction and $V$ is the minimal isometric
dilation of $T.$ Then $T^*$ and $V^*$ have defect spaces of same
dimension.
\end{prop}

\begin{thm}
Suppose $(S_1,\ldots, S_{n})$ is a $\Gamma_n$-contraction on a
Hilbert space $\mathcal H$ such that the adjoint
$(S_1^*,\ldots,S_n^*)$ has  commuting $(n-1)$-tuples of
fundamental operators $(F_1,\ldots,F_{n-1})$ with
 $[F_i,F_{n-j}^*]=[F_j,F_{n-i}^*]$ for $i,j=1,\ldots,n-1.$ Let
 $(\hat{T}_1,\ldots,\hat{T}_{n-1},\hat{V})$ on $\mathcal
 N^*=\mathcal H\oplus \mathcal D_{S_n^*}\oplus \mathcal D_{S_n^*}\oplus
 \ldots$ be defined as
\scriptsize $$\hat{T}_i=\begin{bmatrix} S_i & D_{S_n^*}F_{n-i} & 0 &0 &\ldots\\0& F_i^* & F_{n-i} &0 &\ldots\\0 & 0 & F_i^*& F_{n-i} &\ldots \\\vdots &\vdots &\vdots &\vdots  & \ddots
\end{bmatrix}~~{\rm{and}}~~ \hat{V}=\begin{bmatrix} P & D_{S_n^*} &0 & 0 &\ldots\\0&0 & I &0 &\ldots\\0 & 0 & 0 & I& \ldots \\\vdots &\vdots &\vdots & \vdots & \ddots
\end{bmatrix} ~{\rm{for}}~ i=1,\ldots,n-1.$$\normalsize
 Then
\small{\begin{enumerate}
\item $(\hat{T}_1,\ldots,\hat{T}_{n-1},\hat{V})$ is a $\Gamma_n$-co-isometry, $\mathcal H$ is a
common invariant subspace of
$\hat{T}_1,\ldots,\hat{T}_{n-1},\hat{V}$ and
$\hat{T}_1\mid_{\mathcal H}=S_1,\ldots,\hat{T}_{n-1}\mid_{\mathcal
H}=S_{n-1},\hat{V}\mid_{\mathcal H}=S_n;$

\item there is a orthogonal decomposition $\mathcal N^*=\mathcal
N_1\oplus \mathcal N_2$ into reducing subspace of
$\hat{T}_1,\ldots,\hat{T}_{n-1},\hat{V}$ such that
$\hat{T}_1\mid_{\mathcal N_1},\ldots,\hat{T}_{n-1}\mid_{\mathcal
N_1},\hat{V}\mid_{\mathcal N_1}$ is a $\Gamma_n$ unitary and
$\hat{T}_1\mid_{\mathcal N_2},\ldots,\hat{T}_{n-1}\mid_{\mathcal
N_2},\hat{V}\mid_{\mathcal N_2}$ is a pure $\Gamma_n$-co-isometry;

\item $\mathcal N_2$ can be identified with $H^2(\mathcal
D_{\hat{V}}),$ where $\mathcal D_{\hat{V}}$ has the same dimension
as of $\mathcal D_{S_n}.$  The operators $\hat{T}_1\mid_{\mathcal
N_2},\ldots,\hat{T}_{n-1}\mid_{\mathcal N_2},\hat{V}\mid_{\mathcal
N_2}$ are unitarily equivalent to
$T_{A_1+A_{n-1}^*\bar{z}},\ldots,T_{A_{n-1}+A_{1}^*\bar{z}},T_{\bar{z}}$
on the vectorial Hardy space $H^2(\mathcal D_{\hat{V}}),$ where
$A_1,\ldots,A_{n-1}$ are the $(n-1)$-tuples of fundamental operators
of $(\hat{T}_1,\ldots,\hat{T}_{n-1},\hat{V}).$

\end{enumerate}}

\end{thm}
\begin{proof}
By Corollary \ref{mathcal N}, we conclude that
$(\hat{T}_1^*,\ldots,\hat{T}_{n-1}^*,\hat{V}^*)$ is minimal
$\Gamma_n$-isometric dilation of of $(S_1^*,\ldots,S_n^*),$
where $\hat{V}^*$ is the minimal isometric dilation of $S_n^*.$ By
Proposition \ref{isometric dilation2}, it follows immediately that
$(\hat{T}_1,\ldots,\hat{T}_{n-1},\hat{V})$ is the
$\Gamma_n$-co-isometric extension  of $(S_1,\ldots,S_n).$
Therefore, we conclude that $\mathcal H$ is a common invariant
subspace of $\hat{T}_1,\ldots,\hat{T}_{n-1},\hat{V}$ and
$\hat{T}_1\mid_{\mathcal H}=S_1,\ldots,\hat{T}_{n-1}\mid_{\mathcal
H}=S_{n-1},\hat{V}\mid_{\mathcal H}=S_n.$ Since
$(\hat{T}_1^*,\ldots,\hat{T}_{n-1}^*,\hat{V}^*)$ is a $\Gamma_n$-isometry, by Wold decomposition theorem, there is a orthogonal
decomposition $\mathcal N^*=\mathcal N_1\oplus \mathcal N_2$ into
reducing subspace of $\hat{T}_1,\ldots,\hat{T}_{n-1},\hat{V}$ such
that $\hat{T}_1\mid_{\mathcal
N_1},\ldots,\hat{T}_{n-1}\mid_{\mathcal N_1},\hat{V}\mid_{\mathcal
N_1}$ is a $\Gamma_n$-unitary and $\hat{T}_1\mid_{\mathcal
N_2},\ldots,\hat{T}_{n-1}\mid_{\mathcal N_2},\hat{V}\mid_{\mathcal
N_2}$ is a pure $\Gamma_n$-co-isometry. Therefore, we can write
$(\hat{T}_1,\ldots,\hat{T}_{n-1},\hat{V})$ with respect to the
orthogonal decomposition as $$\hat{T}_i=\begin{bmatrix} T_{1i} &
0\\0 & T_{2i}\end{bmatrix} ~~{\rm{and}}~~ \hat{V}=\begin{bmatrix}
V_{1} & 0\\0 & V_{2}\end{bmatrix}$$ for $i=1,\ldots,n-1.$ Since
$\mathcal D_{\hat{V}}=\mathcal D_{V_2},$ one can easily show that
$(\hat{T}_1,\ldots,\hat{T}_{n-1},\hat{V})$ and
$(T_{21},\ldots,T_{2(n-1)},V_2)$ have the same $(n-1)$-tuples of
fundamental operators. Also, by applying Proposition \ref{pure isometry}, we have the
following:
\small{\begin{enumerate}
\item $\mathcal N_2$ can be identified with $H^2(\mathcal
D_{\hat{V}})$

\item  The operators $\hat{T}_1\mid_{\mathcal
N_2},\ldots,\hat{T}_{n-1}\mid_{\mathcal N_2},\hat{V}\mid_{\mathcal
N_2}$ are unitarily equivalent to
$T_{A_1+A_{n-1}^*\bar{z}},\ldots,T_{A_{n-1}+A_{1}^*\bar{z}},T_{\bar{z}}$
on the vectorial Hardy space $H^2(\mathcal D_{\hat{V}}),$ where
$A_1,\ldots,A_{n-1}$ are the $(n-1)$-tuples of fundamental operators
of $(\hat{T}_1,\ldots,\hat{T}_{n-1},\hat{V}).$
\end{enumerate}}
Since $\hat{V}^*$ is the minimal isometric dilation of $S_n^*,$ using
Proposition \ref{TV} one can easily verify that $\mathcal
D_{\hat{V}}$ has the same dimension as of $\mathcal D_{S_n}.$  This
completes the proof.
\end{proof}

\vskip.3cm

\textsl{Acknowledgements:}
The author gratefully acknowledge the help that was received from Dr. Sourav Pal, Dr. Shibananda Biswas, Dr. Subrata Shyam Roy, Dr. Sayan Bagchi and Prof. Gadadhar Misra.
\vskip-1cm



\begin{thebibliography}{99}
\vskip-.4cm
\bibitem{A}   W. Arveson,
              \textit{Subalgebras of $C^*$-algebras, }
               Acta Math., {\bf 123} (1969), 141 -224.



\bibitem{AW}   W. Arveson,
              \textit{Subalgebras of $C^*$-algebras II,}
               Acta Math., {\bf 128} (1972), 271 -308.




\bibitem{AHR} J. Agler, J. Harland, B. J. Raphael,
             \textit{Classical function theory, operator dilation theory, and machine computation on multiply-connected          domains,}
             Mem. Amer. Math. Soc. {\bf 191} (2008), 289 -- 312.

\bibitem{ALY} J. Agler, Z. A. Lykova, N. J. Young,
             \textit{Extremal holomorphic maps and the symmetrized bidisc,}
             Proc. London Math. Soc., {\bf 106 }(2013) 781--818.



\bibitem{AM} J. Agler, J. McCarthy,
             \textit{Pick Interpolation and Hilbert Function Spaces,}
             Graduate studies in mathematics {\bf 44,} Amer. Math. Soc., Providence, R.I. (2002).




\bibitem{young} J. Agler,  N. J. Young,
               \textit{Operators having the symmetrized bidisc as a spectral set,}
               Proc. Edinburgh Math. Soc., {\bf 43} (2000), 195 -210.




\bibitem{AY} J. Agler, N.J. Young,
             \textit{A commutant lifting theorem for a domain in $\mathbb C^2$ and spectral interpolation,}
             J. Funct. Anal. {\bf 161} (1999) 452--477.






\bibitem{Agler} J. Agler, N.J. Young,
               \textit{A model theory for $\Gamma$-contractions,}
                J. Operator Theory {\bf 49} (2003) 45--60.

\bibitem{JAgler} J. Agler, N.J. Young,
                \textit{The hyperbolic geometry of the symmetrized bidisc,}
                J. Geom. Anal. {\bf 14} (2004) 375--403.



%


\bibitem{Ando}  T. Ando,
               \textit{Structure of operators with numerical radius one,}
               Acta Sci. Math. (Szeged){\bf 34} (1973) 11--15.



\bibitem{BFK}H. Bercovici, C. Foias, L. Kerchy, B. Sz.-Nagy,
             \textit{Harmonic analysis of operators on Hilbert space,}
             Universitext, Springer, New York, (2010).

\bibitem{BPR} T. Bhattacharyya, S. Pal, S. Shyam Roy,
            \textit{Dilations of $\Gamma$-contractions by solving operator equations,}
             Adv. Math.{\bf 230} (2012), 577 -- 606.




\bibitem{BSS} T. Bhattacharyya, Sneh Lata, H. Sau,
              \textit{Admissible fundamental operators,}
              J . Math. Anal. Appl., {\bf 425} (2015), 983 --1003.


\bibitem{SS} S. Biswas, S. Shyam Roy,
             \textit{Functional models for $\Gamma_n$-contractions and characterization of $\Gamma_n$-isometries,}
             J. Func. Anal., {\bf 266} (2014), 6224 --6255.














\bibitem{costara} C. Costara,
                 \textit{The symmetrized bidisc and Lempert?s theorem,}
                  Bull. London Math. Soc. {\bf 36} (2004) 656--662.

\bibitem{Ccostara}C. Costara,
                 \textit{On the spectral Nevanlinna-Pick problem,}
                 Studia Math., {\bf  170} (2005), 23--55.

\bibitem{curto} R. E. Curto,
               \textit{Applications of several complex variables to multiparameter spectral theory,}
               Surveys of Some Recent Results in Operator Theory, Vol. II, Pitman Res.
               Notes Math. Ser., Longman Sci. Tech., Harlow, {\bf 192} (1988), 25-90.





\bibitem{dmp}   R.~G. Douglas, P. S. Muhly, C. Pearcy
            \textit{Lifting commuting Operators},
              Michigan Math, J.,15  (1968),  385--395.

\bibitem{DriRon}
 M. A. Dritschel, J. Rovnyak,
\textit{The Operator Fejér-Riesz Theorem,}
Chapter in A glimpse at Hilbert space operators: Paul R. Halmos in Memorium, Oper. Theory Adv. Appl., 207(2010), Birkhuser Verlag, Basel, 223-254.






\bibitem{Zwonek} A. Edigarian, W. Zwonek, 
                \textit{Geometry of symmetrized polydisc,}
                 Archiv der Mathematik, {\bf 84} (2005), 364 -374.

\bibitem{Ful}  B. Fuglede, 
               \textit{A commutativity theorem for normal operators,}
                Nat. Acad. Sci. {\bf 36} (1950), 35 - 40.

\bibitem{Gamelin} T. Gamelin, 
                 \textit{Uniform Algebras,}
                  Prentice Hall, New Jersey, (1969).

\bibitem{Gorai} S. Gorai, J. Sarkar, 
               \textit{Characterizations of symmetrized polydisc,}
                To appear in Indian Jour. Pure and Applied Math., 
                available at arXiv:1503.03473v1 [math.CV] 11 Mar (2015).



\bibitem{Horn}  R. A. Horn, C. R. Johnson, 
               \textit{Matrix Analysis,}
                Cambridge University Press, 2nd Edition, (2013).


\bibitem{Jarnicki} M. Jarnicki, P. Pflug, 
                  \textit{On automorphisms of the symmetrized bidisc,}
                   Arch. Math.(Basel) {\bf 83} (2004), 264 -266.


\bibitem{Wzonek} L. Kosi´nski,W. Zwonek, 
                \textit{Extremal holomorphic maps in special classes of domains,}
                 Annals of the Scuola Normale Superiore di Pisa - Science Class, 
                 doi: 10.2422/2036- 2145.201401 003.

\bibitem{Lempert} L. Lempert, 
                  \textit{Complex geometry in convex domains,}
                   Proc. Intern. Cong. Math., Berkeley,CA (1986) 759-765.













%




%
%




\bibitem{TW}   T. W. Palmer,
               \textit{Banach algebras and the general theory of $*$-algebra,}
               Cambridge Univ. Press(1994).

\bibitem{pal3} S. Pal
               \textit{The failure of rational dilation on the symmetrized $n$-disk for any $n\geq 3.$, }
                arXiv:1712.05707 


\bibitem{pal4} S. Pal
               \textit{Canonical decomposition of operators associated with the symmetrized polydisc }
               arXiv:1708.00724


\bibitem{pal5} S. Pal
              \textit{Operator theory and distinguished varities in the symmetrized $n$-disk }
               arXiv:1708.06015





\bibitem{paulsen} V. Paulsen,
                 \textit{Completely Bounded Maps and Operator Algebras,}
                 Cambridge  Univ. Press, (2002).






\bibitem{patak} V. Ptak, N. J. Young, 
               \textit{A generalization of zero location theorem of Schur and Cohn,} 
               Trans. Inst. Electrical Electron. Automat. Control, {\bf 25} (1980), 978 -980.

\bibitem{Rose} M. Rosenblum, 
               \textit{Vectorial Toeplitz operator and the Fejer-Riesz theorem,}
                J. Math. Anal. Appl., {\bf 23 } (1968), 139-147.


\bibitem{Rudin} W.Rudin, 
               \textit{Proper holomorphic maps and finite reflection groups,}
                Indiana Univ. Math. J.,{\bf 31 } (1982) 701 -720.

%

\bibitem{sarkar} J. Sarkar, 
                \textit{Operator theory on symmetrized bidisc,} 
                Indiana Univ. Math. J.,{\bf 64} (2015), 847 -873.

\bibitem{Schur} I. Schur, 
               \textit{\"Uber Potenzreihen die im Innern des Einheitskreises beschr\"ankt sind,} 
               Jour. \"fur Math., I {\bf 147 } (1917), 205 --232 ; II{\bf 148} (1918), 122 -- 145.

\bibitem{bela} Bela Sz.-Nagy, 
             \textit{Sur les contractions de \,lespace de Hilbert,} 
             Acta Sci. Math., {\bf 15} (1953), 87 -92.





\bibitem{stout} E. L. Stout,
               \textit{Polynomial Convexity,}
               Progr. Math., Birkh\"auser Boston, Inc., Boston, MA, (2007).



\bibitem{Taylor}  J. L. Taylor, 
                \textit{The analytic-functional calculus for several commuting operators,} 
                Acta Math. {\bf 125} (1970), 1 --38.
                
\bibitem{jtaylor} J. L. Taylor, 
                  \textit{A joint spectrum for several commuting operators,}
                   J. Funct. Anal., {\bf 6} (1970), 172 -- 191.

\bibitem{Vasilescu} F. H. Vasilescu, 
                    \textit{Analytic Functional Calculus and Spectral Decompositions,}
                     Editura Academiei: Bucuresti, Romania and D. Reidel Publishing Company, (1982).

%












\end{thebibliography}
\end{document}